\documentclass[reqno, final,a4paper]{amsart}
\usepackage{color}
\usepackage[colorlinks=true,allcolors=blue
]{hyperref}
\usepackage{amsmath, amssymb, amsthm}
\usepackage{aliascnt}
\usepackage{mathrsfs}
\usepackage{mathtools}
\usepackage[noabbrev,capitalize,nameinlink]{cleveref}
\crefname{equation}{}{}
\usepackage{fullpage}
\usepackage[noadjust,nocompress]{cite}
\usepackage{graphics}
\usepackage{pifont}
\usepackage{tikz}
\usepackage{bbm}
\usepackage[T1]{fontenc}
\usetikzlibrary{arrows.meta}

\usepackage{environ}
\usepackage{framed}
\usepackage{url}
\usepackage[linesnumbered,ruled,vlined]{algorithm2e}
\usepackage[noend]{algpseudocode}
\usepackage[labelfont=bf]{caption}
\usepackage{framed}
\usepackage[framemethod=tikz]{mdframed}
\usepackage{pgfplots}
\usepackage{appendix}
\usepackage{graphicx}
\usepackage[textsize=tiny]{todonotes}
\usepackage{tcolorbox}
\usepackage{xcolor}
\usepackage{flafter}
\usepackage[shortlabels]{enumitem}
\crefformat{enumi}{#2#1#3}
\crefrangeformat{enumi}{#3#1#4 to~#5#2#6}
\crefmultiformat{enumi}{#2#1#3}
{ and~#2#1#3}{, #2#1#3}{ and~#2#1#3}
\allowdisplaybreaks

\usepackage{nicematrix}   
\usetikzlibrary{calc} 
\usepackage{tikz}
\usetikzlibrary{fit}

\let\originalleft\left
\let\originalright\right
\renewcommand{\left}{\mathopen{}\mathclose\bgroup\originalleft}
\renewcommand{\right}{\aftergroup\egroup\originalright}

\newcommand*{\claimproofname}{Proof of claim}
\newenvironment{claimproof}[1][\claimproofname]{\begin{proof}[#1]}{\end{proof}}

\crefname{algocf}{Algorithm}{Algorithms}

\crefname{equation}{}{}
 
\AtBeginEnvironment{appendices}{\crefalias{section}{appendix}} 

\usepackage[color]{showkeys}

\colorlet{refkey}{orange!20}
\colorlet{labelkey}{blue!30}

\crefname{algocf}{Algorithm}{Algorithms}

\numberwithin{equation}{section}

\newtheorem{theorem}{Theorem}[section]

\newaliascnt{proposition}{theorem}

\aliascntresetthe{proposition}

\newaliascnt{lemma}{theorem}
\newtheorem{lemma}[lemma]{Lemma}
\aliascntresetthe{lemma}

\newaliascnt{claim}{theorem}
\newtheorem{claim}[claim]{Claim}
\aliascntresetthe{claim}

\newaliascnt{corollary}{theorem}
\newtheorem{corollary}[corollary]{Corollary}
\aliascntresetthe{corollary}

\newaliascnt{conjecture}{theorem}
\newtheorem{conjecture}[conjecture]{Conjecture}
\aliascntresetthe{conjecture}

\newtheorem*{claim*}{Claim}

\theoremstyle{definition}

\newaliascnt{fact}{theorem}
\newtheorem{fact}[fact]{Fact}
\aliascntresetthe{fact}

\newaliascnt{definition}{theorem}
\newtheorem{definition}[definition]{Definition}
\aliascntresetthe{definition}

\newaliascnt{problem}{theorem}

\aliascntresetthe{problem}

\newaliascnt{question}{theorem}

\aliascntresetthe{question}

\newtheorem*{definition*}{Definition}

\newaliascnt{example}{theorem}

\aliascntresetthe{example}

\newaliascnt{setup}{theorem}

\aliascntresetthe{setup}

\theoremstyle{remark}

\newaliascnt{remark}{theorem}
\newtheorem{remark}[remark]{Remark}
\aliascntresetthe{remark}

\newcommand{\mb}{\mathbb}
\newcommand{\mbf}{\mathbf}

\newcommand{\mc}{\mathcal}

\newcommand{\mr}{\mathrm}

\newcommand{\on}{\operatorname}

\newcommand{\pvec}[1]{\vec{#1}\mkern2mu\vphantom{#1}'}

\newcommand{\cH}{\mathcal{H}}
\newcommand{\cI}{\mathcal{I}}

\newcommand{\cL}{\mathcal{L}}

\newcommand{\cP}{\mathcal{P}}
\newcommand{\cQ}{\mathcal{Q}}

\newcommand{\cS}{\mathcal{S}}
\newcommand{\cT}{\mathcal{T}}
\newcommand{\cU}{\mathcal{U}}

\newcommand{\sC}{\mathscr{C}}

\newcommand{\sR}{\mathscr{R}}
\newcommand{\sS}{\mathscr{S}}

\newcommand{\bfB}{\mathbf{B}}

\newcommand{\bfG}{\mathbf{G}}
\newcommand{\bfH}{\mathbf{H}}

\newcommand{\bfL}{\mathbf{L}}

\newcommand{\bfP}{\mathbf{P}}

\newcommand{\bfR}{\mathbf{R}}
\newcommand{\bfS}{\mathbf{S}}
\newcommand{\bfT}{\mathbf{T}}

\newcommand{\bfX}{\mathbf{X}}
\newcommand{\bfY}{\mathbf{Y}}
\newcommand{\bfZ}{\mathbf{Z}}

\newcommand{\bfg}{\mathbf{g}}
\newcommand{\bfh}{\mathbf{h}}

\newcommand{\bfz}{\mathbf{z}}
\newcommand{\bflambda}{\boldsymbol{\lambda}}
\newcommand{\bfmu}{\boldsymbol{\mu}}
\newcommand{\bfnu}{\boldsymbol{\nu}}

\DeclareMathOperator{\Li}{Li}
\DeclareMathOperator{\supp}{supp}

\newcommand{\coo}{\color{orange}}

\newcommand{\trp}{\operatorname{TRP}}

\renewcommand{\Pr}{\mb P}

\allowdisplaybreaks

\title{Parities in random Latin squares}

\author[Kwan]{Matthew Kwan}

\address{Institute of Science and Technology Austria (ISTA). Am Campus 1, 3400 Klosterneuburg, Austria}
\email{matthew.kwan@ist.ac.at}

\author[Petrova]{Kalina Petrova}
\address{Institute of Science and Technology Austria (ISTA). Am Campus 1, 3400 Klosterneuburg, Austria}
\email{kalina.petrova@ist.ac.at}

\author[Sawhney]{Mehtaab Sawhney}
\address{Department of Mathematics, Columbia University, New York, NY 10027}
\email{m.sawhney@columbia.edu}
\thanks{
MK was supported by ERC Starting Grant ``RANDSTRUCT'' No.~101076777. KP was supported by the European Union’s Horizon 2020 research and innovation programme
under the Marie Skłodowska-Curie grant agreement No.~101034413 \includegraphics[width=4mm]{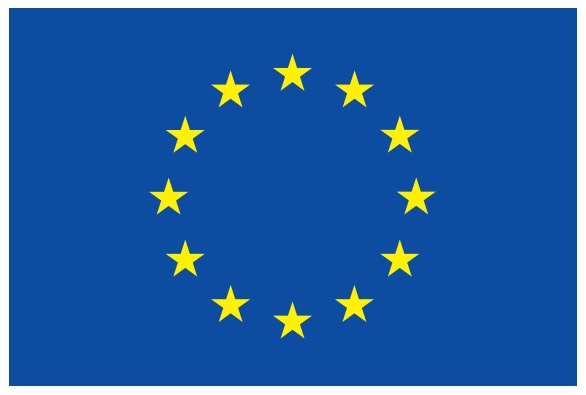}. This research was conducted during the period MS served as a Clay Research Fellow.
}
\begin{document}
\maketitle
\begin{abstract}
In a Latin square, every row can be interpreted as a permutation, and therefore has a parity (even or odd). We prove that in a uniformly random $n\times n$ Latin square, the $n$ row parities are very well approximated by a sequence of $n$ independent unbiased coin flips: for example, the total variation error of this approximation tends to zero as $n\to\infty$.
This resolves a conjecture of Cameron.
In fact, we prove a generalisation of Cameron's conjecture for the
\emph{joint} distribution of the row parities, column parities and symbol parities (the latter are defined
by the symmetry between rows, columns and symbols of a Latin square).

Along the way, we introduce several general techniques for the study
of random Latin squares, including a new re-randomisation technique via ``stable intercalate switchings'', and a new approximation theorem comparing random Latin squares with a certain independent model.
\end{abstract}

\section{Introduction}

An \emph{$n\times n$ Latin square} is an $n\times n$ array filled
with $n$ different ``symbols'' (usually taken to be the integers
$1,\dots,n$), with the property that each symbol appears exactly
once in each row and each column. For example, the multiplication
table of a group is always a Latin square; in general, Latin squares
can be interpreted as multiplication tables of a class of algebraic
structures called \emph{quasigroups}. See for example \cite{KD15}
for an introduction to this vast subject.

Each row or column of an $n\times n$ Latin square $L$ can be interpreted
as a permutation of order $n$, which can be either even or odd. Let
$N_{\mathrm{row}}(L)$ be the number of odd row permutations, and
let $N_{\mathrm{col}}(L)$ be the number of odd column permutations.
If $L$ is the multiplication table of a group, then either $N_{\mathrm{row}}(L)=N_{\mathrm{col}}(L)=0$
or $N_{\mathrm{row}}(L)=N_{\mathrm{col}}(L)=n/2$. However, for general
Latin squares, the row and column parities can have much richer behaviour,
and much is still unknown. For example, one of the most important
conjectures in this direction is the \emph{Alon--Tarsi conjecture},
which (in probabilistic language) says that if $n$ is even, and $\mathbf{L}$
is a uniformly random $n\times n$ Latin square, then
\[
\Pr\big[N_{\mathrm{row}}(\mathbf{L})\text{ is even}\big]=\Pr\big[N_{\mathrm{row}}(\mathbf{L})+N_{\mathrm{col}}(\mathbf{L})\text{ is even}\big]\ne\frac{1}{2}.
\]
(The first equality is not part of the original conjecture; it was
observed by Huang and Rota~\cite{HR94}, in a paper where they also observed that
the Alon--Tarsi conjecture has a number of surprising consequences
in seemingly unrelated areas of mathematics; see \cite{FM19} for
a modern survey). It was first proved by Alp\"oge~\cite{Alp17} that
\begin{equation}
\Pr\big[N_{\mathrm{row}}(\mathbf{L})\text{ is even}\big]=\frac{1}{2}+o(1).\label{eq:alpoge}
\end{equation}
as $n\to\infty$. In other words, if the Alon--Tarsi conjecture is
true, then it is true ``just barely''.

Going far beyond \cref{eq:alpoge}, it has been suggested by Peter
Cameron (in a variety of different sources; see for example \cite{Cam01,Cam02,Cam05,Cam07,Cam15,Cam13,Cam03})
that the row parities of a Latin square might be statistically completely
unconstrained, in the sense that one can model the $n$ row parities
of a random Latin square by simply making $n$ independent coin flips. 
\begin{conjecture}
\label{conj:cameron}Let $\mathbf{L}$ be a uniformly random $n\times n$
Latin square. Then the distribution of $N_{\mathrm{row}}(\mathbf{L})$
is approximately the binomial distribution $\operatorname{Bin}(n,1/2)$, as $n\to\infty$.
\end{conjecture}

(Note that exchanging rows does not affect the distribution of $\mbf L$, so the sequence of row parities $\vec \xi_{\mr{row}}(\mbf L)\in (\mb Z/2\mb Z)^n$ has a permutation-invariant distribution. This means that if we condition on $N_{\mathrm{row}}(\mathbf{L})$, then $\vec \xi_{\mr{row}}(\mbf L)\in (\mb Z/2\mb Z)^n$ is a uniformly random sequence in $(\mb Z/2\mb Z)^n$, constrained to have exactly $N_{\mathrm{row}}(\mathbf{L})$ ``1''s. That is to say, in order to understand the distribution of $\vec \xi_{\mr{row}}(\mbf L)$, it is enough to understand the distribution of  $N_{\mathrm{row}}(\mathbf{L})$.)

To elaborate on the attribution/history of \cref{conj:cameron}: the starting point seems to have been the problem session at the
\emph{British Combinatorial Conference} in 1993, where Cameron asked
a related question that was extended by Jeannette Janssen. However, at that time they did not seem to be very confident that the statement of \cref{conj:cameron} was actually true (in the BCC problem list~\cite{Cam01} they phrase the question as ``is it true that...''). Cameron later posed the problem more assertively in a 2002 survey on permutations and permutation groups~\cite{Cam02}, and it seems he first referred to it as a ``conjecture'' in a 2003 lecture on random Latin squares~\cite{Cam03}. However, as far as we can tell, he never stated the problem in a fully precise form
(always using language like ``approximately'').

There are many different ways to compare distributions, to make rigorous
sense of the word ``approximately''; in this
paper we confirm \cref{conj:cameron} in many different senses. For example, one strong way to compare distributions
is in terms of \emph{total variation distance}: for two probability
distributions $\mu,\nu$ on the same space, write $\operatorname{d}_{\mathrm{TV}}(\mu,\nu)=\sup_{A}|\mu(A)-\nu(A)|$
(so for any event $A$, the probabilities of $A$ with
respect to $\mu$ and $\nu$ differ by at most $\operatorname{d}_{\mathrm{TV}}(\mu,\nu)$).
The following theorem is a consequence of our main (technical) result.
\begin{theorem}
\label{thm:baby}Let $\mathbf{L}$ be a uniformly random $n\times n$
Latin square. Then 
\[
\lim_{n\to\infty}\operatorname{d}_{\mathrm{TV}}\!\big(N_{\mathrm{row}}(\mathbf{L}),\,\operatorname{Bin}(n,1/2)\big)=0.
\]
Equivalently, writing $\vec \xi_{\mr{row}}(\mbf L)\in (\mb Z/2\mb Z)^n$ for the sequence of parities of rows of $\mbf L$, and writing $\on{Unif}((\mb Z/2\mb Z)^n)$ for the uniform distribution on $(\mb Z/2\mb Z)^n$, we have
\[
\lim_{n\to\infty}\operatorname{d}_{\mathrm{TV}}\!\big(\vec\xi_{\mr{row}}(\mathbf{L}),\,\on{Unif}((\mb Z/2\mb Z)^n)\big)=0.
\]
\end{theorem}

We can actually prove much more than \cref{thm:baby}; we will discuss this momentarily, but first it is worthwhile to briefly discuss why this conjecture took so long to be resolved (and, in our opinion, why it is so interesting). Generally speaking, it is quite easy to make plausible predictions about uniformly random Latin squares, by making various kinds of approximate independence assumptions (for example, Cameron's conjecture can be justified from the point of view that there is ``no obvious reason'' for the parities of different rows to be correlated). However, it is surprisingly difficult to rigorously prove anything nontrivial about uniformly random Latin squares, or even to study them empirically.

The main issues are that Latin squares do not enjoy any neat recursive structure, and they are very ``rigid'' objects, in the sense that there are only very limited ways to make a ``local perturbation'' to change a Latin square into another one. To highlight the difficulties here, we remark that (despite some very ambitious conjectures; see \cite[Section~4.1]{ADRA23}) we still have a rather poor understanding of the total number of $n\times n$ Latin squares (the best known upper and lower bounds differ by exponential factors; see \cite[Section~17]{vLW01}), and there is no rigorously justified way to efficiently sample a uniformly random Latin square (there are certain ergodic Markov chains on the space of $n\times n$ Latin squares~\cite{Pit97,JM96}, but these are not known to be rapidly mixing). By now, there are quite a few known results about random Latin squares; the proofs of many of these results have required fundamental new additions to a very limited toolbox of techniques, and \cref{thm:baby} is no exception (we discuss existing and new techniques in \cref{sec:outline}).

\subsection{Row, column and symbol permutations}
Our proof techniques actually allow us to go beyond Cameron's conjecture, to approximate not just the distribution of $N_{\mathrm{row}}(\mathbf{L})$ but also
the \emph{joint} distribution between $N_{\mathrm{row}}(\mathbf{L})$,
$N_{\mathrm{col}}(\mathbf{L})$ and a third parameter $N_{\mathrm{sym}}(\mathbf{L})$,
which counts the number of odd \emph{symbol permutations}. Perhaps
the most natural way to define a symbol permutation is to reinterpret
a Latin square in a way that emphasises the natural symmetry between
rows, columns and symbols: indeed, observe that an $n\times n$ Latin
square (with symbols $1,\dots,n$) can be interpreted as an $n\times n\times n$
array in which every entry is ``0'' or ``1'', such that every
two-dimensional ``slice'' of this array is a permutation matrix\footnote{The correspondence is that if the Latin square has symbol $s$ in
the $(i,j)$-entry, then we put a ``1'' in the $(i,j,s)$-entry
of the corresponding $n\times n\times n$ array.}. So, the row and column permutations correspond to slices in two
of the three possible directions, and the symbol permutations correspond
to slices in the third direction.

The row, column and symbol parities cannot vary completely freely:
it was proved by Janssen~\cite{Jan95} and Zappa~\cite{Zap96} that
for any Latin square $L$ of order $n$, we have
\[
N_{\mathrm{row}}(L)+N_{\mathrm{col}}(L)+N_{\mathrm{sym}}(L)=f(n)\pmod2,
\]
where 
\[
f(n)=\begin{cases}
0 & \text{if }n=0\text{ or 1}\pmod4\\
1 & \text{if }n=2\text{ or 3}\pmod4.
\end{cases}
\]
Subject to this constraint, we are able to prove the natural
extension of Cameron's conjecture to row, column and symbol parities.
Specifically, let $\mu^{*}$ be the conditional distribution of three
independent $\operatorname{Bin}(n,1/2)$ random variables, given that
their sum is $f(n)$ mod 2. We are able to prove that $(N_{\mathrm{row}}(\mathbf{L}),N_{\mathrm{col}}(\mathbf{L}),N_{\mathrm{sym}}(\mathbf{L}))$
approximately has the distribution $\mu^{*}$ as $n\to\infty$, in
many different senses.
\begin{theorem}
\label{thm:shiny}Let $\mathbf{L}$ be a uniformly random $n\times n$
Latin square, and let 
\[
\vec{\mathbf{X}}=(\mathbf{X}_{1},\mathbf{X}_{2},\mathbf{X}_{3})=(N_{\mathrm{row}}(\mathbf{L}),N_{\mathrm{col}}(\mathbf{L}),N_{\mathrm{sym}}(\mathbf{L})).
\]
Then, the following hold as $n\to\infty$.
\begin{enumerate}
\item (Law of large numbers) We have the convergence in probability
\[
\frac1n(\mathbf{X}_{1},\,\mathbf{X}_{2},\,\mathbf{X}_{3})\overset{p}{\to}(1/2,1/2,1/2),
\]
i.e., $\vec{\mathbf{X}}$ satisfies the same law of large numbers
as $\mu^{*}$.
\item \smallskip(Central limit theorem) We have the convergence in distribution
\[
\frac{1}{\sqrt{n/4}}(\mathbf{X}_{1}-n/2,\,\mathbf{X}_{2}-n/2,\,\mathbf{X}_{3}-n/2)\overset{d}{\to}\mathcal{N}(0,I_{3}),
\]
i.e., $\vec{\mathbf{X}}$ satisfies the same trivariate central limit
theorem as $\mu^{*}$.
\item \smallskip(Local central limit theorem) For all $\vec{x}=(x_{1},x_{2},x_{3})$
satisfying $x_{1}+x_{2}+x_{3}=f(n)$ (mod 2), we have
\begin{align*}
 & \Pr[\vec{\mathbf{X}}=\vec{x}]=2\cdot\frac{1}{(2\pi(n/4))^{3/2}}\exp\left(-\frac{(x_{1}-n/2)^{2}+(x_{2}-n/2)^{2}+(x_{3}-n/2)^{2}}{2(n/4)}\right)+o(n^{-3/2}),
\end{align*}
i.e., $\vec{\mathbf{X}}$ satisfies the same \emph{local} central
limit theorem as $\mu^{*}$. 
\item \smallskip(Total variation convergence) We have
\[
\operatorname d_{\mathrm{TV}}(\vec{\mathbf{X}},\mu^{*})\to0.
\]
\item \smallskip(Large deviation principle) Let $H_2:\alpha\mapsto-\alpha\log_2\alpha-(1-\alpha)\log_2(1-\alpha)$ 
be the base-2 binary entropy function, and let $I(x_{1},x_{2},x_{3})=3-H_2(x_{1})-H_2(x_{2})-H_2(x_{3})$.
We have 
\[
-\inf_{\vec{x}\in E^{\circ}}I(\vec{x})\le\liminf n^{-1}\log_2\Pr[n^{-1}\vec{\mathbf{X}}\in E]\le\limsup n^{-1}\log_2\Pr[n^{-1}\vec{\mathbf{X}}\in E]\le-\inf_{\vec{x}\in\overline{E}}I(\vec{x})
\]
for all Borel $E\subseteq\mb R^{3}$, where $\overline{E}$ and $E^{\circ}$
denote the closure and interior of $E$. That is to say, $\vec{\mathbf{X}}$
satisfies the same large deviation principle as $\mu^{*}$.
\end{enumerate}
\end{theorem}
Actually, all parts of \cref{thm:shiny} are really consequences of a single master theorem (\cref{thm:main-technical}), but we feel the statement of that master theorem is a little too technical for this introduction.

\subsection{Previous work}
We end this introduction with a brief discussion of previous work
on Cameron's conjecture. The first theorem in this area was due to H\"aggkvist and Janssen~\cite{HJ96}, who
proved that
\begin{equation}\Pr[N_{\mathrm{row}}(\mathbf{L})=0]\le(\sqrt{3/4}+o(1))^{n}\approx 0.87^{n}\label{eq:HJ}\end{equation}
(note that \cref{thm:shiny}(5) implies that\footnote{Curiously, the methods in this paper do not directly imply the corresponding \emph{lower} bound $\Pr[N_{\mathrm{row}}(\mathbf{L})=0]\ge(1/2+o(1))^{n}$. However, this does seem to be provable with some additional tricks; we intend to explore this in future work with Catherine Greenhill and Lenka Kopfov\'a.} $\Pr[N_{\mathrm{row}}(\mathbf{L})=0]\le(1/2+o(1))^{n}$). Second, generalising \cref{eq:alpoge}, Cavenagh and Wanless~\cite{CW16} proved that $(N_{\mathrm{row}}(\mathbf{L}),N_{\mathrm{col}}(\mathbf{L}))$
is asymptotically equidistributed modulo 2 (this is a special case of
\cref{thm:shiny}(4)). Third, Cavenagh, Greenhill and Wanless~\cite{CGW08} proved that the probability that the first two rows of $\mbf L$ have the same parity lies between $1/4-o(1)$ and $3/4+o(1)$ (note that \cref{thm:shiny}(1) implies that this probability is $1/2+o(1)$, using the invariance of the distribution of $\mbf L$ under permutations of the rows).
Until now, these were the only rigorous pieces of evidence towards
Cameron's conjecture.

Surprisingly, empirical evidence for small Latin squares has run \emph{counter}
to Cameron's conjecture (and therefore \cref{thm:shiny}): it seems
that the approximations in \cref{thm:shiny} only start to ``kick
in'' for reasonably large $n$ (though, there is a limit to what
can be done empirically, because randomly sampling Latin squares is
a notoriously difficult problem). Cameron himself \cite{Cam15} observed
that ``evidence supports this conjecture fairly well, but the tails
of the distribution seem a little heavier than it would predict'',
suggesting that there might be good reason that the H\"aggkvist--Janssen
bound mentioned above is significantly larger than $2^{-n}$. Alimohammadi,
Diaconis, Roghani and Saberi~\cite{ADRA23} attempted to test \cref{conj:cameron}
using \emph{sequential importance sampling} (in the same paper, they
gave rigorous grounding to this method). They were not able to find any evidence
of the total variation convergence\footnote{In fact, they did not even see any evidence of convergence with respect to the so-called \emph{Wasserstein metric} (Wasserstein convergence is significantly weaker than total variation convergence).} $\operatorname d_{\mr{TV}}(N_{\mathrm{row}}(\mathbf{L}),\operatorname{Bin}(n,1/2))\to 0$,
with random Latin squares of size $n\le15$.

\subsection{Notation}
We use standard asymptotic notation throughout: for functions $f=f(n)$ and $g=g(n)$, we write $f=O(g)$ or $g=\Omega(f)$ to mean that there is a constant $C$ such that $|f(n)|\le C|g(n)|$ for sufficiently large $n$, we write $f=\Theta(g)$ to mean that $f=O(g)$ and $f=\Omega(g)$, and we write $f=o(g)$ or $g=\omega(f)$ to mean that $f/g\to0$ as $n\to\infty$. We will often want to treat certain quantities as constants, for the purpose of asymptotic notation; we will always make this clear with phrasing like ``fix a constant $\alpha>0$''.

Somewhat less standardly, for $\varepsilon>0$ we write $f\pm \varepsilon$ to denote a quantity that differs from $f$ by at most $\varepsilon$.

Regarding basic mathematical notation: for a real number $x$, the floor and ceiling functions are denoted $\lfloor x\rfloor=\max(i\in \mb Z:i\le x)$ and $\lceil x\rceil =\min(i\in\mb Z:i\ge x)$. We will however generally omit floor and ceiling symbols and assume large numbers are integers, wherever divisibility considerations are not important. We write $[n]=\{1,\dots,n\}$, and all logarithms are base $e$ unless explicitly stated otherwise.

Finally, since this paper features so many objects of different types, we adopt some typographical conventions: random objects are printed in bold (e.g., a random variable $\mbf X$ or a random Latin square $\mbf L$), and  ``ordered'' objects are decorated with an arrow (e.g., vectors $\vec x$, and later in the paper we will consider ordered partial Latin squares $\vec P$).

\subsection{Acknowledgements}
The authors would like to thank Lenka Kopfov\'a for helpful feedback on an earlier draft of this paper.

\section{Discussion of proof techniques}\label{sec:outline}
In this section we give a high-level overview of the techniques which have previously been brought to bear on random Latin squares, and with this context we describe the new ideas that go into the proof of \cref{thm:main-technical}. (We end the section with an outline of the structure of the rest of the paper.)

As far as we know, all the rigorous work on random Latin squares uses
one of three general classes of techniques.
\begin{enumerate}
\item \textbf{Permutation-invariance.} The distribution
of a uniformly random $n\times n$ Latin square is invariant under random permutation
of the rows, columns, and symbols. So, one can attempt to study random Latin squares purely by considering the effect of such random permutations. This type of reasoning is rather limited, but does have some applications (see e.g.\ \cite{Cam92,LS18}).
\item\smallskip \textbf{Enumeration.} The most obvious way to study the probability
that a random Latin square satisfies a property $\mathcal{P}$ is
to simply count the number of Latin squares satisfying $\mathcal{P}$,
and divide by the total number of Latin squares. Unfortunately, we
only have very crude methods to enumerate $n\times n$ Latin squares
(the best known upper and lower bounds differ by a factor of about $\exp(n\log^2 n)$),
so generally speaking this can only be used to study properties that
are extremely unlikely (i.e., that occur with probability less than
about $\exp(-n\log^2 n)$). Even for such properties, actually
proving the necessary estimates is often a highly nontrivial matter
(and often involves consideration of auxiliary random models, as we discuss in \cref{subsec:enumeration}).
For some examples of properties of random Latin squares proved by
enumeration, see e.g.\ \cite{GK23,KSSS22,KSSS23,KS18,kwan2020almost,LL16,MW99,BM25,Cam15b}.
\item \smallskip{\textbf{Cycle switching.} A common technique in probabilistic combinatorics
is to study a random object by studying the effect of local ``switching''
operations. 
Unfortunately, it is very difficult to make a controlled local change
to a Latin square. To give some idea of the difficulty, it is an open
question to understand the minimum number $f(n)$ such that every
$n\times n$ Latin square can be transformed into some different Latin
square by changing $f(n)$ entries (all we know is that $f(n)$ has
order of magnitude between $\log n$ and $\sqrt{n}$; see \cite{CR17}).
There is only really one type of switching operation that has successfully
been used to study random Latin squares, namely \emph{cycle switchings} (and minor variants thereof). 
We will define cycle switchings properly in \cref{subsec:switchings}, but for now we just emphasise that they are very unwieldy; in general, if we want
to make a particular local change using a cycle switching, we may be forced to make far-reaching
changes to the rest of the Latin square. 
Nonetheless, since cycle
switchings are one of very few available tools, they have played a
crucial role in many of the known results about random Latin squares
(see e.g.\ \cite{HJ96,AW25,CGW08,CW16,KS18}).}
\end{enumerate}

Permutation-invariance alone is certainly not enough to prove Cameron's conjecture, since any permutation of the rows, columns and symbols affects all the row parities in the same way (though, permutation-invariance can be used to easily show that $\Pr[N_{\mathrm{row}}(\mathbf{L})\text{ is even}]=1/2$ when $n\ge 3$ is odd). At first sight, enumeration-based approaches also seem to be fundamentally unsuitable to prove Cameron's conjecture, since we are concerned with events that occur with non-negligible probability, and a multiplicative error of $\exp(n\log^2 n)$ would completely overwhelm the main term. So, cycle-switching techniques seem to be the most promising avenue towards Cameron's conjecture; indeed, such techniques underpin the results of H\"aggkvist--Janssen, Cavenagh--Wanless and Cavenagh--Greenhill--Wanless mentioned in the introduction. However, the usual type of switching analysis is far too crude to have any hope of proving anything like \cref{thm:baby} (we discuss this further in \cref{subsec:switchings,subsec:intercalate-switchings}), and some fundamental new ideas are required.

In this paper, we use a \emph{combination} of cycle-switching and enumeration-based techniques. To very briefly summarise our approach: we use enumeration-based techniques to prove a new approximation theorem, comparing a random Latin square to a certain Erd\H os--R\'enyi-type random hypergraph (in a stronger sense than previously available). This allows us to show that random Latin squares are extremely likely to have a very rich constellation of $2\times 2$ subsquares with certain strong disjointness and canonicity/stability properties. The special properties of these $2\times 2$ subsquares allow us to \emph{independently} perform many cycle switchings, to ``re-randomise'' a random Latin square without biasing its distribution. We are then able to prove the desired results using the randomness of our independent switchings, via some linear algebra over the finite field $\mb F_2$.

In the following subsections, we describe the above ideas in more detail.
\subsection{Enumeration and approximation}\label{subsec:enumeration}

Let $|\mc L_{n}|$ be the number of $n\times n$ Latin squares. This number is known to be
\[
|\mathcal{L}_{n}|=\left(\frac{n}{e^{2}}+o(n)\right)^{n^{2}},
\]
and there are two known ways to prove this.
\begin{enumerate}
\item We can build up an $n\times n$ Latin square in a \emph{row-by-row} fashion: at each step we choose an option for the next row that does not conflict with previous rows. One can use celebrated \emph{permanent estimates} of of Egorychev--Falikman~\cite{Ego81,Fal81} and Bregman~\cite{Bre73} to show that the number of choices for each row does not depend too strongly on previous choices; multiplying these bounds over all steps yields
 \[
 \exp(-O(n\log n))\le\frac{|\mathcal{L}_{n}|}{(n/e^{2})^{n^{2}}}\le\exp(O(n\log^{2}n))
 \]
(see \cite[Theorem 17.3]{vLW01} for the details).

\item{\smallskip We can instead build up an $n\times n$ Latin square in an \emph{entry-by-entry}
fashion: at each step we choose a row/column/symbol triple that does
not conflict with the partial Latin square constructed so far, and
add it. Unfortunately, the number of choices at each step can depend
quite dramatically on previous choices; in particular it is possible
to ``get stuck'', and find oneself in a position where there are no legal choices for the next row/column/symbol triple.

However, it turns out that ``on average'' (i.e., if one makes random choices at each step), this entry-by-entry process is quite well-behaved. Combining a number of very powerful tools, one can use this process to prove
\[
\exp(-O(n^{2-\varepsilon}))\le \frac{|\mathcal{L}_{n}|}{(n/e^{2})^{n^{2}}}\le\exp(O(n^{3/2})).
\]
for some tiny\footnote{It is unclear what is the optimal $\varepsilon$ with this approach, but certainly $\varepsilon$ cannot be greater than $1/2$ (this is a general barrier for techniques based on the \emph{triangle removal process}, which we intend to explore further in upcoming work).} constant $\varepsilon>0$. The details appear in \cite{KSS21}, but are really an adaptation of a general enumeration approach systematised by Keevash~\cite{Kee18}. Specifically, the upper bound is proved using the so-called \emph{entropy method} of Radhakrishnan~\cite{Rad97}, in a form pioneered by Linial and Luria~\cite{LL13}. For the lower bound, one combines an analysis of an instance of the so-called \emph{triangle removal process} (first studied by Spencer~\cite{Spe95} and R\"odl and Thoma~\cite{RT96}), with a \emph{completion theorem} of Keevash~\cite{Kee18c} (here one needs a ``second generation'' completion theorem, building on Keevash's earlier ideas in the setting of the \emph{existence of designs conjecture}~\cite{Kee14}; see also \cite{Kee24,DP,GKLO23}).
}
\end{enumerate}
The approaches in (1) and (2) suggest two different ways to study random Latin squares, by comparing them to two different auxiliary random models. First, a $k\times n$ \emph{Latin rectangle} is a $k\times n$ array filled with $n$ different symbols, such that each symbol appears exactly once in each row and at most once in each column. The ideas in (1) show that all $k\times n$ Latin rectangles can be completed to Latin squares in ``about the same number of ways'' (up to a multiplicative factor\footnote{Of course, this multiplicative factor depends on $k$, but it turns out that most of the contribution comes from the last few rows, so one cannot gain much by taking $k$ very close to $n$.} of about $\exp(n\log^2 n)$), so if we could show that some property is overwhelmingly likely to hold in a random $k\times n$ Latin rectangle, then we could deduce that that property is also very likely to hold in a random $n\times n$ Latin square.

Similarly, consider the random process in (2) (which can be interpreted as an instance of the so-called \emph{triangle removal process}), and suppose we run this process for just a few steps (until the partial Latin square is half-full, say). The ideas in (2) allow us to show that most outcomes of this random process can be completed to Latin squares in ``about the same number of ways'' (up to a multiplicative factor of about $\exp(n^{2-\varepsilon})$), so if we had some way to show that the first few steps of this random process are overwhelmingly likely to satisfy some property, then we could deduce that that property is also very likely to hold in a random $n\times n$ Latin square.

Both of these two methods have quite fundamental quantitative limitations (they can only be used to study properties that hold with overwhelming probability), but they have seen many applications. To compare the two methods: the first method has much better quantitative aspects, and it is ``more elementary'' (not depending on Keevash's sophisticated completion machinery). On the other hand, the second method is generally much easier to apply: the entry-by-entry random process can be straightforwardly coupled with an Erd\H os--R\'enyi random hypergraph, which allows one to take advantage of the huge body of techniques that have been developed to study random graphs\footnote{The second method is also much more robust, and can be applied in much more general settings than random Latin squares, though this is not relevant for the present paper.}. The only known way to directly study random Latin rectangles is via delicate switching arguments (taking advantage of the fact that it is much easier to make local changes to a Latin rectangle than a Latin square).

\subsection{A new approximation lemma}\label{subsec:new-enumeration}
One of our key contributions in this paper is a new approximation lemma (\cref{lem:random_LS_to_triangle_removal}) that combines the advantages of both the above methods (and is proved by a delicate combination of both methods). Roughly speaking, this lemma says that if one can show that the entry-by-entry random process described in the last subsection (i.e., the \emph{triangle removal process}) 
satisfies a property with probability $1-\varepsilon$, then that property holds in a random subset of a random Latin square with probability at least $1-\varepsilon\exp(O(n\log^2 n))$. The reader may also be interested in an easy-to-apply corollary of our main approximation lemma (\cref{cor:random-LS-to-binomial}) comparing random Latin squares to Erd\H os--R\'enyi random hypergraphs.

For example, these tools could have been used to remove the switching analysis (or remove the dependence on general-purpose switching estimates) from \cite{GK23,KSSS22,KS18,MW99,BM25}, and remove the dependence on Keevash's completion machinery from theorems in \cite{KSSS22,kwan2020almost}. For the present paper, our new approximation lemma is rather crucial: quantitative aspects are very important (see \cref{rem:quant}), and the properties we need to consider are so complicated that it would have been a herculean task to study them via switching on Latin rectangles. We hope that our approximation lemma will also be broadly useful in future work on random Latin squares.

\subsection{Cycle switchings}\label{subsec:switchings}
A cycle switching is specified by a pair of rows, a pair of columns
or a pair of symbols, together with an entry of the Latin square ``belonging
to that pair''. For example, consider two
rows $r_{1},r_{2}$ and an entry $(r_{1},c,s)$ (meaning that symbol
$s$ appears in row $r_1$ of column $c$). The first step of the
cycle switching is to ``move our entry into the other row'':
we put the symbol $s$ into cell $(r_2,c)$. However, there was
already another symbol $s'$ in that cell, which we need to move somewhere
else; we move it to the only available cell, which is $(r_1,c)$.
But there is already an instance of $s'$ in row $r_1$ (in some
column $c'$), so we move that instance of $s'$ to the cell $(r_2,c')$.
But there was already some symbol $s''$ in that cell, which we move
to $(r_1,c')$. At this point, if we are very lucky, we will have
$s=s''$, in which case no further moves are necessary, and we have
successfully switched to a new Latin square. If not, we can keep going,
repeatedly moving entries between $r_{1}$ and $r_{2}$ as long as
it takes to resolve all conflicts and reach a Latin square. (In the worst case, we could end
up exchanging the entirety of rows $r_{1}$ and $r_{2}$).

The smallest possible cycle switching is called an \emph{intercalate
switching}. In the above example, it corresponds to the case where
$s=s''$, so the entire switching affects just four entries of the
Latin square. Alternatively, an intercalate switching can be viewed
as ``flipping a $2\times2$ Latin subsquare''. In general, the switchings
described above (``row switchings'', where we are switching between
two rows $r_{1}$ and $r_{2}$) can be viewed as ``flipping a
$2\times k$ Latin subrectangle'', for some $k$.

Cycle switchings have an easy-to-describe impact on the parities
of rows, columns and symbols. For example, a row switching changes the parity of every column and symbol involved
in the switching, and it may or may not change the parity of the two involved rows, depending on whether the length
of the cycle switching is even or odd. More subtly, cycle switchings also have an impact on the structure of other cycle switchings (e.g., switching in a pair of rows can affect the switchings that are possible in some pair of columns), so one can consider combinations of cycle switchings with more complicated effects. It is even possible to combine multiple ``partial'' cycle switchings of different types (cf.\ the ``cross-switch'' operations in \cite{HJ96,CGW08}).

One can use these types of switchings to estimate the \emph{relative likelihoods} of various events. Indeed, given two sets of Latin squares $\mc A,\mc B$, we can design a switching operation to transform a Latin square in $\mc A$ into a Latin square in $\mc B$. Then, for $L_1\in \mc A$ we can try to estimate the number of Latin squares in $\mc B$ that can be reached by such a switching, and for $L_2\in \mc B$ we can try to estimate the number of Latin squares in $\mc A$ that can reach $L_2$ by such a switching. Dividing these two estimates gives us an estimate on $|\mc B|/|\mc A|$.

For example, for a sequence $\vec x\in \{0,1\}^k$, let $\mc L_{\vec x}$ be the set of $n\times n$ Latin squares for which the parities of the first $k$ rows are described by $\vec x$. Using the above types of ideas, H\"aggkvist and Janssen~\cite{HJ96} managed to show that if $k$ is significantly smaller than $n/2$, we have $|\mc L_{\vec x}|\le 3|\mc L_{\vec y}|$ whenever $\vec x$ and $\vec y$ differ in a single coordinate, and they iterated this estimate to deduce \cref{eq:HJ}. Obviously, this comes very far short of proving Cameron's conjecture; the condition on $k$ and the extraneous factor of 3 are basically due to the fact that one has very little control over the structure of cycle switchings in an arbitrary Latin square. To make a desired change to the row parities it is necessary to make case distinctions with different combinations of cycle switchings that interact in different ways, and it is necessary to reserve a large subset of ``junk rows'' which may be affected by the cycle switchings in unpredictable ways.

\subsection{Individual intercalate switchings}\label{subsec:intercalate-switchings}
An important point is that we do not actually need to consider cycle switchings in \emph{arbitrary} Latin squares: we can use the enumeration/approximation methods described in \cref{subsec:enumeration,subsec:new-enumeration} to prove that Latin squares are very likely to satisfy certain properties, and then take advantage of these properties in cycle switching arguments. This sounds like a natural approach, but to our knowledge it has never been employed before, because until recently the properties that could be proved using enumeration/approximation were extremely limited.

In particular, a very relevant direction is the study of intercalates ($2\times 2$ Latin subsquares) in random Latin squares. If an intercalate involves two even (respectively, odd) rows, then switching that intercalate reduces (respectively increases) the number of even rows by exactly two. So, if we had very tight control of the number of intercalates among the even rows, and the number of intercalates among the odd rows, we could hope to estimate the ratios
\begin{equation}\frac{\Pr[N_{\mr{row}}(\mbf L)=x]}{\Pr[N_{\mr{row}}(\mbf L)=x+2]}\label{eq:ratio},\end{equation} for all $x\in \mb N$, by considering all the ways to switch a single intercalate. Note that these ratios fully determine the distribution of $N_{\mr{row}}(\mbf L)$ (except for a possible bias mod 2, but this is handled by \cref{eq:alpoge}).

Recently, resolving a conjecture of McKay and Wanless~\cite{MW99}, Kwan, Sah and Sawhney~\cite{KSSS22} showed how to use enumeration/approximation techniques (together with techniques from large deviations theory) to prove that $\mbf L$ has $n^2/4+o(n^2)$ intercalates with probability $1-o(1)$. These techniques can be adapted to show that $\mbf L$ typically has $k^2/4+o(n^2)$ intercalates in any subset of $k$ rows, which allows one to estimate the ratio in \cref{eq:ratio} up to a multiplicative factor of $1+o(n/(x(n-x)))$. This is already enough to make new progress towards Cameron's conjecture: since binomial tails decay very rapidly, it is straightforward to deduce that $N_{\mr{row}}(\mbf L)=n/2+o(n)$ with probability $1-o(1)$; that is, $N_{\mr{row}}(\mbf L)$ satisfies the same law of large numbers as $\on{Bin}(n,1/2)$. This implies \cref{thm:shiny}(1), and with a bit more work it is also possible to prove the large deviation principle in \cref{thm:shiny}(5) using this type of idea. 

That is to say, by studying the effect of switching an intercalate, one can estimate the ratios \cref{eq:ratio} tightly enough to study the tails of $N_{\mr{row}}(\mbf L)$ (and $N_{\mr{col}}(\mbf L)$ and $N_{\mr{sym}}(\mbf L)$). Unfortunately, it is not feasible to study the bulk of the distribution of $N_{\mr{row}}(\mbf L)$ this way: since the fluctuations of $\on{Bin}(n,1/2)$ have order of magnitude $\sqrt n$, one would need to control the ratios \cref{eq:ratio} up to a factor of $1+o(1/\sqrt n)$; roughly speaking, this is comparable to showing that the number of intercalates $\mbf Y$ in $\mbf L$ satisfies $\Pr[|\mbf Y-\mb E\mbf Y|>t]=o(1)$ for some $t=o(n^{3/2})$. While we believe this to be true, it is beyond the reach of enumeration-based techniques: due to clustering phenomena in the upper tail of $\mbf Y$, the relevant deviation probabilities are simply not small enough\footnote{For this to work, one would need to be able to estimate the number of $n\times n$ Latin squares $|\mc L_n|$ up to a subexponential $\exp(o(n))$ multiplicative error.} to tolerate the super-exponential error terms described in \cref{subsec:enumeration}. (This is an instance of Janson's ``infamous upper tail''~\cite{janson2002infamous}, and is discussed further in \cite{KSSS22}.)

\subsection{Multiple intercalate switchings}
So, instead of studying the effect of a single intercalate switching, our approach is to make \emph{many disjoint intercalate switches at the same time}.

A na\"ive approach for this is as follows. Using our new approximation theorem, together with techniques from \cite{kwan2020almost,KSSS22}, one can show that a random $n\times n$ Latin square $\mbf L$ typically has a collection of disjoint intercalates $\mc I$ which ``robustly span almost all the rows'', in the sense that for every set $R_0$ of about\footnote{We cannot go smaller than about $\log^2 n$, due to the multiplicative error of about $\exp(n\log^2 n)$ in our enumeration estimates.} $\log^2 n$ rows, there is an intercalate in $\mc I$ which involves a row in $R_0$ and a row outside $R_0$. Some simple linear algebra over $\mb F_2$ then shows that if we were to randomly and independently switch all the intercalates in $\mc I$ (thereby obtaining a new random Latin square $\mbf L'$), the resulting distribution of the $n$ row parities would be uniform on some affine-linear subspace of $\mb F_2^n$ with codimension at most about $\log^2 n$, and it is easy to deduce a central limit theorem for $N_{\mr{row}}(\mbf L')$. (With a more sophisticated argument, one can handle the joint distribution of $(N_{\mr{row}}(\mbf L'),N_{\mr{col}}(\mbf L'),N_{\mr{sym}}(\mbf L'))$ and upgrade the central limit theorem to a \emph{local} central limit theorem.)

\begin{remark}\label{rem:quant}
    We are really proving a central limit theorem for $N_{\mr{row}}(\mbf L')$ conditioned on an outcome of $\mbf L$. So, it is important that for different outcomes of $\mbf L$, the corresponding conditional distributions of $N_{\mr{row}}(\mbf L')$ are statistically indistinguishable. This would not be true if the codimension above were greater than $\sqrt n$, since the fluctuations of $\on{Bin}(n,1/2)$ have order of magnitude $\sqrt n$. So, the quantitative aspects of our new approximation theorem are very important here. (E.g., the approximation theorem in \cite{kwan2020almost} would not have sufficed.)
\end{remark}

Unfortunately, this only characterises the distribution of $N_{\mr{row}}(\mbf L')$, which could in principle be very different to the distribution of $N_{\mr{row}}(\mbf L)$. The problem is that switching some intercalates can destroy other intercalates, or create new ones, so the probability of switching from $L$ to $L'$ could be different from the probability of switching from $L'$ to $L$ (i.e., the switching could ``push probability mass'' towards a subset of Latin squares, introducing bias to the distribution of $\mbf L'$). Trying to quantify this bias takes us back to the issues described in \cref{subsec:intercalate-switchings}.

The only way we were able to resolve this issue was to execute the entire argument above with a \emph{stable}/\emph{canonical} collection of disjoint intercalates. Specifically, we associate a collection of disjoint intercalates $\mc I(L)$ to each Latin square $L$, in such a way that if we switch any subset of these intercalates to obtain a new Latin square $L'$, then $\mc I(L)$ and $\mc I(L')$ are exactly the same (except that obviously some of these intercalates are switched). So, the probability of switching from $L$ to $L'$, and the probability of switching from $L'$ to $L$, are both exactly $2^{-|\mc I(L)|}=2^{-|\mc I(L')|}$, and $\bfL'$ is a uniformly random Latin square.

In a typical Latin square, the intercalates heavily intersect each other (and therefore cannot be independently switched). Given a pair of intersecting intercalates, how do we decide which is ``the canonical one'' that we're allowed to switch? Our approach is to \emph{randomly sparsify} the situation, to reduce the intersections between intercalates. Namely, we consider a random ``template'' $\mbf T$, which describes a sparse subset of ``switchable'' row/column pairs. Then, given a Latin square $L$, we define a collection of intercalates $\mc I_{\mbf T}(L)$ by starting with the collection of intercalates in $L$ which use only row/column pairs in $ \mbf T$, then deleting the intercalates which are ``non-canonical'' (the precise definition is a little complicated, but e.g.\ we delete all intercalates in intersecting pairs, and we delete all intercalates that could introduce a new intercalate when switched). The idea is that this deletion step should not have a very severe effect, since $ \mbf T$ is sparse. We need to show that a typical outcome of $\mbf T$ yields collections of intercalates $\mc I_{\mbf T}(L)$ which ``robustly span almost all the rows'' for almost all $L$, with which we can implement the argument described above.

Actually proving the necessary properties of $\mc I_{\mbf T}(\mbf L)$ is a very delicate matter (much more so than any previous work on random Latin squares), largely due to the aforementioned ``infamous upper tail'' issue. Specifically, because enumerative estimates have super-exponential error terms, most steps of the proof need to hold with overwhelming probability, so we constantly need to worry about large deviation behaviour. Whenever we need an upper bound on some quantity, we need to be extremely careful to avoid situations where clustering phenomena cause upper tails to be too heavy. 

\subsection{Outline of the paper} The structure of the proof of \cref{thm:shiny}, as distributed over the rest of the paper, is as follows. First (after \cref{sec:concentration}, in which we record some standard concentration inequalities that will be used throughout the paper), in \cref{sec:approximation-1} we prove our new approximation theorem (\cref{lem:random_LS_to_triangle_removal}) relating random Latin squares to the triangle removal process. In \cref{sec:approximation-2} we prove some simple lemmas relating the triangle removal process to Erd\H os--R\'enyi random hypergraphs, which will be useful in combination with \cref{lem:random_LS_to_triangle_removal}.

In \cref{sec:master}, we reduce all parts of \cref{thm:shiny} to a technical ``master theorem'' (\cref{thm:main-technical}) giving a precise description of the distribution of $(N_{\mr{row}}(\mbf L),N_{\mr{col}}(\mbf L),N_{\mr{sym}}(\mbf L))$. Then, in \cref{sec:stable}, we introduce the notion of a ``stable intercalate'', which is an intercalate which can be safely switched without affecting any other intercalates. In \cref{sec:master-from-stable} we show how to use linear-algebraic arguments with stable intercalate switchings to prove the master theorem described above, given a key lemma (\cref{lemma:existence_split_stable_intercalate}) on stable intercalates in random Latin squares. Roughly speaking, the key lemma says that there is a ``template'' $T$ (obtained via a sparse random set of row/column pairs), such that for a uniformly random Latin square, it is very likely that for any large-enough sets of rows, columns and symbols, we can find a stable intercalate (involving only entries in $T$) ``inside'' those sets.

The rest of the paper is devoted to the proof of \cref{lemma:existence_split_stable_intercalate}. In \cref{sec:setup_int_expander_theorem} we break down this proof into two lemmas. First, \cref{lemma:min_disjoint_split_ints} says that (inside any large-enough sets of rows, columns and symbols) there are likely to be many disjoint intercalates (saying nothing about whether they are stable). Second, \cref{lemma:max_entries_bad_configs} says that there are unlikely to be many entries which lie in certain ``bad'' arrangements of intercalates which could lead to our intercalates being non-stable. (Here, due to ``infamous upper tail'' issues, we cannot hope to consider the number of bad arrangements themselves, only the number of entries in them!)

The first of these lemmas (\cref{lemma:min_disjoint_split_ints}) is proved in \cref{sec:existence_disjoint_ints}, via a careful 2-step application of Freedman's martingale concentration inequality (using a ``maximum disjoint family'' technique of Bollob\'as~\cite{Boll88}).

The second lemma (\cref{lemma:max_entries_bad_configs}) is proved in \cref{sec:upper_bounding_entries_in_bad_configurations}, and is much more involved. Here, our job is to understand ``bad'' arrangements of intercalates; there is an enormous range of possibilities for the structure of such an arrangement, so the first step is a switching argument (unrelated to the main switching argument in \cref{sec:master-from-stable}), which shows that every bad arrangement of intercalates can be switched to obtain one of four ``basic types'', and therefore it suffices to restrict our attention to these types. We then carefully study how the four basic types of bad intercalate arrangements can emerge in the triangle removal process.

\section{Concentration inequalities}\label{sec:concentration}

Throughout this paper we will frequently use a general-purpose concentration inequality for functions of independent Bernoulli random variables. The following statement appears as \cite[Corollary~6]{warnke2016method}\footnote{In~\cite{warnke2016method}, \cref{freedman_inequality} is stated only for the upper tail of $f(\vec {\mathbf g})$; applying it to $-f(\vec {\mathbf g})$ yields the lower tail.}; it follows from the martingale approach of Freedman~\cite{Fre75}.
\begin{theorem}
\label{freedman_inequality}
Let $\vec{\mathbf{g}} = (\mathbf{g}_1, \dots, \mathbf{g}_N)$ be a sequence of independent Bernoulli random variables with $\mathbb{P}[\mathbf{g}_i=1]= p_i$. Let $f:\{0,1\}^N \rightarrow \mathbb{R}$ satisfy $| f(\vec {g}) - f(\pvec{g})| \leq K_i$ for all pairs $\vec{g}, \pvec{g} \in \{0,1\}^N$ differing only in the $i$-th coordinate, and suppose $K_i \leq K$ for all $i\in[N]$. Then for all $t \geq 0$, we have
\[ \mathbb{P}\Big[\big|f(\vec {\mathbf g}) - \mathbb{E}[f(\vec{ \mathbf{g}})]\big| > t\Big] \leq 2\exp\Bigg( - \frac{t^2}{ 2 \sum_{i=1}^N p_i K_i^2 + 2Kt/3 }\Bigg).\]
\end{theorem}

Note that, in order to obtain a strong bound from \cref{freedman_inequality}, we need to know that $f(\vec g)$ is not very sensitive to changes in any individual $g_i$. In this respect \cref{freedman_inequality} is similar to the much more well-known \emph{bounded differences inequality} of McDiarmid~\cite{McD89} (which is a consequence of the Azuma--Hoeffding martingale concentration inequality). However, we need the stronger statement of \cref{freedman_inequality} for its essentially-optimal dependence on $p_1,\dots,p_N$.

We also record a corollary of \cref{freedman_inequality} in the special case of weighted sums of Bernoulli random variables (this is basically the standard Chernoff bound).

\begin{corollary}
\label{thm:chernoff_dependency_graph}
Let $\mathbf{g}_1, \dots, \mathbf{g}_N$ be independent Bernoulli random variables, consider weights $w_1,\dots,w_N\in[0, \Delta]$, and let $\bfX = w_1\mathbf g_1+\dots+w_N\mathbf g_N$. Then, for any $\delta>0$ we have
\[ \mathbb{P}\big[|\bfX - \mathbb E\mathbf X| \geq  \delta\mathbb E\mathbf X\big] \leq 2 \exp \Bigg( - \frac{\delta^2 \mathbb E\mathbf X}{(2+2\delta/3) \Delta }\Bigg). \]
\end{corollary}

\section{Approximating random Latin squares with the triangle removal process}\label{sec:approximation-1}

In this section, we prove an approximation lemma that relates random Latin squares to the \emph{triangle removal process} (which is a tractable probability distribution on \emph{partial} Latin squares). To state this lemma, we need some preparations. First, it will be convenient for us to take an alternative perspective on Latin squares.

\begin{fact}
\label{fact:hypergraph_view}
An $n\times n$ Latin square can be interpreted as a tripartite 3-uniform hypergraph (specifically, a subgraph of the complete tripartite 3-uniform hypergraph $K^{(3)}_{n,n,n}$), in which every pair of vertices in different parts features in exactly one hyperedge.

To see the correspondence, think of the three parts as the set of rows, the set of columns and the set of symbols; a hyperedge $(r,c,s)$ means that symbol $s$ appears in the $c$-th column of the $r$-th row.
\end{fact}
 From now on, we interchangeably take the hypergraph point of view and the  $n\times n$ array point of view, depending on whichever is most convenient at any given moment (generally speaking, the hypergraph point of view is more expressive, and highlights the symmetry between rows, columns and symbols, but it is easier to draw pictures with the $n\times n$ array point of view).

\begin{definition}
    Let $\cL_n$ be the set of all $n \times n$ Latin squares.
    
    A \emph{partial} Latin square is a tripartite $3$-uniform hypergraph in which every pair of vertices in different parts features in \emph{at most} one hyperedge (alternatively, one can also interpret this as an $n\times n$ array in which only some of the entries are filled with symbols). Let $\cP_{n,m}$ be the set of all $n \times n$ partial Latin squares with $m$ hyperedges, and let $\mc P_n=\bigcup_{m \leq n^2}\cP_{n,m}\supseteq \mc L_n$ be the set of all $n\times n$ partial Latin squares (with any number of hyperedges).
    
    An \emph{ordered} partial Latin square is a partial Latin square together with an ordering on its hyperedges. Let $\vec \cP_{n,m}$ be the set of all $n \times n$ ordered partial Latin squares with $m$ hyperedges, and let $\vec \cP_n=\bigcup_{m \leq n^2}\vec \cP_{n,m}$.
\end{definition}

Now, our approximation lemma will compare uniformly random Latin squares to a distribution on (ordered) partial Latin squares; for this to make sense we need a notion of inheritance from a property of Latin squares to a property of (ordered) partial Latin squares.

\begin{definition}\label{def:inherited}
    Let $\vec\cU \subseteq \vec{\cP}_{n}$ be a property of ordered partial Latin squares and let $\cT \subseteq \cL_n$ be a property of Latin squares. For $m \leq n^2$, we say $\vec\cU$ is \emph{$(\rho,m)$-inherited} from $\cT$ if for any $L \in \cT$, taking $\vec{\bfP}_m(L)$ to be a uniformly random subset of $m$ hyperedges of $L$, equipped with a uniformly random order, we have $\vec{\bfP}_m(L) \in \vec\cU$ with probability at least $\rho$. 
\end{definition}
\begin{remark}\label{rem:exotic}
    It will often be quite obvious how to specify the property $\vec {\mc U}$, for a desired event $\mc T$ (and often, the ordering will play no role). For example, if $\cT$ is the property of having fewer than $(1-\varepsilon)n^2/4$  intercalates, it would make sense to take $\vec \cU$ to be the property of having fewer than $(1-\varepsilon/2)(m/n^2)^4n^2/4$ intercalates. However, as we will see later in the paper, there is sometimes reason to make more exotic choices for $\vec \cU$.
\end{remark}

\begin{remark}
    Throughout the paper we will need to consider many different properties of (ordered) (partial) Latin squares, and we have attempted to adopt some notational conventions to help the reader keep track of the ``type'' of each property. Specifically, for properties of partial Latin squares, we will usually use the letter $\mc U$ (think ``unfinished'') and for properties of complete Latin squares we will usually use the letter $\mc T$ (think ``total''). An arrow on top of a property indicates that it is a property of ordered (rather than unordered) partial Latin squares. 
\end{remark}

Now, the main lemma of this section concerns the \emph{triangle removal process}, which we now define.
\newcommand{\abort}{\perp}
\begin{definition}
Let $\trp(n,m)$ be the distribution over $\vec \cP_{n,m} \cup \{\abort\}$ obtained as follows. Start with the complete tripartite graph $K_{n,n,n}$. At each step, choose a uniformly random triangle (among all the triangles in the current graph), and remove it. After $m$ steps of this process, the sequence of removed triangles can be interpreted as an ordered partial Latin square $\vec R\in \vec \cP_{n,m}$ (unless we ran out of triangles at some point in the process, in which case we take $\vec R=\abort$).
\end{definition}

We are now ready to state the main result of this section, comparing a uniformly random Latin square to the triangle removal process.

\begin{lemma}
    \label{lem:random_LS_to_triangle_removal}
    Fix constants $\alpha \in (0,1)$ and  $\rho \in (0,1]$, and let $m=\alpha n^2$. Consider properties $\vec{\cU} \subseteq \vec{\cP}_{n,m}$ and $\cT \subseteq \cL_n$ such that $\vec{\cU}$ is $(\rho,m)$-inherited from $\cT$. Let $\vec{\mathbf{R}} \sim \trp(n,m)$ be an ordered partial Latin square obtained by $m$ steps of the triangle removal process, and let $\mathbf{L} \sim \on{Unif}(\mc L_n)$ be a uniformly random $n \times n$ Latin square. Then
    \[ \mathbb{P}[\mathbf{L} \in \cT] \leq \exp\Bigg(\frac{n \log^2 {n}}{1-\alpha-o(1)}\Bigg) \mathbb{P}[\vec{\mathbf{R}} \in \vec{\cU}].\]
\end{lemma}

\cref{lem:random_LS_to_triangle_removal} is a quantitative improvement to \cite[Theorem~2.4]{KSSS22} (which is itself an adaptation of \cite[Theorem 2.4]{kwan2020almost}). Our proof of \cref{lem:random_LS_to_triangle_removal} bears some similarities to the proofs in \cite{kwan2020almost,KSSS22}, but incorporates various additional ideas (related to our use of enumeration results that are stronger but much less robust).

\begin{remark}
    It may seem that \cref{lem:random_LS_to_triangle_removal} is not suitable for studying ``local'' properties (e.g., properties concerning a particular vertex $v$). Such properties are important for the so-called \emph{absorption method}, which (is not relevant for the main results in this paper but) has played a central role in previous work on random Latin squares. Indeed, it is not hard to see that with probability at least $\exp(-O(n))$, the random hypergraph $\vec{\mbf R}$ has no edges at all containing $v$, and this probability seems too large to have any hope of applying \cref{lem:random_LS_to_triangle_removal} to any interesting local property about $v$. However, one can generally overcome this issue with a judicious choice of $\vec{\mc U}$, cf.\ \cref{rem:exotic}. For example (very informally speaking), if $\mc T$ is the property that $v$ does not participate in a desired local structure, we can take $\vec \cU$ to be the property that $v$ has high degree but still does not participate in the desired local structure.
\end{remark}

The rest of this section is devoted to the proof of \cref{lem:random_LS_to_triangle_removal}. We start with some further definitions.

\begin{definition}
    For an (ordered) partial $n \times n$ Latin square $P$, let $G(P)\subseteq K_{n,n,n}$ be the graph which contains an edge $uv$ (of $K_{n,n,n}$) if and only if there is no hyperedge of $P$ containing $u$ and $v$. That is to say, $G(P)$ is the graph of pairs that would need to be covered to complete $P$ to a full Latin square.

    For a graph $G\subseteq K_{n,n,n}$, let $\on{dens}(G) = e(G)/(3n^2)$.
\end{definition}

\begin{definition}
We say that a graph $G\subseteq K_{n,n,n}$, is \emph{$\varepsilon$-triangle-typical} if its number of triangles is $(1\pm \varepsilon) n^3 \on{dens}(G)^3$. Let $\cP^{\triangle:\varepsilon}_{n, m}\subseteq \cP_{n,m}$ be the set of partial Latin squares $P \in \cP_{n, m}$ such that $G(P)$ is $\varepsilon$-triangle-typical. Let $\vec\cP^{\triangle:\varepsilon}_{n,m}\subseteq \vec \cP_{n,m}$ be the set of ordered partial Latin squares $\vec P\in \vec \cP_{n,m}$ such that for each $i\leq m$, if we consider the partial Latin square $P_i$ consisting of the first $i$ hyperedges of $\vec P$, then $P_i\in \cP^{\triangle:\varepsilon}_{n, m}$.

\end{definition}

\begin{definition}
We say that a graph $G\subseteq K_{n,n,n}$, is \emph{$\gamma$-quasirandom} if for all pairs of vertices $u,v$ in different parts, the number of common neighbours of $u$ and $v$ is $(1 \pm \gamma) n \on{dens}(G)^2$. Let $\cP^\gamma_{n,m}$ be the set of partial Latin squares $P\in \cP_{n,m}$ such that $G(P)$ is $\gamma$-quasirandom.
\end{definition}    

\begin{remark}
    It is not hard to see that if a graph is $\gamma$-quasirandom, then it is necessarily $O(\gamma)$-triangle-typical (this fact is used in the previous approximation lemmas in \cite{kwan2020almost,KSSS22}). However, in our proof of \cref{lem:random_LS_to_triangle_removal} (with its strong quantitative aspects), we will need to consider quasirandomness and triangle-typicality with quite different parameters.
\end{remark}

In the next two lemmas, we show that for \emph{any} Latin square $L\in \cL_n$, if we take a random ordering of a random subset of $L$, then the resulting random object is likely to satisfy certain triangle-typicality and quasirandomness properties. This will allow us to assume that various graphs we encounter later in the proof are triangle-typical/quasirandom, which will be necessary for our enumeration techniques.

\begin{lemma}
\label{lem:almost-always-triangle-typical}
Fix a constant $\alpha \in (0,1)$. Consider any $L \in \cL_n$ and let $\vec \bfP_{\alpha n^2}(L)\in \vec \cP_{n,\alpha n^2}$ be a random ordering of a random set of $\alpha n^2$ hyperedges of $L$. Then, with $\varepsilon = n^{-1} \log{n}$, we have
\[ \mathbb{P}[\vec \bfP_{\alpha n^2}(L) \notin \vec\cP^{\triangle:\varepsilon}_{n,\alpha n^2}] \leq \exp(- \Omega(\log^2{n})).\]

\end{lemma}

\begin{proof}
    Fix some $m \leq \alpha n^2$. We will show that $G(\vec \bfP_m(L))$ is $\varepsilon$-triangle-typical with probability at least $1 - \exp(- \Omega(\log^2{n}))$. The statement of the lemma will then follow by a union bound over all $m \leq \alpha n^2$.

    Note that $\vec \bfP_m(L)$ consists of a uniformly random set of $m$ hyperedges of $L$. Let $\bfB_p(L)$ be a random subset 
    of the hyperedges of $L$ where each hyperedge is taken independently at random with probability $p=m/n^2$. By ``Pittel's inequality'' (see e.g. page 17 in~\cite{janson2011random}), for any property $\cQ$ of unordered partial Latin squares,
    \[ \mathbb{P}[\vec \bfP_m(L) \notin \cQ] \leq 3\sqrt{m} \mathbb{P}[\bfB_p(L) \notin \cQ].\]
    Thus, letting $\mbf X$ be the number of triangles in $G(\vec \bfP_m(L))$, it suffices to show that $\mathbb{P}[\mbf X \notin (1\pm \varepsilon) n^3 (1-p)^3 ] \leq \exp(-\Omega(\log^2{n}))$.

    For any $v_1 \in V_1, v_2 \in V_2, v_3 \in V_3$ such that $v_1v_2v_3 \in L$, the probability that $v_1v_2v_3$ is a triangle in $G(\bfB_p(L))$ is precisely $1-p = 1- m/n^2$. On the other hand, if $v_1v_2 v_3 \notin L$, the probability $v_1v_2v_3$ is a triangle in $G(\bfB_p(L))$ is $(1-p)^3$, since each pair of vertices among $v_1, v_2, v_3$ is in a distinct hyperedge of $L$, and all these hyperedges need to be absent from $\bfB_p(L)$ in order for $v_1v_2v_3$ to be in $G(\bfB_p(L))$. Thus,
    \[\mu := \mathbb{E}[\mathbf{X}] = n^2 (1-p) + (n^3 - n^2) (1-p)^3  = \Theta (n^3).\]
    There are $n^2$ hyperedges in $L$, and adding or removing a hyperedge from $\mbf B_p(L)$ can affect $\mathbf{X}$ by at most $1  + 3(n-1)$, since each pair of vertices in $e$ can form a triangle with $n-1$ other vertices in $G(\mbf B_p(L))$. Thus, by \cref{freedman_inequality}, we have
    \begin{align*} \mathbb{P}[\mathbf{X}\notin (1\pm \varepsilon) n^3 (1-p)^3 ] &\leq  \mathbb{P}[ |\mathbf{X} - \mu | \geq \varepsilon \mu /2 ] \\
    &\leq 2 \exp\Bigg(- \frac{\varepsilon^2 \mu^2 / 4}{ 2n^2 p \big( 1 + 3 (n-1) \big)^2 + 2\big( 1 + 3 (n-1) \big) \varepsilon \mu /6  } \Bigg)\\ &= \exp \big( - \Omega (\log^2{n}) \big),\end{align*}
    where in the first inequality we used that $\mu - n^3 (1-p)^3 = o\big(\varepsilon n^3 (1-p)^3 \big)$.
\end{proof}
\begin{lemma}
\label{lem:almost-always-quasirandom}
Fix a constant $\alpha \in (0,1)$. Consider any $L \in \cL_n$ and let $\bfP_{\alpha n^2}(L)\in \mc P_{n,\alpha n^2}$ be a random subset of $\alpha n^2$ hyperedges of $L$. Let $\gamma =\omega(n^{-1/2}\log^{1/2}{n})$. Then 
\[ \mathbb{P}[\bfP_{\alpha n^2}(L) \notin \cP^{\gamma}_{n,\alpha n^2}] \leq \exp ( - \Omega (\gamma^2 n) ).\]
\end{lemma}
\begin{proof}
    Let $m := \alpha n^2$. We will show that $G(\bfP_m(L))$ is $\gamma$-quasirandom with probability at least $1 - \exp ( - \Omega (\gamma^2 n) )$.
    
    As in the proof of \cref{lem:almost-always-triangle-typical}, let $\bfB_p(L)$ be a random subset of the hyperedges of $L$ where each hyperedge is taken independently at random with probability $p=m/n^2 = \alpha$. Let $\mbf S_{u,v}$ be the set of common neighbours of $u$ and $v$, with respect to the graph $G(\bfB_p(L))$.
    It suffices to show that for every $u,v$ that are in different parts of $G(\bfB_p(L))$, it holds that $\mathbb{P}[|\mbf S_{u,v}| \notin (1\pm \gamma) n (1-\alpha)^2 ] \leq \exp ( - \Omega (\gamma^2 n) )$. A union bound over all pairs of vertices $u,v$ and Pittel's inequality would then complete the proof.

    Fix some $u,v$ in different parts of $G(\bfB_p(L))$. Let $w$ be the unique vertex such that $uvw$ is a hyperedge of $L$. Then $\mathbb{P}[w \in \mbf S_{u,v}] = 1-p$. On the other hand, for all $w' \neq w$ that are in the same part as $w$, we have $\mathbb{P}[w' \in \mbf S_{u,v}] = (1-p)^2$ (since $w'$ participates in some hyperedge together with $u$ and in another hyperedge together with $v$, and both of these need to not be selected in $\bfB_p(L)$ to have $w' \in \mbf S_{u,v}$). Thus, for $\mathbf{X} = |\mbf S_{u,v}|$, 
    \[\mu := \mathbb{E}[\mathbf{X}] = (1-p) + (n-1) (1-p)^2  = \Theta (n).\]
    There are $2n-1$ hyperedges in $L$ that affect $\bfX$, and adding or removing such a hyperedge from $\mbf B_p(L)$ can affect $\mathbf{X}$ by at most $1$. Thus, by \cref{freedman_inequality}, we have
    \[ \mathbb{P}[\mathbf{X}\notin (1\pm \gamma) n (1-p)^2 ] \leq  \mathbb{P}[ |\mathbf{X} - \mu | \geq \gamma \mu /2 ] \leq \exp\Big(- \frac{\gamma^2 \mu^2 / 4}{ 2(2n-1) + \gamma \mu/3 } \Big) = \exp ( - \Omega (\gamma^2 n) ),\]
    where in the first inequality we used that $\mu - n (1-p)^2 = (1-p)p = o(\gamma \mu)$.
\end{proof}

We next prove an upper bound on the number of completions of a partial Latin square, under triangle-typicality and quasirandomness assumptions. 
\begin{theorem}
    \label{thm:upper-bound-completions}
    Fix a constant  $\alpha\in (0,1)$, let $\gamma =o(1)$ and let $\varepsilon = n^{-1}\log{n}$.
    For any partial Latin square $P \in \cP^\gamma_{n,\alpha n^2} \cap \cP^{\triangle:\varepsilon}_{n,\alpha n^2}$, the number $|\mc L_n(P)|$ of completions of $P$ to an $n\times n$ Latin square satisfies
    \[ |\mc L_n(P)| \leq \Bigg( \frac{(1-\alpha)^2 n}{e^2} \Bigg)^{(1-\alpha)n^2}\!\!\!\! \cdot \exp\left(\frac{n \log^2{n}}{1-\alpha-o(1)}\right).\]
\end{theorem}

\begin{proof}
    For a row $x$ (respectively, column $y$), we write $r_x$ (respectively, $c_y$) for the number of empty cells in that row, in $P$. Then for an empty cell $e = (x,y)$, we write $Q_{x,y}$ for the set of symbols that do not appear in the row or column of $e$. Note that $Q_{x,y}$ is precisely the set of common neighbours of $x$ and $y$ in $G(P)$. Then, let $\bfL(P)$ be a uniformly random completion of $P$. We will estimate the entropy $H[\bfL(P)] = \log |\mc L_n(P)|$ of $\bfL(P)$.

    Let $E \subseteq [n]^2$ be the set of empty cells in $P$. For each $e = (x,y) \in E$, let $\bfz_e$ be the symbol in cell $e$ in $\bfL(P)$. So the sequence $(\bfz_e)_{e \in E}$ determines $\bfL(P)$. For any total ordering $\prec$ on $E$, we have
    \begin{align}
    \label{equation:entropy-cell-by-cell}
    H[\bfL(P)] = \sum_{e \in E} H\big[\bfz_e \,\big|\,(\bfz_{e'} : e' \prec e)\big].
    \end{align}
    Now, consider a pair of sequences $\mu = (\mu_x)_x, \nu = (\nu_y)_y\in [0,1]^n$ (of real numbers in the range $[0,1]$, which we assume to be all distinct from each other). Let $\lambda = (\lambda_{x,y})_{x,y}$ be the array defined by $\lambda_{x,y} = (\mu_x, \nu_y)$. We can use $\lambda$ to define a total order $\prec_\lambda$ on $E$, via reverse lexicographic order on the pairs $\lambda_{x,y}$. To be precise, write $(x',y')\prec_\lambda(x,y)$ when 
    $\mu_{x'}>\mu_{x}$ or when $x'=x$ and $\nu_{y'} > \nu_y$. 
    Let $\bfR_e(\lambda)$ be an upper bound on the conditional support size
    \[\Big|\supp\big(\bfz_e\,\big|\,(\bfz_{e'} : e' \prec_\lambda e)\big)\Big|\]
    defined as follows. $\bfR_{e}(\lambda)$ for $e=(x,y)$ is $1$ plus the number of symbols $z \in Q_e \setminus \bfz_e$ for which $\mu_{x_z} < \mu_x$ and $\nu_{y_z} < \nu_y$, where $x_z$ and $y_z$ are the row and column such that cells $(x_z,y)$ and $(x,y_z)$ contain symbol $z$ in $\bfL(P)$.

    Since $\bfR_e(\lambda)$ is an upper bound on $|\supp(\bfz_e\,|\,(\bfz_{e'}:e' \prec_\lambda e))|$, we have
    \begin{align}
    \label{inequality:entropy-upper-bound}
    H\big[\bfz_e\,\big|\,(\bfz_{e'}:e' \prec_\lambda e)\big] \leq \mathbb{E}[\log{\bfR_e(\lambda)}].
    \end{align}
    It follows from~(\ref{equation:entropy-cell-by-cell}) and~(\ref{inequality:entropy-upper-bound}) that
    \[ H[\bfL(P)] \leq \sum_{e \in E} \mathbb{E}[\log{\bfR_e(\lambda)}].\]
    This is true for any fixed $\lambda$, so it is also true if $\lambda$ is chosen randomly: let $\bflambda \in \Big( [0,1]^2\Big)^{n^2}$ be obtained via $2n$ independent random variables $\bfmu_x, \bfnu_y$ each uniformly distributed in $[0,1]$ (with probability $1$ they are all distinct). Then
    \[ H[\bfL(P)] \leq \sum_{e \in E} \mathbb{E}[\log{\bfR}_e(\bflambda)]. \]
    Next, for any completion $L$ of $P$, any $e = (x,y)$ and any $\lambda_e = (\mu_x, \nu_y) \in [0,1]^2$, let
    \[ R_e^{L,\lambda_e} = \mathbb{E}[\bfR_e(\bflambda)\, |\, \bfL(P) = L, \bflambda_e=\lambda_e].\]
    (Note that $\bflambda_e = \lambda_e$ occurs with probability zero, so formally we should condition on $\bflambda_e=\lambda_e \pm \mathrm{d} \lambda_e$ and take limits in what follows, but there are no continuity issues so we will ignore this detail). By the definition of $\bfR_e(\bflambda)$ and linearity of expectation, we have
    \[ R_e^{L, \lambda_e} = 1 + (|Q_e| - 1) \mu_x \nu_y.\]
    By Jensen's inequality,
    \[ \mathbb{E}[\log{\bfR_e(\bflambda)} \,|\, \bfL(P) = L, \bflambda_e=\lambda_e] \leq \log{R_e^{L,\lambda_e}},\]
    so
    \begin{align*}
    \mathbb{E}[\log{\bfR_e(\bflambda)} \,|\, \bfL(P) = L] &\leq \mathbb{E}\Big[ \log{R_e^{L, \bflambda_e}} \Big]\\
    &= \int_{[0,1]^2} \log\big(1 + (|Q_e| - 1) \mu_x \nu_y\big) \,\mathrm{d} \lambda_e.\end{align*}
    For $C > 0$ we compute
    \[ \int_0^1 \int_0^1 \log(1 + Cts) \,\mathrm{d}t\, \mathrm{d}s = \log(1+C) - 2 + \frac{\log(1+C) - \Li_2(-C)}{C}, \]
    where
    \[ \Li_2(C)= - \int_0^C \frac{\log(1-t)}{t}\,\mathrm{d}t \]
    is an evaluation of the polylogarithm function, so, as $C \rightarrow \infty$,
    \begin{align*} -\Li_2(-C) =  \int_0^{-C} \frac{\log(1-t)}{t} \,\mathrm{d}t &= \int_0^{C} \frac{\log(1+t)}{t}\,\mathrm{d}t\\
    &\leq \int_{0}^1 \frac{\log(1+t)}{t}\, \mathrm{d}t + \sum_{t=1}^{C-1} \frac{\log(2+t)}{t} \leq \log^2(C+1)  + O(1).\end{align*}
    Therefore, substituting $C = |Q_e|-1$, we have
    \[ \mathbb{E}[\log{\bfR_e(\bflambda)} \,|\, \bfL(P) = L]  \leq \log{|Q_e|} - 2 + \frac{\log{|Q_e|} + \log^2{|Q_e|} + O(1) }{|Q_e|} = \log{|Q_e|} - 2 +\frac{(1 + o(1)) \log^2{|Q_e|} }{|Q_e|}. \]
    Now recall that $P \in \cP^\gamma_{n,\alpha n^2}$, and that $|Q_e| = |Q_{x,y}|$ is precisely the number of common neighbours of $x$ and $y$ in $G(P)$. Since $G(P)$ is $\gamma$-quasirandom, we have $ |Q_e| = (1\pm \gamma) n (1-\alpha)^2=n (1-\alpha-o(1))^2$. Thus, using $|Q_e| \leq n$, and noting that the expression above does not depend on the choice of $L$, we get
    \[ \mathbb{E}[\log{\bfR_e(\bflambda)} ] \leq \log{|Q_e|} - 2 +\frac{ \log^2{n}}{ (1-\alpha-o(1))^2 n}, \]
    Next, since $P \in \cP^{\triangle:\varepsilon}_{n,\alpha n^2}$ with $\varepsilon:=n^{-1}\log{n}$, we have that the number of triangles in $G(P)$  is
    \[\sum_{e \in E} |Q_e|=(1 \pm \varepsilon)n^3(1-\alpha)^3.\]
    Since the natural logarithm is a concave function, Jensen's inequality yields
    \[ \sum_{e\in E} \log{|Q_e|} \leq |E| \log\Bigg( \frac{\sum_{e \in E} |Q_e|}{|E|} \Bigg) \leq |E| \log \Big( (1+\varepsilon) n(1-\alpha)^2 \Big).\]
    We now get
    \begin{align*}
     \log{|\mc L_n(P)|} = H[\bfL(P)]\leq \sum_{e \in E} \mathbb{E}[\log{\bfR_e(\bflambda)}] &\leq \sum_{e \in E} \log{|Q_e|} + |E| \Bigg(\frac{ \log^2{n}}{(1-\alpha-o(1))^2 n} - 2 \Bigg)\\
     &\leq n^2(1-\alpha) \Bigg( \log\big((1+\varepsilon)(1-\alpha)^2 n \big) - 2  \Bigg) + \frac{n \log^2{n}}{1-\alpha-o(1)}.
    \end{align*}
   Thus,
    \begin{align*}|\mc L_n(P)| &\leq \Bigg( \frac{(1+\varepsilon)(1-\alpha)^2 n}{e^2} \Bigg)^{(1-\alpha)n^2} \!\!\!\! \cdot \exp\left(\frac{n \log^2{n}}{1-\alpha-o(1)}\right)\\
     &\leq \Bigg( \frac{(1-\alpha)^2 n}{e^2} \Bigg)^{(1-\alpha)n^2}\!\!\!\! \cdot\exp\left(\varepsilon (1-\alpha) n^2+\frac{n \log^2{n}}{1-\alpha-o(1)}\right)\\
     &= \Bigg( \frac{(1-\alpha)^2 n}{e^2} \Bigg)^{(1-\alpha)n^2}\!\!\!\! \cdot\exp\left(\frac{n \log^2{n}}{1-\alpha-o(1)}\right),\end{align*}
    recalling that $\varepsilon = n^{-1}\log{n}$ and so $\exp(\varepsilon (1-\alpha) n^2) = \exp(o(n \log^2{n})).$
\end{proof}

The next ingredient we will need is a lower bound on the total number of $n\times n$ Latin squares. The current state of the art is the following bound (obtained via a celebrated result of Egorychev--Falikman~\cite{Ego81,Fal81} on permanents of doubly stochastic matrices).
\begin{lemma}[{\cite[Theorem~17.2]{vLW01}}]
    \label{lem:lower-bound-Latin-squares}
    The number of $n \times n$ Latin squares is at least
    \[|\cL_n| \geq \Big( \frac{n}{e^2}  \Big)^{n^2}\cdot \exp(-O(n\log n)).\]
\end{lemma}

Finally, we need the following lemma, which shows that each triangle-typical ordered partial Latin square is roughly equally likely to be the outcome of the triangle removal process.
\begin{lemma}
    \label{lem:triangle-removal-process-prob-instance}
    Fix a constant $\alpha \in (0,1)$. Consider $\vec P \in \vec \cP^{\triangle:\varepsilon}_{n, \alpha n^2}$ for some $\varepsilon > 0$. Then for $\vec\bfR \sim \trp(n,m)$, we have
    \[ \mathbb{P}[\vec\bfR = \vec P] = (1 \pm 2\varepsilon)^{\alpha n^2} \exp\big( O(\log{n})\big) \Big(\frac{e}{n} \Big)^{3 \alpha n^2} (1-\alpha)^{3n^2 (1-\alpha)} . \]
\end{lemma}
\begin{proof}
Let $N=n^2$. Since $\vec P \in \vec\cP^{\triangle:\varepsilon}_{n,\alpha N}$,
\[ \mathbb{P}[\vec \bfR = \vec P] = \prod_{i=0}^{\alpha N - 1} \frac{1}{(1 \pm \varepsilon)(1-i/N)^3 n^3} = (1 \pm 2 \varepsilon)^{\alpha N}n^{-3 \alpha N} \exp \Bigg( -3 \sum_{i=0}^{\alpha N -1} \log\Big(1-\frac{i}{N}\Big) \Bigg).\]
Note that 
\[ \sum_{i=0}^{\alpha N -1} \frac{1}{N} \log\Big(1 - \frac{i+1}{N}\Big) \leq \int_{0}^\alpha \log(1-t) \mathrm{d}t \leq  \sum_{i=0}^{\alpha N -1} \frac{1}{N} \log\Big(1 - \frac{i}{N}\Big). \]
We have
\[ \sum_{i=0}^{\alpha N - 1} \Bigg( \log\Big(1 - \frac{i}{N}\Big) -\log\Big(1 - \frac{i+1}{N}\Big)   \Bigg) = \sum_{i=0}^{\alpha N - 1}  \log\Big( 1 + \frac{1}{N- (i+1)} \Big) \leq \sum_{i=0}^{\alpha N - 1}  \frac{1}{N- (i+1)} = O(\log{n}).\]
Thus, using the fact that $\int \log{x} = x(\log{x} - 1)$, we get
\begin{align*}
3 \sum_{i=0}^{\alpha N -1} \log\Big(1 - \frac{i}{N}\Big) &=   3N \int_{0}^\alpha \log(1-t) \,\mathrm{d}t \pm O(\log{n})\\
&=  3 N \int_{1-\alpha}^1 \log{s}\ \mathrm{d}s \pm O(\log{n})\\
&= 3 N \Big( 1(\log{1} - 1) - (1-\alpha) \big(\log(1-\alpha) - 1 \big) \Big) \pm O (\log{n})\\
&= 3N \big( - \alpha - (1-\alpha)\log(1-\alpha)\big) \pm O (\log{n}),\end{align*}
so
\[ \exp \Bigg(- 3 \sum_{i=0}^{\alpha N -1} \log\Big(1 - \frac{i}{N}\Big) \Bigg) = e^{3N \alpha} (1-\alpha)^{3N (1-\alpha)} \exp\big( O(\log{n})\big).\]
Therefore,
\[ \mathbb{P}[\vec\bfR = \vec P] = (1 \pm 2\varepsilon)^{\alpha N} \exp\big( O(\log{n})\big) \Big(\frac{e}{n} \Big)^{3 \alpha N} (1-\alpha)^{3N (1-\alpha)} .\qedhere\]
\end{proof}

We are now ready to prove \cref{lem:random_LS_to_triangle_removal}.
\begin{proof}[Proof of \cref{lem:random_LS_to_triangle_removal}]
Let $N:=n^2$ and $\varepsilon := n^{-1}\log{n}$ and $\gamma := n^{-1/4}$. By the fact that $\vec{\cU}$ is $(\rho,m)$-inherited from $\cT$ and by \cref{lem:almost-always-triangle-typical,lem:almost-always-quasirandom}, we have that for every Latin square $L \in \cT$, a random ordering $\vec \bfP_m(L)$ of a random set of $m$ hyperedges of $L$ satisfies
\[ \mathbb{P}[\vec \bfP_m(L) \in \vec{\cU} \cap \vec\cP^{\triangle:\varepsilon}_{n,\alpha N} \cap \cP^\gamma_{n,\alpha N}] \geq \rho - o(1),\]
where we say that an ordered partial Latin square is in $\cP^\gamma_{n,\alpha N}$ if its underlying unordered partial Latin square is in $\cP^\gamma_{n,\alpha N}$. Thus, letting $\cS := \vec{\cU} \cap \vec\cP^{\triangle:\varepsilon}_{n,\alpha N} \cap \cP^\gamma_{n,\alpha N}$,
\[ \mathbb{P}[\mathbf{L} \in \cT] = \sum_{L \in \cT} \frac{1}{|\cL_n|} \leq \sum_{L \in \cT} \frac{1}{|\cL_n|}\cdot\frac{\mathbb{P}[\vec\bfP_m(L) \in \cS]}{\rho - o(1)} = \frac{1}{\rho-o(1)}\cdot \frac{1}{|\cL_n|}\sum_{L \in \cT} \sum_{\vec P \in \cS} \mathbb{P}[\vec\bfP_m(L) = \vec P].\]
Changing the order of the two sums, we get
\[\mathbb{P}[\mathbf{L} \in \cT] \leq \frac{1}{\rho-o(1)}\sum_{\vec P \in \cS} \mathbb{P}[\vec{\bfP}_m(\bfL) = \vec P], \]
where $\vec{\bfP}_m(\bfL)$ is a uniformly random ordering of a uniformly random subset of $m$ hyperedges of $\bfL$. Letting $|\mc L_n(\vec P)|$ denote the number of completions of the underlying unordered partial Latin square of $\vec P$, we have
\[ \mathbb{P}[\vec{\bfP}_m(\bfL) = \vec P] = |\mc L_n(\vec P)|\cdot \frac{1}{|\cL_n|}\cdot \frac{(N - m)!}{N!},\]
since for each Latin square, there are $N (N-1) \dots (N - m + 1)$ ways to select an (ordered) sequence of $m$ hyperedges. Therefore, by \cref{thm:upper-bound-completions,lem:lower-bound-Latin-squares}, 
\begin{align*}
 \mathbb{P}[\mathbf{L} \in \cT]&\le \sum_{\vec P \in \cS} (\rho^{-1} + o(1))\cdot  \frac{|\mc L_n(\vec P)|}{ |\cL_n|}\cdot \frac{(N - m)!}{ N! }\\
&\leq \sum_{\vec P \in \cS} (\rho^{-1} + o(1))\cdot \Big( \frac{n(1-\alpha)^2}{e^2} \Big)^{N - |\vec P|} \exp\left(\frac{n \log^2 n}{1-\alpha-o(1)}\right) \cdot n^{O(n)} \Big( \frac{e^2}{n}\Big)^{n^2} \cdot \frac{(N - m)!}{ N! }\\
&\leq \sum_{\vec P \in \cS}  \exp\left(\frac{n \log^2 n}{1-\alpha-o(1)}\right) \cdot \Big( \frac{e^2}{n}\Big)^{\alpha N}\cdot (1-\alpha)^{2 N(1-\alpha)} \cdot \frac{(N - m)!}{ N! }.\end{align*}
By Stirling's approximation, we deduce
\begin{align*}\mathbb{P}[\mathbf{L} \in \cT] &\leq \sum_{\vec P \in \cS}  \exp\left(\frac{n \log^2 n}{1-\alpha-o(1)}\right) \cdot \Big( \frac{e^2}{n}\Big)^{\alpha N} \cdot (1-\alpha)^{2N(1-\alpha)} \cdot \frac{N^{N(1-\alpha)} (1-\alpha)^{N(1-\alpha)}}{e^{N(1-\alpha)}} \cdot \frac{e^N}{N^N}\\
&=\sum_{\vec P \in \cS}  \exp\left(\frac{n \log^2 n}{1-\alpha-o(1)}\right)  \Big( \frac{e}{n} \Big)^{3\alpha N} (1-\alpha)^{3N(1-\alpha)}.  \end{align*}
On the other hand, for each $\vec P \in \vec{\cU} \cap \vec\cP^{\triangle:\varepsilon}_{n,\alpha N} \cap \cP^{\gamma}_{n,\alpha N}$, by \cref{lem:triangle-removal-process-prob-instance} we have
\begin{align*} \mathbb{P}[\vec \bfR = \vec P] &\geq  (1 - 2\varepsilon)^{\alpha N} e^{- \Theta(\log{n})}\Big(\frac{e}{n} \Big)^{3 \alpha N} (1-\alpha)^{3N (1-\alpha)}  \\&\geq \exp\big(- (2+o(1)) \alpha n \log{n} \big)\Big(\frac{e}{n} \Big)^{3 \alpha N} (1-\alpha)^{3N (1-\alpha)} .\end{align*}
Therefore,
\[\mathbb{P}[\bfL \in \cT] \leq \sum_{\vec P \in \cS} \exp\left(\frac{n \log^2 n}{1-\alpha-o(1)}\right) \mathbb{P}[\vec \bfR = \vec P]\leq \exp\left(\frac{n \log^2 n}{1-\alpha-o(1)}\right) \mathbb{P}[\vec \bfR \in \vec{\cU}].\qedhere\]
\end{proof}
\section{Approximating the triangle removal process}\label{sec:approximation-2}
In this section, we collect some simple lemmas that allow us to analyse the triangle removal process via Erd\H os--R\'enyi (binomial) random hypergraphs. The latter are defined in terms of independent choices, and are therefore much easier to analyse.

We then use these lemmas to deduce a user-friendly corollary (\cref{cor:random-LS-to-binomial}) of \cref{lem:random_LS_to_triangle_removal}.

\begin{definition}
\label{def:cleaned_up_binomial_models}
Let $K_{n,n,n}^{(3)}$ be the complete 3-uniform 3-partite hypergraph on $n+n+n$ vertices, and let $\mb G^{(3)}(n,p)$ be the probability distribution on subgraphs of $K_{n,n,n}^{(3)}$ obtained by including each possible hyperedge with probability $p$ independently.

Note that $\bfG \sim \mb G^{(3)}(n,p)$ may not be a partial Latin square, as it may contain two hyperedges intersecting in more than one vertex. Call this a \emph{conflict}.
\end{definition}

Now, our first lemma is similar to \cite[Lemma~2.6]{ferber2020almost}. It approximates the triangle removal process with a binomial random hypergraph, for properties $\mc U$ that are \emph{monotone increasing} in the sense that $ P \in \cU$ and $ P \subseteq  P'$ then $P' \in \cU$ (i.e. ``adding additional hyperedges doesn't hurt''). For multiple-exposure arguments, it will be important that this lemma can handle the triangle removal process started from a proper subgraph of $K_{n,n,n}$.

\begin{definition}
    For a partial Latin square $P$, let $\trp(P, m)$ be the partial Latin square distribution obtained with $m$ steps of the triangle removal process starting from $G(P)$. Thus, an outcome of $\trp(P,m)$ always has $m$ hyperedges, and none of these hyperedges conflict with $P$ (unless we run out of triangles and get the outcome ``$\abort$'')
\end{definition}

\begin{lemma}
\label{lemma:triangle_removal_chunk_to_binomial}
Fix constants $\gamma,\alpha_1,\alpha_2$ such that $\gamma>0$, and $\alpha_1,\alpha_2\in [0,1]$ are sufficiently small with respect to $\gamma$. Define $m_1=\alpha_1 n^2$ and $m_2=\alpha_2 n^2$ and $p = (1 + \gamma)\alpha_2/n$. 
Let $P \in \cP_{n,m_1}$ be a partial Latin square, and consider a monotone increasing property $\cU$ of subhypergraphs of $K_{n,n,n}^{(3)}$.
Let $\bfR \sim \trp(P, m_2)$ and $\bfG \sim \mb G^{(3)}(n, p)$. Then 
\[ \mathbb{P}[\bfR \in \cU ] \leq \mathbb{P}[\bfG \in \cU] + \exp(-\Omega(n^2)).\]
\end{lemma}
\begin{proof}
Let $\bfG^*$ be obtained from $\bfG$ by deleting all hyperedges involved in conflicts. We say that the hyperedges of $\bfG^*$ are \emph{isolated} in $\bfG$. Note that we can couple $(\bfG^*, \bfG)$ and $\bfR$ such that $\bfR \subseteq \bfG$ as long as $\bfG^*$ has at least $m_2 = \alpha_2 n^2$ hyperedges (to see this, randomly order the hyperedges in $\bfG$ and run the triangle removal process on that ordering).

Let $\bfX$ be the number of hyperedges in $\bfG^*$ (i.e., the number of isolated hyperedges in $\bfG$). Note that each hyperedge in $P$ conflicts with no more than $3n$ of the $n^3$ potential hyperedges in $\bfG$. Thus, there are at least $(1-3\alpha_1)n^3$ potential hyperedges that do not conflict with $P$. Each of these is present and isolated in $\bfG$ with probability at least $p (1-p)^{3n}  \geq (1 + \gamma)\alpha_2 \exp(-6(1+\gamma)\alpha_2) /n$. Thus, we have 
\[ \mathbb{E}[ \bfX] \geq (1-3\alpha_1)n^3 \cdot (1+\gamma)\alpha_2 \exp(-6(1+\gamma)\alpha_2) /n  \geq (1+\gamma/2)\alpha_2 n^2 . \]
Adding a hyperedge to $\bfG$ can increase $\bfX$ by at most $1$, and removing a hyperedge from $\bfG$ can increase $\bfX$ by at most $3$, by making three hyperedges isolated. Thus, by \cref{freedman_inequality},
\[ \mathbb{P}[\bfX < \alpha_2 n^2] \leq \mathbb{P}[|\bfX - \mathbb{E}[\bfX]| \geq \gamma \alpha_2 n^2/2 ] \leq \exp(-\Omega(n^2)).\]
Therefore,
\[ \mathbb{P}[\bfR \in \cU] \leq \mathbb{P}[\bfG \in \cU] + \mathbb{P}[\bfX < \alpha_2 n^2] = \mathbb{P}[\bfG \in \cU] + \exp(-\Omega(n^2)).\qedhere \]
\end{proof}

The next lemma appears as \cite[Lemma~5.2]{KSSS22}, and handles properties $\mc U$ that are monotone decreasing in the sense that if $P \in \cU$ and $P' \subseteq P$ then $P' \in \cU$ (i.e., ``removing hyperedges doesn't hurt'').
\begin{lemma}[{\cite[Lemma~5.2]{KSSS22}}]
\label{lemma:triangle_removal_to_cleaned_binomial}
Fix a constant $\alpha \in (0,1)$. Let $\cU \subseteq \cP_{n}$ be a monotone decreasing property of partial Latin squares, let\footnote{We previously defined $\trp(n, m)$ as a distribution on ordered partial Latin squares, but here we are somewhat abusively viewing it as a distribution on unordered partial Latin squares (we simply ignore the ordering).} $\bfR \sim \trp(n, \alpha n^2)$, let $\bfG \sim \mb G^{(3)}(n,\alpha/n)$, and let $\mbf G^*$ be obtained from $\mbf G$ by deleting all hyperedges involved in conflicts. Then
\[ \mathbb{P}[ \bfR \in \cU \text{ or } \bfR = \abort] = O (\mathbb{P}[\bfG^* \in \cU]).\]
\end{lemma}
For completeness, we reproduce the simple proof of \cref{lemma:triangle_removal_to_cleaned_binomial} (from \cite{KSSS22}).
\begin{proof}
We can couple $\bfR$ and $\bfG^*$ such that, if $e(\bfG) \leq \alpha n^2$, then either $\bfG^* \subseteq \bfR$ or $\bfR = \abort$ (to see this, randomly order the hyperedges of $\bfG$ and run the triangle removal process on that ordering).

Since $\cU$ is monotone decreasing, we have
\[ \Pr[\bfR \in \cU \text{ or } \bfR = \abort] \leq \Pr[\bfG^* \in \cU \,|\, e(\bfG) \leq \alpha n^2].\]
Since $e(\bfG)$ is binomial with mean $\alpha n^2$, we have that $\Pr[e(\bfG) \leq \alpha n^2] = \Omega(1)$. Thus, 
\[ \Pr[\bfR \in \cU \text{ or } \bfR = \abort] \leq \Pr[\bfG^* \in \cU] / 
\Pr[e(\bfG) \leq \alpha n^2] = O(\Pr[\bfG^* \in \cU]). \qedhere\]
\end{proof}

To illustrate the tools from this section, we conclude this section by stating a corollary of \cref{lem:random_LS_to_triangle_removal}, which is an easy-to-apply tool for studying random Latin squares. We do not actually use this corollary in the present paper (due to some technical issues related to multiple-exposure arguments, we need the full power of \cref{lem:random_LS_to_triangle_removal}), but we hope that it may be convenient in future work on random Latin squares.

Note that we have defined $(\rho,m)$-inheritedness (in \cref{def:inherited}) only for properties of ordered partial Latin squares, but in what follows we extend this definition to unordered partial Latin squares in the natural way.

\begin{corollary}
\label{cor:random-LS-to-binomial}
Fix constants $\alpha \in (0,1)$ and $\rho \in (0,1)$, and let $m = \alpha n^2$. Let $\cT$ be a property of $n\times n$ Latin squares and let $\cU$ be a property of subhypergraphs of $K^{(3)}_{n,n,n}$, such that $\cU$ is \emph{$(\rho,m)$-inherited} from $\cT$. Let $\bfL$ be a uniformly random $n\times n$ Latin square.
\begin{enumerate}
    \item If $\cU$ is a \emph{monotone decreasing} property, we have \[ \mathbb{P}[\mathbf{L} \in \cT] \leq \exp\big(O(n \log^2 {n})\big)  \mathbb{P}[\mathbf{G^*} \in \cU],\]
    where $\bfG \sim \mb G^{(3)}(n,\alpha/n)$, and $\bfG^*$ is obtained from $\bfG$ by deleting all hyperedges involved in conflicts.
    \item If $\cU$ is a \emph{monotone increasing} property, then for any constant $\gamma > 0$, if $\alpha$ is sufficiently small with respect to $\gamma$, we have \[ \mathbb{P}[\mathbf{L} \in \cT] \leq \exp\big(O(n \log^2 {n})\big)  \mathbb{P}[\mathbf{G} \in \cU] + \exp(-\Omega(n^2)),\]
    where $\bfG \sim \mb G^{(3)}(n,(1+\gamma)\alpha/n)$.
\end{enumerate}
\end{corollary}
Part (1) is obtained by combining \cref{lem:random_LS_to_triangle_removal} with \cref{lemma:triangle_removal_chunk_to_binomial} (with $P=\emptyset$ and $\alpha_1=0$ and $\alpha_2=\alpha$), and Part (2) is obtained by combining \cref{lem:random_LS_to_triangle_removal} with \cref{lemma:triangle_removal_to_cleaned_binomial}.

\section{The master theorem for parity distributions}\label{sec:master}

All parts of \cref{thm:shiny} will be deduced from a master theorem which describes the distribution of $(N_{\mr{row}}(\mbf L),N_{\mr{col}}(\mbf L),N_{\mr{sym}}(\mbf L))$ in somewhat technical terms. To state this theorem, we need some terminology to describe certain types of distributions that arise naturally from random intercalate switches.
\begin{definition}\label{def:near-binomial}
Consider a random vector $(c_{1}+\mathbf{B}_{1},c_{2}+\mathbf{B}_{2},c_{3}+\mathbf{B}_{3})$,
where $\mathbf{B}_{1},\mathbf{B}_{2},\mathbf{B}_{3}$ are independent
binomial random variables satisfying $\mathbf{B}_{i}\sim\operatorname{Bin}(n_{i},1/2)$. We say that this random vector is \emph{$(n,d)$-near-binomial} if $n-d\le n_{i}\le n$ and $c_{i}\le d$ for each $i$.
\end{definition}

The reader should think of $d$ as a ``defect parameter'': if $d$ is small, then $(n,d)$-near-binomial random vectors are approximately
distributed like three independent copies of $\operatorname{Bin}(n,1/2)$.
We remark that if we want this approximation to have $o(1)$ total variation error, then we need $d=o(\sqrt{n})$. Indeed, for each $i$
we need $\mb E[\mathbf{X}_{i}]-n/2$ to be negligible compared to the fluctuations
of $\operatorname{Bin}(n,1/2)$ (which have order of magnitude $\sqrt{n}$).
\begin{definition}\label{def:2-near-binomial}
Say that a random vector $(\mathbf{X}_{1},\mathbf{X}_{2},\mathbf{X}_{3})\in\mathbb{N}^{3}$
is \emph{$(n,d)$-2-near-binomial} if it can be obtained from an $(n,d)$-near-binomial
random vector by conditioning on a particular parity (even or odd) for each of $\mathbf{X}_{1},\mathbf{X}_{2},\mathbf{X}_{3}$.
\end{definition}

\begin{definition}\label{def:2-near-binomial-mixture}
Say that a random vector $\vec{\mathbf{X}}$ is a \emph{2-near-binomial
mixture} with parameters $(n,d,\varepsilon)$ if one can define
an auxiliary random object $\mathbf{Q}$ (on the same probability
space as $\vec{\mathbf{X}}$), and an event $\mathcal{A}$ depending
only on $\mathbf{Q}$, such that:
\begin{itemize}
\item $\Pr[\mathcal{A}]\ge1-\varepsilon$, and
\item For any of the possible outcomes $Q$ of $\mathbf{Q}$ satisfying
$\mathcal{A}$, the conditional distribution of $\vec{\mathbf{X}}$
given $\mathbf{Q}=Q$ is $(n,d)$-2-near-binomial.
\end{itemize}
\end{definition}
If $d$ is small, then the above definition says that $\vec{\mathbf{X}}$
is closely approximated by a sequence of three independent $\operatorname{Bin}(n,1/2)$
random variables, except for possible biases mod 2. Note that the distribution $\mu^*$ featuring in \cref{thm:shiny} is a 2-near-binomial mixture with parameters $(n,0,0)$.

Now, our main technical theorem is as follows.
\begin{theorem}\label{thm:main-technical}
Let $\mathbf{L}$ be a uniformly random $n\times n$ Latin square. Then $(N_{\mathrm{row}}(\mathbf{L}),N_{\mathrm{col}}(\mathbf{L}),N_{\mathrm{sym}}(\mathbf{L}))$ is a 2-near-binomial mixture with parameters
\[
\big(n,\;O(\log^{11} n),\;\exp(-\omega(n\log^2 n))\big).
\]
\end{theorem}

\subsection{Corollaries of the master theorem}
In this subsection we briefly explain how to deduce all parts of \cref{thm:shiny} from our master theorem (\cref{thm:main-technical}). This is all quite routine. Note that \cref{thm:baby} is an immediate consequence of \cref{thm:shiny}(4).

We start with basic consequences of the de Moivre--Laplace local central limit theorem and a large deviation principle for binomial distributions.
\begin{lemma}\label{lem:deM-Laplace}
Let $\vec{\mathbf{X}}=(\mathbf{X}_{1},\mathbf{X}_{2},\mathbf{X}_{3})$
be a $(n,o(\sqrt n))$-2-near-binomial random vector.
\begin{enumerate}
    \item[(i)] Recall that $(\mathbf{X}_{1}\text{ mod 2},\,\mathbf{X}_{2}\text{ mod 2},\,\mathbf{X}_{3}\text{ mod 2})$
always takes a common value $\vec{v}\in\{0,1\}^{3}$. For any $\vec{x}\in\mathbb{Z}^{3}$
satisfying $\vec{x}=\vec{v}$ mod 2, we have
\[
\Pr[\vec{\mathbf{X}}=\vec{x}]=8\cdot\frac{1}{(2\pi(n/4))^{3/2}}\exp\left(-\frac{(x_{1}-n/2)^{2}+(x_{2}-n/2)^{2}+(x_{3}-n/2)^{2}}{2(n/4)}\right)+o(n^{-3/2}).
\]
\item [(ii)] Let $H_{2}:\alpha\mapsto-\alpha\log_2\alpha-(1-\alpha)\log_2(1-\alpha)$
be the base-2 binary entropy function, and let $I(x_{1},x_{2},x_{3})=3-H_{2}(x_{1})-H_{2}(x_{2})-H_{2}(x_{3})$.
We have
\[
-\inf_{\vec{x}\in E^{\circ}}I(\vec{x})\le\liminf n^{-1}\log_2\Pr[n^{-1}\vec{\mathbf{X}}\in E]\le\limsup n^{-1}\log_2\Pr[n^{-1}\vec{\mathbf{X}}\in E]\le-\inf_{\vec{x}\in\overline{E}}I(\vec{x})
\]
for all Borel $E\subseteq\mb R^{3}$.
\end{enumerate}
\end{lemma}

\begin{proof}
Recall that $\vec{\mathbf{X}}$ is obtained from a $(n,d)$-near-binomial
random vector $\vec{\mathbf{Y}}$ by conditioning on the event
that $\vec{\mathbf{Y}}=\vec{v}$ mod 2 (which occurs with probability exactly\footnote{Here we are using that $\on{Bin}(n,1/2)$ is equally likely to be even and odd, which is easy to prove e.g.\ by induction.} $1/8$). 
Now, the de Moivre--Laplace
local central limit theorem (see e.g. \cite[Chapter~VII, Section~3, Theorem~1]{Fel68}) tells us that
\[
\Pr[\vec{\mathbf{Y}}=\vec{x}]=
\prod_{i=1}^3\left(\frac{1}{\sqrt{2\pi(n_i/4)}}\exp\left( -\frac{(x_{i}'-n_i/2)^{2}}{2(n_i/4)}\right)+o(n^{-1/2})\right)
\]
for some $n_1,n_2,n_3=n-o(\sqrt n)$, and some $x_1',x_2',x_3'$ satisfying  $x'_i = x_i - o(\sqrt{n})$. 
It is easy to see that the error terms in the $n_i$ and $x_i'$ can be absorbed into the main error term (i.e., it makes no difference to replace $n_1,n_2,n_3$ with $n$, and replace each $x_i'$ with $x_i$). The desired result follows, noting that conditioning on an event that occurs with probability $1/8$ introduces
a factor of $8$.

We can similarly deduce (ii) from a large deviation principle for binomial distributions (e.g. Sanov's theorem, see \cite[Theorem~2.1.10]{DZ10}\footnote{A simple approximation argument (approximating general open and closed sets with product sets) can be used to combine large deviation principles for three independent random variables, see e.g.\ \cite[Lemma~2.8]{LS87}.}). Actually, this is even simpler than (i), because in this case the conditioning mod 2 has a negligible impact. 
\end{proof}
Now we are ready to deduce \cref{thm:shiny}.
\begin{proof}[Proof of \cref{thm:shiny}]
We start with \cref{thm:shiny}(3) (the local central limit theorem). Consider
any $\vec{x}\in\mathbb{Z}^{3}$ satisfying $x_{1}+x_{2}+x_{3}=f(n)$
(mod 2). By \cref{thm:main-technical}, together with \cref{lem:deM-Laplace}(i) and the definition of a 2-near-binomial mixture, $\Pr[\vec{\mathbf{X}}=\vec{x}]$ is of the form 
\begin{align*}
&\Pr[\vec{\mathbf{X}}=\vec{x}\text{ mod 2}]\cdot \left( 8\cdot\frac{1}{(2\pi(n/4))^{3/2}}\exp\left(\frac{(x_{1}-n/2)^{2}+(x_{2}-n/2)^{2}+(x_{3}-n/2)^{2}}{2(n/4)}\right)+o(n^{-3/2})\right)\\
&\qquad\qquad+\exp(-\omega(n\log^2 n)).
\end{align*}
By the Cavenagh--Wanless theorem~\cite{CW16} (mentioned in the introduction), we
have $\Pr[\vec{\mathbf{X}}=\vec{x}\text{ mod 2}]=1/4+o(1)$, and \cref{thm:shiny}(3) follows.

We can then deduce \cref{thm:shiny}(4) (total variation convergence)
by summing over all $\vec{x}$ (using e.g.\ Chebyshev's inequality to control the tails). Indeed, let $\vec{\mathbf{Z}}\sim\mu^{*}$. For any $N$, the variance of $\on{Bin}(N,1/2)$ is $N/4=O(N)$, so (recalling the definition of a 2-near-binomial mixture and noting that the means of all relevant binomials are within $O((\log{n})^{11}) = O(\sqrt{n})$ of each other), the variance of each component of $\vec{\mbf X}$ and $\vec{\mbf Z}$ is $O(n)$. Then, for any $K\ge1$, taking $\vec{\mathbf{Z}}\sim\mu^{*}$ and writing $(n/2)\vec 1=(n/2,n/2,n/2)$,
we have
\begin{align*}
{\operatorname d}_{\mathrm{TV}}(\vec{\mathbf{X}},\mu^{*}) & =\frac{1}{2}\sum_{\vec{x}\in\mathbb{Z}^{3}}\big|\Pr[\vec{\mathbf{X}}=\vec{x}]-\Pr[\vec{\mathbf{Z}}=\vec{x}]\big|\\
&\le O((K\sqrt{n})^{3})\cdot o(n^{-3/2})+\Pr\big[\|\vec{\mathbf{X}}-(n/2)\vec 1\|_{\infty}\ge K\sqrt{n}\big]+\Pr\big[\|\vec{\mathbf{Z}}-(n/2)\vec 1\|_{\infty}\ge K\sqrt{n}\big]\\ 
 & \le o(1)\cdot K^{3}+O(1/K^{2}).
\end{align*}
Taking $K\to\infty$ sufficiently slowly, \cref{thm:shiny}(4) follows. 

It is then a near-trivial matter to deduce \cref{thm:shiny}(1--2) (the central limit theorem
and law of large numbers), since total variation convergence is much stronger than convergence in distribution (see e.g.\ \cite[Proposition~1.2]{Ros11}). 

Finally, \cref{thm:shiny}(5) (the large deviation principle) follows from \cref{thm:main-technical} and \cref{lem:deM-Laplace}(ii). Indeed, \cref{thm:main-technical} says that $\vec {\mbf X}$ is a mixture of $(n,o(\sqrt n))$-2-near-binomial distributions except with an exceptional probability $\exp(-\omega(n\log^2 n))$ (which is negligible). 
\end{proof}

\begin{remark}
    In the above proof, it was convenient to use the Cavenagh--Wanless theorem as a black box. However, we remark that given the machinery we developed in \cref{sec:approximation-1,sec:approximation-2}, it would be a very simple matter to re-prove the Cavenagh--Wanless theorem. Indeed, using this machinery it is easy to prove that $\mbf L$ is very likely to have a $2\times 3$ Latin subrectangle (in the first half of the rows, columns and symbols of $\mbf L$, say), and a $3\times 2$ Latin subrectangle (in the second half of the rows, columns and symbols of $\mbf L$, say), and switching on these two Latin subrectangles mixes between all four different possibilities for $(N_{\mr{row}}(\mbf L),N_{\mr{col}}(\mbf L),N_{\mr{sym}}(\mbf L))$ mod 2.
\end{remark}
\section{Stable intercalates}\label{sec:stable}

An \emph{intercalate} in a Latin square is a $2 \times 2$ Latin subsquare. To \emph{switch} an intercalate with symbols $a$ and $b$ means to replace both occurrences of $a$ with $b$ and vice versa. The result of this operation is always another Latin square.

As discussed in \cref{sec:outline}, we will prove \cref{thm:main-technical} via random intercalate switches. It is very important that if we start with a uniformly random Latin square, and perform our random switches, the resulting Latin square is still uniformly random. To ensure this, we need to restrict our attention to intercalates with a certain \emph{stability} property. We define this property in such a way that switching stable intercalates can never create a new intercalate or destroy an existing intercalate (and, in particular, if we only ever switch stable intercalates, then the set of available stable intercalate switches never changes).

Specifically, we define an intercalate to be stable if it does not intersect other intercalates, and if switching it and/or some other intercalates cannot make it intersect another intercalate, as follows.

\begin{definition}\label{def:stable_intercalates}
Let $P$ be a partial Latin square. 
For an intercalate $A$, we denote by $\bar{A}$ the outcome of switching $A$.
\begin{itemize}
    \item We say that an intercalate in $P$ is \emph{isolated} 
    if it does not share an entry with any other intercalate.
    \item Consider a set of isolated intercalates $\{A_1, \dots, A_t\}$ in $P$. We say that this is a \emph{critical} set of intercalates if it is possible to switch a subset of $A_1, \dots, A_t$ (resulting in intercalates $B_1, \dots, B_t$ where $B_i = A_i$ or $B_i = \bar{A_i}$ for each $i\in [t]$) to create a new intercalate that shares an entry with each of $B_1, \dots, B_t$.
    \item We say that an isolated intercalate in $P$ is \emph{stable} if it is not contained in a critical set of intercalates.
\end{itemize}
\end{definition}

Observe that a critical set $\{A_1, \dots, A_t\}$ may be a subset of another critical set $\{A_1, \dots, A_{t+1}\}$ (e.g., if switching a subset of $A_1, \dots, A_t$ creates a new intercalate that intersects $A_{t+1}$). Any critical set $\mc D$ has size at most $4$, since its members $A_1, \dots, A_t$ are pairwise entry-disjoint, and for each $A_i$, either $A_i$ or $\bar{A}_i$ shares an entry with the new intercalate that can be produced by switching a subset of $\mc D$.

Now, in the following lemma we record the fact that switching stable intercalates leaves the set of stable intercalates switches unchanged.

\begin{definition}\label{def:intercalate-abstraction}
    For an intercalate $A$ with rows $r_1,r_2$, columns $c_1,c_2$, and symbols $s_1,s_2$, define the tuple $\sigma(A) = (\{r_1,r_2\}, \{c_1,c_2\}, \{s_1,s_2\})$ (i.e., this records the rows, columns and symbols of the intercalate, without recording which of the two possible ``states'' it is in). For a (partial) Latin square $P$, let $\cS(P)$ be the set of all tuples $\sigma(A)$ for all stable intercalates $A$ in $P$ (i.e., this records the intercalate switches we can perform, without recording which state each of those intercalates is in). 
\end{definition}

\begin{lemma}
\label{lem:canonicity_of_stable_intercalates}
Let $P_1$ be a Latin square, and suppose that we can switch a single stable intercalate in $P_1$ to obtain a second Latin square $P_2$. Then $\cS(P_1) = \cS(P_2)$.
\end{lemma}

\begin{proof}
Let $A$ be the stable intercalate in $P_1$ which is switched to obtain $P_2$. Then $\sigma(A) \in \cS(P_1)$. We start with a simple observation which will be used multiple times.
\begin{claim}\label{claim:stay-isolated}
    If an intercalate $B \neq \bar{A}$ is isolated in $P_2$, it is also present and isolated in $P_1$.
\end{claim}
\begin{claimproof}
    Since $B$ shares no entry with $\bar{A}$ in $P_2$, we see that $B$ is also present in $P_1$. Furthermore, if there were any intercalate $C$ in $P_1$ that shares an entry with $B$, that intercalate $C$ would not be present in $P_2$, so it would have to share an entry with $A$ in $P_1$. But that would imply $A$ is not isolated in $P_1$, which would contradict our assumption that $A$ is stable in $P_1$.
\end{claimproof}

\medskip\noindent\textit{Step 1: reversibility of the $A$ switch.} We will first show $\sigma(A) \in \cS(P_2)$, i.e., we could equally well have switched $\bar A$ in $P_2$ to obtain $P_1$. This amounts to showing that $\bar{A}$ is isolated in $P_2$, and not contained in a critical set of intercalates.

First, note that if $\bar{A}$ were not isolated in $P_2$, then there would be some intercalate $B$ in $P_2$ that shares an entry with $\bar{A}$. Since that entry would not be in $P_1$, we would have that $B$ is not in $P_1$, implying that $\{A\}$ is a critical set in $P_1$, contradicting the stableness of $A$ in $P_1$. Thus, $\bar{A}$ is isolated in $P_2$.

Now, suppose for contradiction that $\bar{A}$ is contained in a critical set $\{\bar{A}, A_2, \dots, A_t\}$ of isolated intercalates in $P_2$. If $t=1$, this would contradict the isolatedness of either $A$ in $P_1$ or $\bar{A}$ in $P_2$; thus $t \geq 2$. Since the intercalates $A_2, \dots, A_t$ are isolated in $P_2$, they are also present and isolated in $P_1$, by \cref{claim:stay-isolated}. But then $\{A, A_2, \dots, A_t\}$ is a critical set in $P_1$: if switching $S \subseteq \{\bar{A}, A_2, \dots, A_t\}$ creates a new intercalate in $P_2$, then switching $\bar{S} = S \cup \{A\}$ (if $\bar{A} \notin S$) or $\bar{S} = S \setminus \{\bar{A}\}$ (if $\bar{A}\in S$) creates a new intercalate in $P_1$. 
This contradicts the stableness of $A$ in $P_1$. We conclude that $\bar{A}$ is stable in $P_2$, so $\sigma(A) \in \cS(P_2)$.

\medskip\noindent\textit{Step 2: the intercalates other than $A$.}
Since $\sigma(A) \in \cS(P_1)\cap \cS(P_2)$, by symmetry it is enough to show that for every intercalate $B$ that is stable in $P_1$, we have $\sigma(B) \in \cS(P_2)$. Let $B \neq A$ be a stable intercalate in $P_1$. Then $B$ must be present and isolated in $P_2$, by \cref{claim:stay-isolated} and the symmetry between $P_1$ and $P_2$.

Suppose for the purpose of contradiction that $B$ is not stable in $P_2$. Then $B$ is contained in a critical set $\mc Z=\{B, A_2, \dots, A_t\}$ of isolated intercalates in $P_2$. Since the intercalates in $\{A_2, \dots, A_t\}\setminus \{\bar{A}\}$ are isolated in $P_2$ and distinct from $\bar{A}$, they are also present and isolated in $P_1$. 
Now, either $\mc Z \setminus \{\bar{A}\} \cup \{A\}$ or $\mc Z \setminus \{\bar{A}\}$ is a critical set in $P_1$. In particular, if the new intercalate that can be created in $P_2$ by switching a subset of $\mc Z$ shares an entry with either $A$ or $\bar{A}$, then $\mc Z \setminus \{\bar{A}\} \cup \{A\}$ is a critical set in $P_1$. Otherwise, $\mc Z \setminus \{\bar{A}\}$ is a critical set in $P_1$. In either case, $B$ is contained in a critical set in $P_1$, contradicting that $B$ is stable in $P_1$. Thus, $B$ must be stable in $P_2$, so $\sigma(B) \in \cS(P_2)$, concluding the proof.
\end{proof}

\section{Deducing the master theorem from stable intercalate information}\label{sec:master-from-stable}
As discussed in \cref{sec:outline}, the main challenge in this paper is to show that random Latin squares typically have very rich constellations of stable intercalates. These intercalates can then be randomly switched to prove our master theorem on parity distributions (\cref{thm:main-technical}). In this section, we state our key lemma about stable intercalates in random Latin squares (\cref{lemma:existence_split_stable_intercalate}), and show how to use this key lemma (together with some linear-algebraic arguments) to prove \cref{thm:main-technical}. After this section, the rest of the paper will be devoted to the proof of \cref{lemma:existence_split_stable_intercalate}.

\subsection{The key lemma on stable intercalates}
Roughly speaking, we need the property that for all ``reasonably large'' sets of rows, columns and symbols, we can find stable intercalates compatible with those choices. This will imply that the stable intercalates are so well-spread throughout the rows, columns and symbols that randomly switching them will thoroughly mix up the row, column and symbol parities.

\begin{definition}
\label{def:intercalate-expander}Let $P$ be an $n \times n$ (partial) Latin square, and consider some $1\le \ell\le n$. We say that $P$ is an \emph{$(\ell,\beta)$-intercalate-expander} if the following property holds.
For any sets of rows $\sR, \sR^*$, any sets of columns $\sC, \sC^*$, and any sets of symbols $\sS, \sS^*$, such that five out of these six sets have size $\beta n$ and the last one has size $\ell$, there is a stable intercalate in $P$ with one row in $\sR$, the other row in $\sR^*$, one column in $\sC$, the other column in $\sC^*$, one symbol in $\sS$, and the other symbol in $\sS^*$.
\end{definition}
Note that the sets of rows, columns, and symbols above are not necessarily disjoint. Thus, the stable intercalate promised by \cref{def:intercalate-expander} may for instance have both its rows in $\sR \cap \sR^*$.

Random Latin squares themselves are unlikely to be good intercalate-expanders, because they have too many intercalates (that are too likely to intersect each other and form critical sets of intercalates). We are able to overcome this with a \emph{sparsification} technique: instead of considering all the stable intercalates in a Latin square $L$, we consider only the stable intercalates in some sparsified \emph{partial} Latin square (and we make sure to use the same ``sparsification template'' $T$ for all Latin squares). To formalise this, we need another definition.

\begin{definition}\label{def:template}
An $n\times n$ \emph{template} is a subset $T\subseteq [n]^2$, which we interpret as a set of row/column pairs. For a Latin square $L\in \mc L_n$, we write $T \cap L\in \mc P_n$ for the partial Latin square containing just the entries of $L$ in positions specified by $T$.
\end{definition}

Now, our key lemma is as follows.
\begin{lemma}
\label{lemma:existence_split_stable_intercalate}Let $\beta >0$ be a constant. There exists a template $T \subseteq [n]^2$ such that the following holds. Let $\bfL \sim \mr{Unif}(\mc L_n)$ be a uniformly random Latin square (interpreted as a subgraph of $K^{(3)}_{n,n,n}$) and let $\ell = \log^{11}n$. Then
\[ \mathbb{P}[ T\cap \bfL \text{ is an }(\ell,\beta)\text{-intercalate-expander}\,] \geq 1- \exp(-\omega(n \log^2{n})).\]
\end{lemma}

\subsection{Deducing the master theorem on parity distributions}
To deduce \cref{thm:main-technical} from \cref{lemma:existence_split_stable_intercalate}, we will need some more notation and preliminaries.
\begin{definition}
Consider an $n\times n$ Latin square $L$. For a set of rows $R$
and a set of intercalates $\mathcal{I}$ in $L$, let $M_{R,\mathcal{I}}\in\{0,1\}^{|R|\times|\cI|}$
be the ``incidence matrix'' defined as follows:
\[
M_{R,\mathcal{I}}(r,I)=\begin{cases}
1 & \text{if row }r\text{ participates in intercalate }I\\
0 & \text{otherwise}.
\end{cases}
\]
We can similarly define $M_{C,\mc I}$ or $M_{S,\mc I}$ for a set of columns $C$ or a set
of symbols $S$.
\end{definition}

\begin{lemma}\label{lem:RCS}
Consider an $n\times n$ partial Latin square $P$ that is an $(\ell,1/10)$-intercalate-expander, where $n$ is sufficiently large and $1 \leq \ell = o(\sqrt n)$. Then we can find sets $R,C,S$ of rows, columns and symbols, satisfying $|R|,|C|,|S|\ge n-\ell$,
with the following property.

If we let $\mathcal{I}(R,C,S)$ be the set of
all stable intercalates in $P$ which are contained in the rows in $R$ and the columns in
$C$, only using the symbols in $S$, then the matrix
\[
\begin{pmatrix}M_{R,\mathcal{I}(R,C,S)}\\M_{C,\mathcal{I}(R,C,S)}\\M_{S,\mathcal{I}(R,C,S)}\end{pmatrix}\in\{0,1\}^{(|R|+|C|+|S|)\times|\mathcal{I}(R,C,S)|}
\]
has rank exactly $|R|+|C|+|S|-3$ over $\mb F_2$. The left kernel vectors of this matrix are those vectors of the form $(\vec{x}_R\; \vec{x}_C\; \vec{x}_S)\in \mb F_2^{|R|+|C|+|S|}$ where $\vec{x}_Z\in \mb F_2^{|Z|}$ is the all-zero vector $(0,\dots,0)$ or the all-one vector $(1,\dots,1)$ for each $Z \in \{R,C,S\}$.
\end{lemma}

\begin{proof}
First, we translate the desired statement into more combinatorial language. Consider the auxiliary 6-uniform hypergraph $\mc Q(R,C,S)$ on the vertex set\footnote{Here we use ``$\sqcup$'' to indicate a \emph{disjoint} union.} $R\sqcup C\sqcup S$, where for every intercalate in $\mc I(R,C,S)$ we put a hyperedge consisting of the two rows, two columns and two symbols of that intercalate. We want to choose $R,C,S$ to have the property that for any choices of vertex subsets $R'\subseteq R$ and $C'\subseteq C$ and $S'\subseteq S$ which are not all ``trivial'' (i.e., at least one of these three subsets is nonempty and proper), there is some hyperedge of $\mc Q(R,C,S)$ which intersects $R'\sqcup C'\sqcup S'$ in an odd number of vertices.

Now, we specify $R,C,S$ as follows. Let $R$ be the set of all rows $r$ for which there exist at least $6\ell$ different stable intercalates in $P$ which all involve $r$ and which otherwise involve pairwise disjoint rows, columns and symbols. Define $C,S$ similarly.

\begin{claim*}
Each of $R,C,S$ has size greater than $n-\ell$.
\end{claim*}
\begin{claimproof}
    Suppose for contradiction (and without loss of generality) that there is a set $R_0$ of $\ell$ rows which are not in $R$. Then we can use the intercalate-expansion property to greedily find $n/10$ stable intercalates in $P$ which all involve disjoint rows, columns and symbols (except for their rows in $R_0$). So, there must be some $r\in R_0$ involved in $(n/10)/|R_0|\ge 6\ell$ stable intercalates which involve disjoint rows, columns and symbols (except that they all involve $r$). This contradicts the choice of $R_0$.
\end{claimproof}
Recall that the defining properties of $R,C,S$ were that they are included in many (at least $6\ell$) ``externally disjoint'' stable intercalates in $P$. The above claim shows that almost all of those intercalates lie in $\mc I(R,C,S)$,
and therefore correspond to hyperedges of $\mc Q(R,C,S)$. Specifically, each ``bad'' row/column/symbol (which is not in $R,C,S$) kills at most one of our externally disjoint stable intercalates, so each vertex of $\mc Q(R,C,S)$ is contained in more than $6\ell-3\ell\ge 3\ell$ externally disjoint 6-uniform hyperedges.

Now, consider vertex subsets $R'\subseteq R$ and $C'\subseteq C$ and $S'\subseteq S$ which are not all trivial (at least one of these subsets is nonempty and proper). We need to prove there is some hyperedge intersecting $R'\sqcup C'\sqcup S'$ in an odd number of vertices.

\medskip
\noindent\textbf{Case 1: }First, suppose that for each $Z \in \{R', C', S'\}$, either $Z$ or its complement has size less than $\ell$. This gives us three ``small sets'' of size less than $\ell$. At least one of our small sets must be non-empty; let $v$ be a vertex in that set. Recall that $v$ is contained in at least $3\ell$ externally disjoint hyperedges of $\mc Q(R,C,S)$, so at least one of these hyperedges (externally) avoids our three small sets. Any such hyperedge has odd intersection with $R\sqcup C\sqcup S$.

\medskip
\noindent\textbf{Case 2: }If Case 1 does not occur, we may assume without loss of generality that $R'$ and $R\setminus R'$ both have size at least $\ell$. Then we apply the intercalate-expansion property directly. We take the smaller of $R'$ and $R\setminus R'$ as $\mathscr R^*$ and the larger as $\mathscr R$; we take the larger of $C'$ and $C\setminus C'$ as $\mathscr C = \sC^*$; and we take the larger of $S'$ and $S\setminus S'$ as $\mathscr S = \sS^*$. This again gives us a stable intercalate in $P$ (corresponding to an edge of $\mc Q(R,C,S)$) with the desired odd intersection.
\end{proof}

Now we can prove \cref{thm:main-technical}.
\begin{proof}[Proof of \cref{thm:main-technical}]
Let $T$ be as given by \cref{lemma:existence_split_stable_intercalate}, and let $\ell = \log^{11}{n}$. For a Latin square $L\in \mc L_n$, we say an intercalate in $L$ is \emph{$T$-stable} if it is present and stable in $T\cap L$. Recall the notation $\mc S(P)$ from \cref{def:intercalate-abstraction} (recording the set of stable intercalates in a partial Latin square $P$, but without recording which of the two possible ``states'' each intercalate is in).

We consider an auxiliary multigraph $\mc G$, whose vertices are the Latin squares $L\in \mc L_n$, and where there is an edge between two Latin squares $L_1,L_2$ if it is possible to switch some of the $T$-stable intercalates in $L_1$ to transform it into $L_2$. 
So, for every $L\in \mc L_n$, the degree of $L$ in $\mc G$ is precisely $2^{|\mc S(T\cap L)|}$ (including one loop edge from $L$ to itself). In fact, by \cref{lem:canonicity_of_stable_intercalates}, the connected component of $L$ is a clique (with loops) of order $2^{|\mc S(T\cap L)|}$, where each vertex $L'\in \mc L_n$ in this clique has $\mc S(T\cap L')=\mc S(T\cap L)$. In particular, each component is regular, so the uniform distribution $\on{Unif}(\mc L_n)$ is a stationary distribution for the random walk on $\mc G$.

Therefore, if we first sample $\mathbf{L}\sim \on{Unif}(\mc L_n)$, and then we ``re-randomise'' $\mathbf{L}$
by randomly switching each $T$-stable intercalate with probability $1/2$ independently (this is the same as stepping to a random neighbour of $\mathbf L$ in $\mc G$),
then the resulting random Latin square $\mathbf{L}'$ still has
the same uniform distribution as $\mathbf{L}$.

Now, for the rest of the proof, our goal is to show that $\vec{\mathbf{X}}'=(N_{\mathrm{row}}(\mathbf{L}'),N_{\mathrm{col}}(\mathbf{L}'),N_{\mathrm{sym}}(\mathbf{L}'))$ is a 2-near-binomial mixture with parameters $(n, \ell, \exp(-\omega(n\log^2{n}))$. Recalling \cref{def:2-near-binomial-mixture}, this means we need to show how to reveal certain information about $\bfL$ and $\bfL'$, in such a way that, conditional on a typical outcome of this revealed information, $\vec{\mathbf{X}}'$ is a $(n, \ell)$-2-near-binomial random vector (where here ``typical'' refers to an event which occurs with probability $1-\exp(-\omega(n\log^2 n))$).

First, reveal an outcome of $\bfL$ such that $T\cap \bfL$ is an $(\ell,1/10)$-intercalate-expander. By \cref{lemma:existence_split_stable_intercalate}, this occurs with probability $1-\exp(-\omega(n\log^2 n))$. From now on, we will view $\bfL$ as a non-random object (i.e., all probabilities will implicitly be with respect to the conditional probability space given our outcome of $\bfL$). Let $R,C,S$ be the sets of rows, columns and symbols from \cref{lem:RCS}, and let $\mc I=\mc I(R,C,S)$ be as in the statement of \cref{lem:RCS}. Recall that $\bfL'$ is obtained from $\bfL$ via random $T$-stable intercalate switches; reveal any outcome of these random switches for $T$-stable intercalates which are \emph{not} in $\mathcal{I}$. We will show that, in the resulting conditional probability
space (which can be described in terms of $|\mc I|$ random coin flips), $\vec{\mathbf{X}}'$ has a 2-near-binomial distribution with
parameters $(n,\ell)$.

To see this, let $\vec{\mathbf{x}}_{1}\in\{0,1\}^{n},\vec{\mathbf{x}}_{2}\in\{0,1\}^{n},\vec{\mathbf{x}}_{3}\in\{0,1\}^{n}$
be the sequences of parities of rows, columns and symbols in $\mathbf{L}'$.
In our conditional probability space, the only parts of $\vec{\mathbf{x}}_{1},\vec{\mathbf{x}}_{2},\vec{\mathbf{x}}_{3}$
that remain random are the subsequences $\vec{\mathbf{x}}_{R}\in\{0,1\}^{R},\vec{\mathbf{x}}_{C}\in\{0,1\}^{C},\vec{\mathbf{x}}_{S}\in\{0,1\}^{S}$
corresponding to our identified sets of rows, columns and symbols.
We can describe their joint distribution as
\[
(\vec{\mathbf{x}}_{R}\;\vec{\mathbf{x}}_{C}\;\vec{\mathbf{x}}_{S})=(\vec y_R\; \vec y_C\; \vec y_S)+\begin{pmatrix}M_{R,\mathcal{I}}\\ M_{C,\mathcal{I}}\\M_{S,\mathcal{I}}\end{pmatrix}\vec{\mathbf{r}},
\]
where $\vec y_R,\vec y_C,\vec y_S$ describe the parities of our identified rows, columns and symbols in $\mbf L$ (which we are viewing as non-random), $\vec{\mathbf{r}}\in\{0,1\}^{|\mathcal{I}|}$ is a uniformly random zero-one sequence of length $|\mathcal{I}|$, and arithmetic is mod 2.

Now, recall the structure of the left kernel vectors from \cref{lem:RCS}. We see that $(\vec{\mathbf{x}}_{R}\;\vec{\mathbf{x}}_{C}\;\vec{\mathbf{x}}_{S})$ is a uniformly random element of $\{0,1\}^R\times \{0,1\}^C\times \{0,1\}^S$, except that the values of $\vec 1\cdot \vec{\mbf x}_R$, $\vec 1\cdot \vec{\mbf x}_C$ and $\vec 1\cdot \vec{\mbf x}_S$ are constrained (i.e., the parities of the rows, columns and symbols in $R,C,S$ are uniformly random, except that the number of odd rows, number of odd columns, and number of odd symbols are constrained mod 2). This implies the desired result.
\end{proof}

\section{Setup for the proof of the intercalate-expander lemma}
\label{sec:setup_int_expander_theorem}
In this section, we set the stage for the proof of \cref{lemma:existence_split_stable_intercalate}. We first rephrase it by slightly changing the nature and role of the template. After that, we state two main lemmas, and show how to deduce the (rephrased) intercalate-expander lemma from them (we will then spend the following sections of the paper proving these two lemmas).

Recall the sparsifying role of the template $T$ in \cref{lemma:existence_split_stable_intercalate}.
To prove \cref{lemma:existence_split_stable_intercalate}, we will take a slightly different view on $T$. Namely, for most of the proof it will be convenient to instead work with a template \emph{hypergraph} $H$ (which specifies a set of row/column/symbol triples, instead of a set of row/column pairs). Our template hypergraph will be obtained by sampling $\mb G^{(3)}(n,\varepsilon)$ with an appropriate~$\varepsilon$, as specified in the following definition.

\begin{definition}
\label{def:ell_eps}
Throughout this and the following sections, we fix a constant $\beta > 0$ (cf.\ \cref{lemma:existence_split_stable_intercalate}). We also let $\varepsilon = \eta \log^{-1}{n}$ for some $\eta = o(1)$ that slowly goes to $0$ (concretely, take $\eta = 1/\log\log{n}$) and let $\ell = \log^{11}{n}$.
\end{definition}
We now state a version of the intercalate-expander lemma with a random template hypergraph.

\begin{lemma}
\label{lemma:intercalate_expander_random_hypergraph_template}
Fix a constant $\beta >0$. Consider a random hypergraph $\bfH \sim \mb G^{(3)}(n,\varepsilon)$ and an independent random Latin square $\bfL \sim \on{Unif}(\cL_n)$ (both interpreted as subgraphs of $K^{(3)}_{n,n,n}$). Then
\[ \mathbb{P}[ \bfH\cap \bfL \text{ is an }(\ell,\beta)\text{-intercalate-expander}\,] \geq 1- \exp(-\omega(n \log^2{n})).\]
\end{lemma}
\begin{remark}
    \cref{lemma:intercalate_expander_random_hypergraph_template} trivially implies that there is a hypergraph $H\subseteq K_{n,n,n}^{(3)}$ such that \[\Pr[H \cap \bfL\text{ is an }(\ell,\beta)\text{-intercalate-expander}]\ge 1 - \exp(-\omega(n\log^2{n})).\] However, this would not have sufficed to prove \cref{thm:main-technical} (we would have no guarantee that if we switch an intercalate which lies in $H$, the resulting switched intercalate still lies in $H$). The slightly different notion of a template in \cref{def:template} was chosen to avoid this issue.
\end{remark}

We first show the (simple) deduction of \cref{lemma:existence_split_stable_intercalate} from \cref{lemma:intercalate_expander_random_hypergraph_template}. For this, we record the following simple fact, which we repeatedly use throughout the rest of the paper.
\begin{fact}\label{fact:joint_prob_vs_most_outcomes}
Suppose $\bfH_1 \in \cH_1$ and $\bfH_2 \in \cH_2$ are independent random objects, let $\cP^{(2)} \subseteq \cH_1 \times \cH_2$, and for any $H_1\in \cH_1$ and $p\in [0,1]$ let $\cP(p)$ be the set of all $H_1\in \mc H_1$ such that $\Pr[(H_1,\mbf H_2)\in \mc P^{(2)}]> p$. Then
\begin{enumerate}
    \item \label{fact:joint_to_separate} if $\mathbb{P}[(\bfH_1, \bfH_2) \in \cP^{(2)}] \leq p$, then $\Pr[\mbf H_1\in \mc P(\sqrt p)]\le \sqrt p$.
    \item \label{fact:separate_to_joint} If $\Pr[\mbf H_1\in \mc P(p_2)]\le p_1$, then $\mathbb{P}[(\bfH_1, \bfH_2) \in \cP^{(2)}] \leq p_1+p_2$.
\end{enumerate}
\end{fact}

\begin{proof}[Proof of \cref{lemma:existence_split_stable_intercalate}]
Say that $L\in \cL_n$ is \emph{good} if \[\Pr[\bfH \cap L\text{ is an }(\ell,\beta)\text{-intercalate-expander}]\ge 1 - \exp(-\omega(n\log^2{n})).\]
By \cref{lemma:intercalate_expander_random_hypergraph_template} and \cref{fact:joint_prob_vs_most_outcomes}(\ref{fact:joint_to_separate}), $\mbf L$ is good with probability $1-\exp(-\omega(n\log^2{n} )$.

Now, let $\bfT \subseteq [n]^2$ be a random subset of row/column pairs, where each is included with probability $\varepsilon$ independently. Note that for each fixed outcome $L$ of $\bfL$, the two random partial Latin squares $\bfH \cap L$ and $\bfT \cap L$ have exactly the same distribution. Therefore, for each good $L$,
\[\Pr[\bfT \cap L\text{ is an }(\ell,\beta)\text{-intercalate-expander}]\ge 1 - \exp(-\omega(n\log^2{n})),\]
and \cref{fact:joint_prob_vs_most_outcomes}(\ref{fact:separate_to_joint}) yields
\[\Pr[\bfT \cap \bfL\text{ is an }(\ell,\beta)\text{-intercalate-expander}]\ge 1 - \exp(-\omega(n\log^2{n})).\]
Applying \cref{fact:joint_prob_vs_most_outcomes}(\ref{fact:joint_to_separate}) again, we see that almost all outcomes $T$ of $\mbf T$ are suitable for the conclusion of \cref{lemma:existence_split_stable_intercalate} (namely, $\mbf T$ satisfies  the required property with probability $1-\exp(-\omega(n\log^2{n}))$).
\end{proof}

We now move on to reducing \cref{lemma:intercalate_expander_random_hypergraph_template} to two main lemmas. Recall that the property of being an intercalate-expander says that (with respect to any appropriate 6-tuple $(\sR, \sR^*, \sC, \sC^*, \sS, \sS^*)$) there is at least one stable intercalate consistent with our 6-tuple. Our first lemma (\cref{lemma:min_disjoint_split_ints}) will say that there are likely to be \emph{many} intercalates consistent with every appropriate 6-tuple (and says nothing about isolatedness or stability). Our second lemma (\cref{lemma:max_entries_bad_configs}) will say that there is likely to be a much smaller number of entries which participate in ``bad sets'' of intercalates which violate isolatedness or stability (i.e., they participate in a critical set of intercalates or a pair of intersecting intercalates).

\begin{remark}\label{rem:upper-tail}It is crucially important here that we consider the \emph{number of entries} in bad sets of intercalates (as opposed to the number of bad sets of intercalates themselves). This is due to the ``infamous upper tail'' problem which arises when trying to estimate the upper tail of subgraph counts in random graphs and hypergraphs. Due to clustering phenomena, these upper tails are `fatter' than their lower counterparts (we need to pay very close attention to tail bounds due to the exponential error terms in \cref{lem:random_LS_to_triangle_removal}). 
\end{remark}

Before stating our two lemmas, we need some preparations.

\begin{definition}
    For sets $\sR$ and $\sR^*$ of rows, sets $\sC$ and $\sC^*$ of columns and sets $\sS$ and $\sS^*$ of symbols, we say that an intercalate is $(\sR,\sR^*,\sC,\sC^*,\sS,\sS^*)$-\emph{split} (or just \emph{split} if the parameters are clear from context) if it has one row in $\sR$, the other row in $\sR^*$, one column in $\sC$, the other column in $\sC^*$, one symbol in $\sS$, and the other symbol in $\sS^*$.
    \end{definition}
    By symmetry between the rows, columns, and symbols of a Latin square, in the context of \cref{def:intercalate-expander} it suffices to consider the case where $|\sR^*| = \ell$ and $|\sR|, |\sC|, |\sC^*|, |\sS|, |\sS^*| \geq \beta n$.
\begin{definition}
We say a $6$-tuple $(\sR,\sR^*,\sC,\sC^*,\sS,\sS^*)$ is \emph{$(\ell,\beta)$-permissible} if we have $|\sR^*| = \ell$ and $\ |\sR|, |\sC|, |\sC^*|, |\sS|, |\sS^*| \geq \beta n$. 
\end{definition} Our goal is then to show that, with very high probability, all $(\ell,\beta)$-permissible $6$-tuples have a split stable intercalate.

\begin{definition}
Let $P$ be a partial Latin square, let $\{A_1, \dots, A_t\}$ be a critical set of isolated intercalates in $P$, and let $A'$ be an intercalate that can be created by switching a subset of $\{A_1, \dots, A_t\}$. Then we say the entries in $A_1 \cup \dots \cup A_t \cup (A' \cap P)$ comprise a \emph{critical configuration} with respect to $P$.

We say that a set of entries in a partial Latin square $P$ is a \emph{bad configuration} with respect to $P$ either if it is a critical configuration with respect to $P$, or if it comprises two intercalates that intersect in one entry. (We will from now on omit ``with respect to $P$'' if $P$ is clear from the context.)

We say that a bad configuration is $(\sR, \sR^*, \sC, \sC^*, \sS, \sS^*)$-\emph{split} if at least one intercalate in it (that is, one of the at most four intercalates in the critical set of intercalates or one of the two intersecting intercalates, depending on the type of bad configuration) is $(\sR, \sR^*, \sC, \sC^*, \sS, \sS^*)$-split.
\end{definition}

\begin{remark}
    We will sometimes need to talk about sets of entries which make up some arrangement of intercalates, and we will sometimes need to talk about sets of intercalates themselves. To try to assist the reader to keep track of this distinction, we will reserve the word ``configuration'' for a set of entries in a partial Latin square (i.e. a set of hyperedges, in the hypergraph perspective in \cref{fact:hypergraph_view}).
\end{remark}
 
Note that pairs of intercalates in a partial Latin square can intersect in at most one entry, so if an intercalate in a partial Latin square does not intersect any other intercalate in one entry, then it is isolated. 
Thus, an intercalate in a (partial) Latin square is stable if and only if it is not a subset of any bad configuration.

Instead of reasoning directly about the total number of split intercalates, it will be easier to show lower bounds on the size of the maximum family of disjoint split intercalates (as this random variable behaves better with respect to \cref{freedman_inequality}). The next lemma states that if we consider a random Latin square intersected with a random sparse template hypergraph, then there is very likely to be a large family of disjoint split intercalates.

\begin{definition}
    For $c>0$ and a tripartite hypergraph $H \subseteq K^{(3)}_{n,n,n}$, let $\cT^{\mr{int}}(H,c)\subseteq \cL_n$ be the set of Latin squares $L$ such that for some $(\ell,\beta)$-permissible $6$-tuple $(\sR, \sR^*, \sC, \sC^*, \sS, \sS^*)$, the maximum family of disjoint $(\sR, \sR^*, \sC, \sC^*, \sS, \sS^*)$-split intercalates in $H \cap L$ has size smaller than $c \varepsilon^4 n\ell$.
\end{definition}

\begin{lemma}
\label{lemma:min_disjoint_split_ints}
Fix a constant $c>0$ which is sufficiently small with respect to $\beta$. Then, consider a random hypergraph $\bfH \sim \mb G^{(3)}(n,\varepsilon)$ and an independent random Latin square $\bfL \sim \on{Unif}(\cL_n)$. We have
\[ \mathbb{P}[ \bfL \in \cT^{\mr{int}}(\bfH,c)] \leq \exp(-\omega(n\log^2{n})).\]
\end{lemma}

Next, the following lemma gives an upper bound on the number of entries in split bad configurations.
\begin{definition}
    We say that an entry in a row in $\sR^*$ is \emph{covered} by some split bad configuration if it belongs to a split intercalate in it. We say an intercalate or a bad configuration is $\sR^*$-split if it is $(\sR, \sR^*, \sC, \sC^*, \sS, \sS^*)$-split for some $(\ell,\beta)$-permissible $6$-tuple $(\sR, \sR^*, \sC, \sC^*, \sS, \sS^*)$ containing $\sR^*$. For $C>0$ and a hypergraph $H \subseteq K^{(3)}_{n,n,n}$, let $\cT^{\mr{bad}}(H, C)\subseteq \cL_n$ be the set of Latin squares $L$ such that for some set of rows $\sR^*$ of size $\ell$, the number of entries in rows in $\sR^*$ covered by $\sR^*$-split bad configurations in $H \cap L$ is more than $C \varepsilon^5 n\ell$.
\end{definition}
\begin{lemma}
\label{lemma:max_entries_bad_configs}
There is an absolute constant $C>0$ such that the following holds. Consider a random hypergraph $\bfH \sim \mb G^{(3)}(n,\varepsilon)$ and an independent random Latin square $\bfL \sim \on{Unif}(\cL_n)$. We have
\[ \mathbb{P}[ \bfL \in \cT^{\mr{bad}}(\bfH, C) ] \leq \exp(-\omega(n \log^2{n})).\]
\end{lemma}
We can now deduce~\cref{lemma:intercalate_expander_random_hypergraph_template} from \cref{lemma:min_disjoint_split_ints,lemma:max_entries_bad_configs}.
\begin{proof}[Proof of \cref{lemma:intercalate_expander_random_hypergraph_template}]
By the union bound, \cref{lemma:min_disjoint_split_ints}, and~\cref{lemma:max_entries_bad_configs}, there are $C, c>0$ such that we have
\[ \mathbb{P}[\bfL \in \cT^{\mr{int}}(\bfH, c)\cup \cT^{\mr{bad}}(\bfH, C)] \leq \exp(-\omega(n\log^2{n})).\]
Note that if $\bfL \notin \cT^{\mr{int}}(\bfH, c)\cup \cT^{\mr{bad}}(\bfH, C)$, then for every $(\ell,\beta)$-permissible $6$-tuple $(\sR, \sR^*, \sC, \sC^*, \sS, \sS^*)$, there are many $(\sR, \sR^*, \sC, \sC^*, \sS, \sS^*)$-split intercalates in $\bfH\cap \bfL$, and only a small fraction of them can have entries covered by $(\sR, \sR^*, \sC, \sC^*, \sS, \sS^*)$-split bad configurations. Therefore, at least one of them must be stable. By symmetry between rows, columns, and symbols, and a union bound, we see that we have
\[ \mathbb{P}[ \mbf H\cap \bfL \text{ is an }(\ell,\beta)\text{-intercalate-expander}] \geq 1- \exp(-\omega(n \log^2{n})).\]
That is to say, almost all outcomes $H \cap L$ of $\mbf H \cap \bfL$ satisfy the conclusion of \cref{def:intercalate-expander}.
\end{proof}

It remains to prove \cref{lemma:min_disjoint_split_ints,lemma:max_entries_bad_configs}, which we will do in \cref{sec:existence_disjoint_ints,sec:upper_bounding_entries_in_bad_configurations}, respectively. Both proofs will make crucial use of \cref{lem:random_LS_to_triangle_removal} (to deduce results about random Latin squares from very-high-probability results about the triangle removal process).
Given \cref{lem:random_LS_to_triangle_removal}, the proof of \cref{lemma:min_disjoint_split_ints} is a rather quick (though somewhat delicate) consequence of \cref{freedman_inequality}. The proof of \cref{lemma:max_entries_bad_configs} is much longer, and requires several additional ideas.

\section{Existence of many disjoint intercalates}
\label{sec:existence_disjoint_ints}
As already mentioned, to prove \cref{lemma:min_disjoint_split_ints}, we employ \cref{lem:random_LS_to_triangle_removal} and work in the setting of the triangle removal process. To do that, we define a property of ordered partial Latin squares that satisfies suitable inheritance properties with respect to $\cT^{\mr{int}}(H,c)$ (recall \cref{def:inherited}).

\begin{definition}
Recall the definitions of $\ell$ and $\varepsilon$ from \cref{def:ell_eps}.
    For $c>0$ and $H \subseteq K^{(3)}_{n,n,n}$, let $\vec{\cU}^{\mr{int}}(H,c) \subseteq \vec \cP_{n}$ be the set of ordered partial Latin squares $\vec{P}$ such that for some $(\ell,\beta)$-permissible $6$-tuple $(\sR, \sR^*, \sC, \sC^*, \sS, \sS^*)$, the maximum family of disjoint $(\sR, \sR^*, \sC, \sC^*, \sS, \sS^*)$-split intercalates in $H \cap \vec{P}$ has size less than $c \varepsilon^4 n\ell$.
\end{definition}
Now, the following lemma is a version of \cref{lemma:min_disjoint_split_ints} for the triangle removal process.
\begin{lemma}
\label{lemma:min_disjoint_split_ints_binomial}
Fix constants $\alpha,c>0$, such that $\alpha$ is sufficiently small and $c$ is sufficiently small in terms of $\alpha$ and $\beta$. Consider independent random hypergraphs $\vec{\bfR} \sim \trp(n,\alpha n^2)$ and $\bfH \sim \mb G^{(3)}(n,\varepsilon)$. Then
\[ \mathbb{P}[\vec \bfR \in \vec {\cU}^{\mr{int}}(\bfH, c)] \leq \exp(-\omega(n\log^2{n})).\]
\end{lemma}
We first deduce \cref{lemma:min_disjoint_split_ints} from \cref{lemma:min_disjoint_split_ints_binomial} using \cref{lem:random_LS_to_triangle_removal}.
\begin{proof}[Proof of \cref{lemma:min_disjoint_split_ints}]
Let $\alpha,c$ be small enough for \cref{lemma:min_disjoint_split_ints_binomial} and let $\rho = 1$. 
By \cref{lemma:min_disjoint_split_ints_binomial} and \cref{fact:joint_prob_vs_most_outcomes}(\ref{fact:joint_to_separate}), with probability $1-\exp(-\omega(n\log^2{n})$ our random hypergraph $\bfH \sim \mb G^{(3)}(n,\varepsilon)$ is such that $\mathbb{P}[\vec \bfR \in \vec{\cU}^{\mr{int}} (\bfH, c)\,|\,\mbf H] \leq \exp(-\omega(n \log^2{n}))$. Let $H$ be an outcome of $\bfH$ for which this holds. Note that $\vec{\cU}^{\mr{int}}(H,c)$ is $(\rho, \alpha n^2)$-inherited from $\cT^{\mr{int}}(H,c)$: indeed, if $L \in \cT^{\mr{int}}(H,c)$, then there is some $(\ell,\beta)$-permissible $6$-tuple $(\sR, \sR^*, \sC, \sC^*, \sS, \sS^*)$ such that the maximum family of disjoint $(\sR, \sR^*, \sC, \sC^*, \sS, \sS^*)$-split intercalates in $H\cap L$ has size smaller than $c\varepsilon^4 n \ell$, so for any $P \subseteq L$, the maximum family of disjoint $(\sR, \sR^*, \sC, \sC^*, \sS, \sS^*)$-split intercalates in $H\cap P$ also has size smaller than $c\varepsilon^4 n \ell$. Thus, by \cref{lem:random_LS_to_triangle_removal},
\[ \mathbb{P}[\bfL \in \cT^{\mr{int}}(H,c)] \leq \exp(2n\log^2{n}) \mathbb{P}[\vec \bfR \in \vec{\cU}^{\mr{int}} (H,c)] \leq \exp(-\omega(n\log^2{n})).\]
Recalling our choice of $H$, \cref{fact:joint_prob_vs_most_outcomes}(\ref{fact:separate_to_joint}) then implies that
\[\mathbb{P}[\bfL \in \cT^{\mr{int}}(\bfH,c)] \leq \exp(- \omega(n\log^2{n})),\]
as desired.
\end{proof}

\begin{proof}[Proof of \cref{lemma:min_disjoint_split_ints_binomial}]
We will prove the desired statement with $c = 10^{-5}\beta^5 \alpha^4$.

Note that the property $\vec{\cU}^{\mr{int}}(H,c)$ is monotone decreasing. Let $\bfG \sim \mb G^{(3)}(n, \alpha/n)$, and let $\mbf G^*$ be obtained from $\mbf G$ by deleting all hyperedges involved in conflicts (recall that a conflict is a pair of hyperedges that intersect in more than one vertex). By \cref{lemma:triangle_removal_to_cleaned_binomial}, for each $H \subseteq K^{(3)}_{n,n,n}$ we have
\[ \mathbb{P}[\vec \bfR \in \vec{\cU}^{\mr{int}}(H,c)] = O( \mathbb{P}[\bfG^* \in \vec{\cU}^{\mr{int}}(H,c)]),\]
so for our random hypergraph $\mbf H$ we have
\[ \mathbb{P}[\vec \bfR \in \vec{\cU}^{\mr{int}}(\bfH,c)] = O(\mathbb{P}[\bfG^* \in \vec{\cU}^{\mr{int}}(\bfH,c)]).\]
Thus, from now on we work in $\bfG^*$. Fix an $(\ell,\beta)$-permissible $6$-tuple $(\sR, \sR^*, \sC, \sC^*, \sS, \sS^*)$. We will show that the probability that the maximum family of disjoint $(\sR, \sR^*, \sC, \sC^*, \sS, \sS^*)$-split intercalates in $\bfH \cap \bfG^*$ has size smaller than $c\varepsilon^4 n \ell$ is at most $\exp(-\omega(n\log^2{n}))$.

Our main tool will be \cref{freedman_inequality}, but to effectively apply this concentration inequality we need to be quite careful with the way we interpret $\bfH \cap \bfG^*$ as a function of independent choices. Specifically, for each potential hyperedge $e$, let $\bfg^{\bfh}_e$ and $\bfg_e$ be two independent Bernoulli random variables with $\mathbb{P}[\bfg^{\bfh}_e=1] = \alpha \varepsilon/{n} $ and $\mathbb{P}[\bfg_e=1] = \alpha (1-\varepsilon)/(n - \alpha \varepsilon)$. We generate $\bfG$ and $\bfH \cap \bfG$ by sampling $\bfg^{\bfh}_e$ and $\bfg_e$ for each potential edge $e$, setting $e \in \bfH \cap\bfG$ if $\bfg^{\bfh}_e = 1$, and setting $e \in \bfG$ if $\bfg^{\bfh}_e = 1$ or $\bfg_e = 1$ (or both). To see that this correctly describes the joint distribution of $\bfG$ and $\bfH \cap\bfG$, note that
\begin{align*}
 \mathbb{P}[e \in \bfG] = 1 - (1 - \mathbb{P}[\bfg^{\bfh}_e = 1])(1 - \mathbb{P}[\bfg_e=1]) &= 1 - \Big(1- \frac{\alpha \varepsilon}{n}\Big)\Big(1 - \frac{\alpha(1-\varepsilon)}{n - \alpha \varepsilon}\Big) \\
 &= \frac{\alpha \varepsilon}{n} + \frac{\alpha(1-\varepsilon)}{n - \alpha\varepsilon} - \frac{\alpha^2 \varepsilon(1-\varepsilon)}{n(n-\alpha \varepsilon)} = \frac{\alpha}{n}.
\end{align*}

Now, the plan is to reveal $\bfG$ and $\bfH$ (via the random variables $\bfg^{\bfh}_e$ and $\bfg_e$) in two phases. First, we reveal all entries in rows outside of $\sR^*$, and then we reveal entries in rows in $\sR^*$.

Recall that a conflict is a pair of hyperedges that intersect in more than one vertex. We say a pair of entries $\{(r, c_1, s_1), (r, c_2, s_2)\}$ in $\bfH \cap \bfG$ on the same row $r \notin \sR^*$ has a \emph{second-order conflict} if $s_1$ appears in column $c_2$ or if $s_2$ appears in column $c_1$ on a row outside of $\sR^*$, with respect to $\bfG$. For $r \in \sR\setminus \sR^*$, $c_1 \in \sC$, $c_2\in \sC^*$, $s_1\in \sS$ and $s_2 \in \sS^*$, we say a pair $\{(r, c_1, s_1), (r, c_2, s_2)\}$ is a \emph{half-intercalate} if both its hyperedges are present in $\bfH \cap \bfG$ and have no conflicts with entries of $\bfG$ outside of $\sR^*$, and the pair has no second-order conflicts (which would make it impossible to complete it to an intercalate in $\mbf G^*$).

The idea is that in the first phase, by revealing all the entries in the rows outside $\sR^*$, we reveal all the half-intercalates. In the second phase, each half-intercalate $\{(r, c_1, s_1), (r, c_2, s_2)\}$ can then be completed to a split intercalate in $\bfH \cap \mbf G^*$ via any pair of entries $(r', c_1, s_1),(r', c_2, s_2)$ (with $r' \in \sR^*$) in $\mbf G$, as long as the four entries involved in the intercalate do not conflict with other entries in $\bfG$ on rows in $\sR^*$.

Let $\bfX$ be the number of half-intercalates. We will apply \cref{freedman_inequality} to show concentration of $\bfX$. First note that
\[\mathbb{E}[\bfX] \geq \binom{\beta n}{2}^2 (\beta n-\ell) (\varepsilon \alpha / n)^2 (1-\alpha/n)^{8n-8} \geq 10^{-1} \varepsilon^2 \alpha^2 (\beta n-\ell)\beta^4 n^2 e^{-16 \alpha} \geq 10^{-2} \varepsilon^2 \alpha^2 \beta^5 n^3. \]
To see this, recall that $|\sR\setminus \sR^*|\ge \beta n-\ell$ and each of $\sR, \sC, \sC^*,\sS,\sS^*$ have size at least $\beta n$; also, each pair of entries is present in $\bfH \cap \bfG$ with probability $(\varepsilon\alpha/n)^2$, it has no conflicts with other entries in $\bfG$ with probability at least $(1-\alpha/n)^{6n-6}$, and it has no second-order conflicts with probability at least $(1-\alpha/n)^{2n-2}$. 

Recall that we work with the independent random variables $\bfg^\bfh_e$ and $\bfg_e$ for each of the $O(n^3)$ many potential hyperedges $e$ outside the rows in $\sR^*$, and each of these random variables is $1$ with probability $O (1 / n)$.  The appearance of a hyperedge in $\bfH \cap \bfG$ can increase $\bfX$ by at most $n-1$, since each entry can be in at most $n-1$ half-intercalates. The appearance of a hyperedge in $\bfG$ can decrease $\bfX$ by at most $4n$, since it can destroy at most $3$ hyperedges in $\bfG^*$ due to conflicts (and each of them could be in at most $n-1$ pairs counted by $\bfX$), and can introduce at most $n$ second-order conflicts. Thus, the effect of changing any individual random variable (of the form $\bfg^\bfh_e$ or $\bfg_e$) is $O(n)$. By \cref{freedman_inequality}, with probability $1 - \exp(-\Omega(\varepsilon^4  n^2))$, we have $\bfX \geq 10^{-3} \varepsilon^2 \alpha^2 \beta^5 n^3$. For the rest of the proof, we condition on an outcome of the first phase such that this is the case (i.e., in the rest of the proof, all probabilistic considerations are implicitly with respect to this conditional probability space).

Now, in the second phase we reveal $\bfG$ and $\bfH \cap \bfG$ in rows in $\sR^*$, and estimate the size $\bfZ$ of the maximum family of disjoint $(\sR, \sR^*, \sC, \sC^*, \sS, \sS^*)$-split intercalates in $\bfH \cap \bfG^*$. Note that each half-intercalate counted by $\bfX$ forms an intercalate with two other entries on row $r\in\sR^*$ if and only if these two other entries are present in $\bfH \cap \bfG$ and no entry in $\bfG$ in a row in $\sR^*$ conflicts with any of the four entries of the intercalate. We can compute a lower bound on $\mathbb{E}[\bfZ]$ by subtracting the expected number of pairs of intersecting split intercalates from the expected number of split intercalates. Indeed,
\[ \mathbb{E}[\bfZ] \geq \bfX \ell (\varepsilon\alpha/n)^2(1-\alpha/n)^{8n-8} - \bfX n \ell^2 (\varepsilon\alpha/n)^4 - \bfX n \ell (\varepsilon\alpha/n)^3 \geq \bfX \frac{\ell\alpha^2 \varepsilon^2 }{2n^2} \geq 10^{-4}  \alpha^4 \beta^5\varepsilon^4 n \ell,\]
where the term $(1-\alpha/n)^{8n-8}$ accounts for the probability that some entry in a fixed potential split intercalate conflicts with some entry in a row in $\sR^*$. The second term bounds the expected number of pairs of split intercalates that intersect in an entry outside of $\sR^*$, since for a fixed half-intercalate, there are at most $n$ choices for the other column of the other intercalate and at most $\ell$ choices for the other row of each intercalate, and for each such combination of choices, four entries need to be present in rows in $\sR^*$. The third term bounds the expected number of pairs of split intercalates that intersect in an entry in $\sR^*$, since for a fixed half-intercalate,  there are at most $\ell$ choices for the common to both intercalates row in $\sR^*$ and at most $n$ choices for the other row of the other intercalate, which is outside of $\sR^*$ (and this determines the entries of both intercalates completely). We again apply \cref{freedman_inequality} with the remaining unrevealed random variables $\bfg^\bfh_e$ and $\bfg_e$.
Since we are interested in a family of disjoint intercalates, the appearance of a hyperedge can increase $\bfZ$ by at most $1$ and can decrease it by at most $3$ (as it can have a conflict with at most $3$ edges in $\bfG^*$). Crucially, there are now only $O(n^2\ell)$ many unrevealed random variables. Thus, by \cref{freedman_inequality}, we have that $\bfZ \geq 10^{-5} \alpha^4 \beta^5 \varepsilon^4 n\ell$ with probability at least $1 - \exp(-\Omega_{\alpha,\beta}(\varepsilon^8 n\ell))$.

We conclude the proof with a union bound over all (at most $2^{6n}$) choices for an $(\ell,\beta)$-permissible $6$-tuple $(\sR, \sR^*, \sC, \sC^*, \sS, \sS^*)$.
\end{proof}

\section{Upper-bounding the number of entries in bad configurations}
\label{sec:upper_bounding_entries_in_bad_configurations}

We now turn to the proof of \cref{lemma:max_entries_bad_configs}. For any set of $\ell$ rows $\sR^*$, our goal is to upper bound the probability that there are many entries in rows in $\sR^*$ covered by $\sR^*$-split bad configurations.

\subsection{Setup} As in the proof of \cref{lemma:min_disjoint_split_ints_binomial}, the special role of the rows in $\sR^*$ requires us to consider two ``phases''. In the first phase we reveal a subset of the entries of our random Latin square, and study configurations of entries which are ``almost'' bad configurations (they are just missing some entries from $\sR^*$). In the second phase we reveal the rest of our random Latin square, and we study which of our ``almost bad configurations'' give rise to actual bad configurations. However, this breakdown into phases is much more complicated than for \cref{lemma:min_disjoint_split_ints_binomial} (for several different reasons), and requires some setup.
\begin{definition}
    Given a $\sR^*$-split bad configuration $F$, we say that the \emph{special entries} of $F$ are its two entries which belong to rows in $\sR^*$. (Technically, it is possible that there is more than one split intercalate in a bad configuration, or that the bad configuration contains an intercalate with both its rows in $\sR^*$. In these cases we imagine multiple ``copies'' of the bad configuration, each time making a different choice for the two special entries). Also, we say that the intercalate in $F$ which contains the special entries is the \emph{special intercalate}. 
\end{definition}
\begin{definition}
Given a partial Latin square $P$ and a set of rows $\sR^*$, consider a pair $\{e_1,e_2\}$ of row/column/symbol triples, belonging to some row in $\sR^*$, which are not already present in $P$ but which can be added to $P$ without causing conflicts (i.e.,  $P\cup\{e_1,e_2\}$ is a partial Latin square). Say that $\{e_1,e_2\}$ is a \emph{threatened pair} for $P$ if there is some split bad configuration $F$ in $P\cup\{e_1,e_2\}$ which has $e_1$ and $e_2$ as special entries. In this case, we say $F\cap P$ is a \emph{threat configuration} in $P$. 
\end{definition}

Informally speaking, threat configurations are the sets of entries which are ``in danger of being completed to a split bad configuration'', only requiring the addition of their two special entries. Each of the possible ways to complete a threat configuration to a split bad configuration is described by a threatened pair. A single threatened pair may complete multiple different threat configurations to split bad configurations, and a single threat configuration may be completeable by multiple different threatened pairs (to different split bad configurations).

Now, given the above terminology, we can be a bit more concrete about the high-level plan to prove \cref{lemma:max_entries_bad_configs}. In a first phase, we reveal part of our random Latin square and upper-bound the number of threatened pairs in it (we do \emph{not} upper bound the number of threat configurations themselves, cf.\ \cref{rem:upper-tail}). Then, in a second phase, we reveal the rest of the Latin square and upper bound the number of threatened pairs from the first phase that actually appear as entries in the second phase. By an averaging argument (over all the ways to split into two phases) we are able to deduce the desired bound on the number of entries in split bad configurations.

For the first phase, we consider the following property.
\begin{definition}
Recall the definitions of $\ell$ and $\varepsilon$ from \cref{def:ell_eps}.
    For $H \subseteq K^{(3)}_{n,n,n}$ and $C, \alpha, \phi>0$, let $\cT^{\mr{threat}} (H, C, \alpha, \phi)\subseteq \cL_n$ be the set of Latin squares $L$ such that the following holds. If we take a random subset $\bfP_{\alpha n^2}(L)\in \mc P_{n,\alpha n^2}$ of $\alpha n^2$ entries in $L$, then with probability at least $\phi$, for some set of rows $\sR^*$ of size $|\sR^*| = \ell$, the number of threatened pairs for $H \cap \bfP_{\alpha n^2}(L)$ is more than $C\varepsilon^3 n^3 \ell$.
\end{definition}
Note that here we are defining a property of Latin squares in terms of how their random subsets of entries behave. This may seem rather unwieldy, but it is very convenient for the $\rho$-inheritedness assumption in \cref{lem:random_LS_to_triangle_removal} (when it comes time to apply \cref{lem:random_LS_to_triangle_removal}, the $\rho$-inheritedness assumption will hold basically by definition). The issue with more straightforward properties is that bad configurations and threatened pairs do not behave well under subsampling entries, and therefore $\rho$-inheritedness would be quite difficult to prove. (The issue is that a bad configuration is defined in terms of \emph{isolated} intercalates, and deleting entries can cause more intercalates to be isolated\footnote{The reader may then wonder why we insist on isolatedness in the definition of a bad configuration. The reason is that it is very convenient to restrict our attention to isolated intercalates in a switching argument which will appear later in this section.}).

Now, the following two lemmas correspond to the two ``phases'' informally discussed above.

\begin{lemma}
\label{lemma:upper_bound_threatened_pairs} 
There is an absolute constant $C'>0$ such that the following holds. Fix constants $\alpha, \phi >0$, and consider a random hypergraph $\bfH \sim \mb G^{(3)}(n,\varepsilon)$ and an independent random Latin square $\bfL \sim \on{Unif}(\mc L_n)$. We have
\[ \mathbb{P}[ \bfL \in \cT^{\mr{threat}}(\bfH,C', \alpha, \phi)] \leq \exp(-\omega(n \log^2{n})).\]
\end{lemma}

\begin{lemma}
\label{lemma:upper_bound_bad_configs_given_not_too_many_reduced_threats}
Fix constants $\alpha,C',C,\phi>0$, such that $\alpha$ is sufficiently small, $C$ is sufficiently large in terms of $\alpha$ and $C'$, and $\phi$ is sufficiently small in terms of $\alpha$. Consider a random hypergraph $\bfH \sim \mb G^{(3)}(n,\varepsilon)$ and an independent random Latin square $\bfL \sim \on{Unif}(\mc L_n)$. We have
\[ \mathbb{P}[ \bfL \in \cT^{\mr{bad}}(\bfH, C) \setminus \cT^{\mr{threat}}(\bfH, C', \alpha, \phi)] \leq \exp(-\omega(n \log^2{n})).\]
\end{lemma}
\cref{lemma:upper_bound_threatened_pairs,lemma:upper_bound_bad_configs_given_not_too_many_reduced_threats}  are both proved using \cref{lem:random_LS_to_triangle_removal}, comparing random Latin squares to the triangle removal process. In addition, the proof of \cref{lemma:upper_bound_threatened_pairs} also involves a switching argument directly on random Latin squares (to tame the complexity of all the different possibilities for the structure of a threat configuration).

We conclude this subsection with the brief deduction of \cref{lemma:max_entries_bad_configs} from \cref{lemma:upper_bound_threatened_pairs,lemma:upper_bound_bad_configs_given_not_too_many_reduced_threats}.
\begin{proof}[Proof of \cref{lemma:max_entries_bad_configs}]
 Let $C'$ be as in \cref{lemma:upper_bound_threatened_pairs}, and then let $\alpha,C,\phi$ be as in \cref{lemma:upper_bound_bad_configs_given_not_too_many_reduced_threats} (for the same value of $C'$). Then we have
\begin{align*}
\mathbb{P}[\bfL \in \cT^{\mr{bad}}(\bfH, C)] &\leq \mathbb{P}[ \bfL \in \cT^{\mr{bad}}(\bfH, C) \setminus \cT^{\mr{threat}}(\bfH, C', \alpha, \phi)] + \mathbb{P}[ \bfL \in \cT^{\mr{threat}}(\bfH,C', \alpha, \phi)] \\
&\leq \exp(-\omega(n \log^2{n}))+\exp(-\omega(n \log^2{n}))
\end{align*}
by \cref{lemma:upper_bound_threatened_pairs,lemma:upper_bound_bad_configs_given_not_too_many_reduced_threats}.
\end{proof}

In the next three subsections, we will prove \cref{lemma:upper_bound_threatened_pairs,lemma:upper_bound_bad_configs_given_not_too_many_reduced_threats} (in reverse order).

\subsection{Few bad configurations or many threatened pairs}
\label{subsec:bad_configs}
In this subsection we prove \cref{lemma:upper_bound_bad_configs_given_not_too_many_reduced_threats}. As previously mentioned, we will work with the triangle removal process, via \cref{lem:random_LS_to_triangle_removal}. In contrast to our previous application of \cref{lem:random_LS_to_triangle_removal} in \cref{sec:existence_disjoint_ints}, this time the order of the edges will be important.
\begin{definition}\label{def:block}
    Let $\vec P\in \vec \cP_{n}$ be an ordered partial Latin square, and write $e_1,\dots,e_m$ for its entries (in order). For $\iota,\kappa\in[0,1]$, we write $\vec P[\iota,\kappa]\in \cP_{n,(\kappa-\iota)m}$ for the partial Latin square consisting of the edges $e_i$ with $\iota m< i\le \kappa m$.
\end{definition}

Recall the definitions of $\ell$ and $\varepsilon$ from \cref{def:ell_eps}.
\begin{definition}
    Let $\alpha \in (0,1)$, $C>0$, and $H \subseteq K^{(3)}_{n,n,n}$.
Let $\vec{\cU}^{\mr{threat}}(H, C) \subseteq \vec\cP_{n}$ be the set of partial ordered Latin squares $\vec P$ such that for some set of rows $\sR^*$ of size $\ell= |\sR^*|$, the number of threatened pairs for $H \cap \vec{P}[0,1/2]$ is more than $C \varepsilon^3 n^3 \ell$.
\end{definition}

\begin{definition}
We call a bad configuration \emph{consistent with} $\vec P$ if its corresponding threat configuration is in $\vec P[0,1/2]$, one of its special entries is in $\vec P[1/2,3/4]$, and its other special entry is in $\vec P[3/4,1]$. Let $\vec{\cU}^{\mr{bad}}(H, C)\subseteq \vec{\cP}_{n}$ be the set of partial ordered Latin squares $\vec P$ such that for some set of rows $\sR^*$ of size $\ell=|\sR^*|$, there are more than $C\varepsilon^5 n\ell$ entries in $\vec P[3/4,1]$ that are in rows in $\sR^*$ and are covered by $\sR^*$-split bad configurations in $H \cap \vec{P}$ that are consistent with $\vec P$.
\end{definition}

The following lemma is a version of \cref{lemma:upper_bound_bad_configs_given_not_too_many_reduced_threats} for the  triangle removal process.

\begin{lemma}
\label{lemma:upper_bound_bad_configs_given_not_too_many_reduced_threats_bin}
Fix constants $\alpha,C'$, such that $\alpha$ is sufficiently small. Consider independent random hypergraphs $\vec{\bfR} \sim \trp(n, \alpha n^2)$ and $\bfH \sim \mb G^{(3)}(n, \varepsilon)$. Then
\[ \mathbb{P}[ \vec{\bfR} \in \vec{\cU}^{\mr{bad}}(\bfH, C') \setminus  \vec{\cU}^{\mr{threat}}(\bfH, C')] \leq \exp(-\omega(n \log^2{n})).\]
\end{lemma}

Before proving \cref{lemma:upper_bound_bad_configs_given_not_too_many_reduced_threats_bin}, we show how it implies \cref{lemma:upper_bound_bad_configs_given_not_too_many_reduced_threats}, using \cref{lem:random_LS_to_triangle_removal} and an averaging argument.

\begin{proof}[Proof of \cref{lemma:upper_bound_bad_configs_given_not_too_many_reduced_threats}]
Let $m=2\alpha n^2$ and $\rho=\alpha^{16}/2^4$. We assume that $\phi\le\rho$ and $\rho C\ge C'$.

By \cref{lemma:upper_bound_bad_configs_given_not_too_many_reduced_threats_bin} and \cref{fact:joint_prob_vs_most_outcomes}(\ref{fact:joint_to_separate}), with probability $1-\exp(-\omega(n\log^2{n}))$ over the randomness of $\bfH$ we have
\[ \mathbb{P}[ \vec{\bfR} \in \vec{\cU}^{\mr{bad}}(\bfH, C') \setminus  \vec{\cU}^{\mr{threat}}(\bfH, C')\,|\,\bfH] \leq \exp(-\omega(n \log^2{n})).\]
The desired result would therefore follow from \cref{lem:random_LS_to_triangle_removal} and \cref{fact:joint_prob_vs_most_outcomes}(\ref{fact:separate_to_joint}), if we knew that $\vec{\cU}^{\mr{bad}}(H, C') \setminus \vec\cU^{\mr{threat}}(H, C')$ is $(\rho,m)$-inherited from $\cT^{\mr{bad}}(H, C) \setminus \cT^{\mr{threat}}(H, C', \alpha, \phi)$ for all outcomes $H$ of $\bfH$. For the rest of the proof, our goal is to prove that this is the case.

Suppose $L \in \cT^{\mr{bad}}(H, C) \setminus \cT^{\mr{threat}}(H, C', \alpha, \phi)$, and let $\vec \bfP_m(L)\in \vec \cP_{n,m}$ be a random ordering of $m$ random entries of $L$. We need to prove that $\vec \bfP_m(L)\in  \vec{\cU}^{\mr{bad}}(H, C') \setminus  \vec{\cU}^{\mr{threat}}(H, C')$ with probability at least $\rho$.

By the definition of $\cT^{\mr{threat}}(H, C', \alpha, \phi)$ and the fact that $\vec \bfP_{m}(L)[0,1/2]$ is a uniformly random set of $\alpha n^2$ entries of $L$, first note that $\mathbb{P}[ \vec \bfP_m(L) \in  \vec\cU^{\mr{threat}}(H,C') ] \leq \phi$. So, it suffices to show that $\vec \bfP_m(L) \in \vec\cU^{\mr{bad}}(H, C') $ with probability at least $\rho + \phi$.

Since $L \in \cT^{\mr{bad}}(H, C)$, there is some $\sR^*$ such that there are $X \geq C \varepsilon^5 n \ell$ many entries in rows in $\sR^*$ covered by $\sR^*$-split bad configurations in $H \cap L$. Let $\bfY\le X$ be the number of such entries in $\vec\bfP_m(L)[3/4,1]$ which are covered by $\sR^*$-split bad configurations in $H \cap \vec \bfP_m(L)$ consistent with $\vec \bfP_m(L)$. For each covered entry $(r,c,s)$ in $H\cap L$ counted by $X$, the probability that its associated split bad configuration with $s \leq 16$ entries is present in $\vec \bfP_m(L)$ is $\binom{n^2 - s}{m-s}/\binom{n^2}{m}=(1+o(1))(2\alpha)^s$. Given that it is present in $\vec \bfP_m(L)$, the probability that it is consistent with $\vec \bfP_m(L)$ and that $(r,c,s) \in \vec \bfP_m(L)[3/4,1]$ is \[\frac{m/4}{m} \cdot \frac{m/4}{m-1}\cdot  \binom{m/2}{s-2} / \binom{m-2}{s-2}=(1+o(1))(1/4)^2(1/2)^{s-2}.\] Note that each isolated intercalate in $H\cap L$ whose entries are in $\vec \bfP_m(L)$ is also isolated in $H\cap \vec \bfP_m(L)$, so each bad configuration in $H \cap L$ that is present in $H\cap \vec \bfP_m(L)$ is also a bad configuration in $H\cap \vec \bfP_m(L)$. Therefore
\[ \mathbb{E}[\bfY] \geq (1+o(1))X (2\alpha)^s(1/4)^2(1/2)^{s-2}
\ge X \alpha^{16}/4.\]
Recall our assumptions on $\phi,C$ from the start of the proof. Let $\delta=\alpha^{16}/8\geq \rho+\phi$ and note that $\delta X\ge C'\varepsilon^5 n \ell$. By Markov's inequality, $\mathbb{P}[\vec \bfP_m(L) \in \vec\cU^{\mr{bad}}(H, C')]$ is at least
\[\Pr[\mathbf Y\ge \delta X]=1- \Pr[X-\mathbf Y> (1-\delta)X]\ge 1- \frac{\mb E[X-\mathbf Y]}{(1-\delta) X}\ge 1-\frac{1-2\delta}{1-\delta}\ge \rho+\phi,\]
as desired.
\end{proof}

Now we prove \cref{lemma:upper_bound_bad_configs_given_not_too_many_reduced_threats_bin}.
\begin{proof}[Proof of \cref{lemma:upper_bound_bad_configs_given_not_too_many_reduced_threats_bin}]

First, we reveal $\bfH \cap \vec{\bfR}[0,1/2]$. We assume that for each choice of $\sR^*$ of size $\ell$, the number of threatened pairs for $\bfH \cap \vec{\bfR}[0,1/2]$ is at most $C'\varepsilon^3 n^3 \ell$ (otherwise $\vec{\cU}^{\mr{threat}}(\bfH, C')$ holds, and there is nothing to prove). From now on, when we talk about threatened pairs in this proof, we mean threatened pairs for $\bfH \cap \vec{\bfR}[0,1/2]$.

Given revealed information, our goal is to show that with probability $1-\exp(-\omega(n\log^2 n))$, for every $\sR^*$ with $|\sR^*|= \ell$, there are at most $C' \varepsilon^5 n\ell$ entries of $\vec{\bfR}[3/4,1]$ that are in rows in $\sR^*$ and are covered by $\sR^*$-split bad configurations in $\mbf H\cap \vec{\bfR}$ that are consistent with $\vec{\bfR}$. Actually, it suffices to show this for a fixed choice of $\sR^*$ (a union bound over choices of $\sR^*$ has a negligible impact). So, fix some $\sR^*$ with $|\sR^*|= \ell$.

The next step is to consider $\bfH \cap \vec{\bfR}[1/2, 3/4]$. We say a threatened pair \emph{survives} if one of its entries is in $\bfH \cap \vec{\bfR}[1/2, 3/4]$.
\begin{claim*}
    Given revealed information, the number of surviving threatened pairs is at most $C' \varepsilon^4 n^2 \ell$ with probability $1-\exp(-\omega(n\log^2 n))$.
\end{claim*}
\begin{claimproof}
    Let $\mc U$ be the set of $G\subseteq K^{(3)}_{n,n,n}$ which intersect more than $C' \varepsilon^4 n^2 \ell$ threatened pairs. Then, for each possible outcome $H$ of $\bfH$, the set $\{G\subseteq K^{(3)}_{n,n,n}:H\cap G\in \mc U\}$ is a monotone increasing property of subgraphs of $K^{(3)}_{n,n,n}$, so by \cref{lemma:triangle_removal_chunk_to_binomial}, the probability that there are more than $C' \varepsilon^4 n^2 \ell$ surviving threatened pairs is
    \[\Pr[ \bfH\cap \vec{\bfR}[1/2,3/4]\in \mc U]\le \Pr[ \bfH \cap \bfG'\in \mc U]+\exp(-\Omega(n^2))=\Pr[\bfG\in \mc U]+\exp(-\Omega(n^2)),\]
    where $\bfG' \sim \mb G^{(3)}(n,(\alpha/2)/n)$ and $\bfG \sim \mb G^{(3)}(n,(\varepsilon\alpha/2)/n)$.

From here, we just need to apply a concentration inequality in $\bfG$, to upper bound $\Pr[\bfG\in \mc U]$. For each row/column/symbol triple $e$, let $w_e$ be the number of threatened pairs that $e$ is in, and let $\mbf X$ be the sum of $w_e$ over all hyperedges $e$ of $\mbf G$. Note that (for small $\alpha>0$) we have
    \[\mb E[\mbf X]\le 2\cdot C'\varepsilon^3 n^3 \ell\cdot \frac{\varepsilon\alpha/2}n\le \frac{C'\varepsilon^4 n^2 \ell}2.\]
    Then, for any row/column/symbol triple $e$, consider each of the threatened pairs that $e$ is in. Each of these threatened pairs, if added to $\vec{\bfR}[1/2,1]$, would complete a split bad configuration in which $e$ would be in a different intercalate. Since a single entry can be in at most $n$ many intercalates (in any Latin square), it follows that $e$ is in at most $n$ threatened pairs; that is, $w_e\le n$.
    
    We can therefore apply \cref{thm:chernoff_dependency_graph} with $\Delta = n$ and $\delta=1$, to obtain
    \[\Pr[\bfG\in \mc U]\le \Pr\left[\mbf X\ge \mb E [\mbf X]+\frac{C'\varepsilon^4n^2\ell}2\right]\le\exp(-\Omega(\varepsilon^4 n\ell)),\]
    and the desired result follows (recalling the relationship between $\ell$ and $\varepsilon$ in \cref{def:ell_eps}).
\end{claimproof}
Now, reveal $\bfH \cap \vec{\bfR}[1/2, 3/4]$; by the above claim we may assume that our revealed outcome is such that there are at most $C' \varepsilon^4 n^2 \ell$ surviving threatened pairs. Then, the final step is to consider $\bfH \cap \vec{\bfR}[3/4,1]$.

Each surviving threatened pair's presence in $\bfH \cap \vec{\bfR}$ is determined by whether its second entry is in $\bfH \cap \vec{\bfR}[3/4,1]$. Say a row/column/symbol triple $e$ is a \emph{threatened entry} if it is the second entry of some surviving threatened pair.
\begin{claim*}
    Given revealed information, there are at most $C' \varepsilon^5 n\ell$ threatened entries present in $\bfH \cap \vec{\bfR}[3/4,1]$, with probability $1-\exp(-\omega(n\log^2 n))$.
\end{claim*}
\begin{claimproof}
We proceed very similarly to the last claim (in fact, the situation is even simpler). Let $\bfG \sim \mb G^{(3)}(n,(\alpha\varepsilon/2)/n)$, and let $\bfY$ be the number of threatened entries present in $\bfG$. Using \cref{lemma:triangle_removal_chunk_to_binomial} as before, the probability that there are more than $C' \varepsilon^5 n\ell$ threatened entries present in $\bfH \cap \vec{\bfR}[3/4,1]$ is at most $\Pr[\mbf Y>C' \varepsilon^5 n\ell]+\exp(-\Omega(n^2))$.

We have\[\mb E[\mbf Y]\le C' \varepsilon^4 n^2 \ell\cdot \frac{\alpha\varepsilon/2}n\le \frac{C'\varepsilon^5 n\ell}2,\]
so by \cref{thm:chernoff_dependency_graph} (with $\Delta=1$), we have
$\Pr[\mbf Y>C' \varepsilon^5 n\ell]\le \exp(-\Omega(\varepsilon^5 n\ell))$; the desired result follows.
\end{claimproof}
Now, if $e$ is an entry of $\vec{\bfR}[3/4,1]$ that is in a row in $\sR^*$ and is covered by a split bad configuration in $\mbf H\cap \vec{\bfR}$ consistent with $\vec{\bfR}$, then $e$ was a threatened entry which happened to be present in $\bfH \cap \vec \bfR[3/4,1]$ (the converse may not hold, as bad configurations in subsets of $\bfH \cap \vec{\bfR}$ are not always bad configurations in $\bfH \cap\vec{\bfR}$ itself). So, the desired result follows from the above claim.
\end{proof}

\subsection{Upper-bounding the number of threatened pairs: setup and switching}
\label{sec:counting_threats}
\definecolor{nicegreen}{RGB}{0,180,0}

It remains to prove \cref{lemma:upper_bound_threatened_pairs}, which amounts to a careful study of threat configurations in a random subset of a random Latin square. 

Recall that threat configurations are sets of entries that are ``two entries away'' from a bad configuration, and there are a very large number of possibilities for the structure of a bad configuration (e.g., critical configurations can have different numbers of intercalates, they can require different subsets of intercalates to be switched to create a new intercalate, and they can have different rows in $\sR^*$). To tame this complexity, one crucial observation is that for every bad configuration in a partial Latin square, there is some way to switch isolated intercalates to obtain a pair of intersecting intercalates. So, with a switching argument (unrelated to the main switching argument in \cref{sec:master-from-stable}), we can reduce the number of cases significantly: it suffices to consider bad configurations that correspond to pairs of intersecting intercalates, and bad configurations that are ``one switch away'' from a pair of intersecting intercalates.

\begin{definition}
Given a partial Latin square $P\in \cP_{n}$ and a set of rows $\sR^*$, we define a \emph{basic split bad configuration} to be a set $F$ of seven entries in $P$, such that $F$ contains an intercalate $A$ with a row in $\sR^*$, and either $F$ or $(F\setminus A)\cup \bar A$ is a union of two intersecting intercalates.

Then, say that a pair $\{e_1,e_2\}$, belonging to some row in $\sR^*$, is a \emph{basic threatened pair} for $P$ if there is some basic split bad configuration $F$ in $P\cup\{e_1,e_2\}$ with $e_1,e_2 \in F$. In this case, we say $F\setminus \{e_1,e_2\}$ is a \emph{basic threat configuration} in $P$.
\end{definition}
There are four cases for the structure of a basic threatened pair/basic threat configuration; see \cref{fig:basic}. Note that we make no isolatedness assumptions in the above definition, so technically it is possible that a basic split bad configuration is not actually a split bad configuration (or that a basic threatened pair is not actually a threatened pair, or that a basic threat configuration is not actually a threat configuration). However, we still borrow the same terminology (e.g., we may talk about the special entries of a basic split bad configuration).
\begin{figure}
\begin{center}
    \begin{tabular}{ c| c c c }
     & $c_1$ & $c_2$ & $c_3$  \\
    \hline
    $r_1^*$ & \coo$\mbf s_1$& \coo$\mbf s_2$ & \\
    $r_2$ & $s_2$ & $s_1$ & $s_3$\\ 
    $r_3$ &  & $s_3$ & $s_2$ 
    \end{tabular} 
    \qquad
 \begin{tabular}{ c| c c c }
     & $c_1$ & $c_2$ & $c_3$  \\
    \hline
    $r_1^*$ & \coo$\mbf s_2$& \coo$\mbf s_1$ & \\
    $r_2$ & $s_1$ & $s_2$ & $s_3$\\ 
    $r_3$ &  & $s_3$ & $s_2$ 
    \end{tabular}
\qquad
\begin{tabular}{ c| c c c }
     & $c_1$ & $c_2$ & $c_3$  \\
    \hline
    $r_1$ & $s_1$ & $s_2$ & \\
    $r^*_2$ & \coo$\mbf s_2$& \coo$\mbf s_1$& $s_3$\\ 
    $r_3$ &  & $s_3$ & $s_2$ 
    \end{tabular}
    \qquad
    \begin{tabular}{ c| c c c }
     & $c_1$ & $c_2$ & $c_3$  \\
    \hline
    $r_1$ & $s_2$ & $s_1$ & \\
    $r^*_2$ & \coo$\mbf s_1$& \coo$\mbf s_2$& $s_3$\\ 
    $r_3$ &  & $s_3$ & $s_2$ 
    \end{tabular}
\end{center}
    \caption{The four possibilities for the structure of a basic threat configuration, each illustrated with a basic threatened pair in {\coo \bf orange}. In each case, rows in $\sR^*$ are marked with a star. (Rows not marked with a star may or may not be in $\sR^*$.)
    }
    \label{fig:basic}
\end{figure}

\begin{fact}\label{fact:switch}
Consider any partial Latin square $R\in \mc P_n$ and set of rows $\sR^*$, and let $\mbf R'$ be obtained from $R$ by switching each isolated intercalate with probability 1/2, independently. Let $\{e_1,e_2\}$ be any threatened pair for $R$. Then, $\{e_1,e_2\}$ is a basic threatened pair for $\mbf R'$, with probability at least $(1/2)^5$.
\end{fact}
\begin{proof}
    There is some threat configuration $F_0$ in $R$ corresponding to the threatened pair $\{e_1,e_2\}$. By the definition of a split bad configuration, in $R\cup\{e_1,e_2\}$, there is a way to switch isolated intercalates to introduce a pair of intersecting intercalates containing $\{e_1,e_2\}$. If we do the same switches in $R$ (except possibly the switch of the special intercalate, which may not actually exist in $R$), we make $\{e_1,e_2\}$ a basic threatened pair. See \cref{fig:switch-example} for an example.
    
    Now, $F_0$ intersects at most five isolated intercalates in $R$ (it is clear that $F_0$ contains at most three isolated intercalates, but the other entries in $F_0$ could be contained in totally different isolated intercalates). The desired result follows.
\end{proof}

\begin{figure}
\newcommand{\HLinset}{2pt}    
\newcommand{\HLline}{0.8pt}     
\newcommand{\HLfill}{blue!10} 
    \centering
    \begin{center}
\begin{tabular}{c c c}
\begin{NiceTabular}{ c| c c c c c c }
 \CodeBefore[create-cell-nodes]
\begin{tikzpicture}
  \node[fill=\HLfill, rounded corners, inner sep=\HLinset, fit=(3-5) (4-6)] {};
  \node[fill=\HLfill, rounded corners, inner sep=\HLinset, fit=(5-5) (7-7)] {};
  \node[fill=\HLfill, rounded corners, inner sep=\HLinset, fit=(5-2) (6-4)] {};
\end{tikzpicture}
\Body
     & $c_1$ & $c_2$ & $c_3$ & $c_4$ & $c_5$ & $c_6$      \\
    \hline
    $r^{*}_1$ &  &  \coo $\mbf{s_2}$        &      \coo $\mbf{s_1}$             & & &            \\
    $r_2$ &            &       &            & $s_3$ & $s_4$ &            \\
    $r_3$ &            & $s_1$ & $s_2$      & $s_4$ & $s_3$ &            \\ 
    $r_4$ & $s_3$ &       & $s_5$ & $s_6$ &            & $s_2$ \\ 
    $r_5$ & $s_5$ &       & $s_3$ &            &            &            \\ 
    $r_6$ &            &       &            & $s_2$ &            & $s_6$ \\
\end{NiceTabular}  

&

$\rightarrow$

&

\begin{NiceTabular}{ c| c c c c c c }
 \CodeBefore[create-cell-nodes]
\begin{tikzpicture}
  \node[fill=green!20, rounded corners, inner sep=\HLinset, fit=(4-3) (5-5)] {};
\end{tikzpicture}
\Body
     & $c_1$ & $c_2$ & $c_3$ & $c_4$ & $c_5$ & $c_6$      \\
    \hline
    $r_1^*$ &            &      \coo $\mbf{s_2}$ &  \coo $\mbf{s_1}$          & & &            \\
    $r_2$ &       &            &            & $s_4$      & $s_3$ &       \\
    $r_3$ &       &   $s_1$ & $s_2$ & $s_3$ & $s_4$ &       \\ 
    $r_4$ & $s_5$ &            & $s_3$ & $s_2$ &       & $s_6$ \\ 
    $r_5$ & $s_3$ &            & $s_5$      &            &       &       \\ 
    $r_6$ &       &            &            & $s_6$      &       & $s_2$ \\
\end{NiceTabular}    
\end{tabular}
\end{center}
    \caption{On the left, an example of a threat configuration in a partial Latin square (with an example threatened pair in {\coo \bf orange}). If we switch the three highlighted intercalates, then we obtain the partial Latin square on the right, which contains a basic threat configuration (highlighted). The threatened pair on the left has now become a basic threatened pair on the right.}
    \label{fig:switch-example}
\end{figure}

We next record a lemma saying that there are typically not many entries covered by basic threatened pairs. Recall the definitions of $\ell$ and $\varepsilon$ from \cref{def:ell_eps}.

\begin{definition}
     For $C>0$ and a hypergraph $H \subseteq K^{(3)}_{n,n,n}$, let $\cT^{\mr{basic}}(H, C) \subseteq \cL_n$ be the set of Latin squares $L$ such that for some set of rows $\sR^*$ of size $|\sR^*| =\ell$, the number of basic threatened pairs in $H \cap L$ is more than $C \varepsilon^3 n^3 \ell$.
\end{definition}

\begin{lemma}
\label{lemma:upper_bound_reduced_basic_threats}
There is an absolute constant $C>0$ such that the following holds. 
Consider a random hypergraph $\bfH \sim \mb G^{(3)}(n,\varepsilon)$ and an independent random Latin square $\bfL \sim \on{Unif}(\mc L_n)$. We have
\[ \mathbb{P}[ \bfL \in \cT^{\mr{basic}}(\bfH, C)] \leq \exp(-\Omega(\varepsilon^3 n^2)).\]
\end{lemma}

We defer the proof of \cref{lemma:upper_bound_reduced_basic_threats} to \cref{subsec:basic}. First, we show how to use a switching argument (with \cref{fact:switch}) to deduce \cref{lemma:upper_bound_threatened_pairs}. Note that, as we will see, this deduction crucially requires that the error probability in \cref{lemma:upper_bound_reduced_basic_threats} is $\exp(-\Omega(\varepsilon^3n^2))$ rather than our usual $\exp(-\omega(n\log^2{n}))$. We also need an upper bound on the size of the largest family of pairwise disjoint intercalates, as follows.

\begin{definition}
    Let $\cT^{\mr{int}}_{\mr{upper}}(H)\subseteq \mc L_n$ be the set of Latin squares $L$ for which there is a family of more than $80 \varepsilon^4 n^2$ disjoint intercalates in $H\cap L$.
\end{definition}
\begin{lemma}
\label{lemma:not_too_many_isolated_intercalates}
Consider a random hypergraph $\bfH \sim \mb G^{(3)}(n,\varepsilon)$ and an independent random Latin square $\bfL \sim \on{Unif}(\mc L_n)$. We have
\[\Pr[\mbf L\in \cT^{\mr{int}}_{\mr{upper}}(\mbf H)]\le \exp(-\omega(n\log^2 n)).\]
\end{lemma}
We prove \cref{lemma:not_too_many_isolated_intercalates} with an averaging argument and \cref{lem:random_LS_to_triangle_removal} (this is a subset of the ideas we have already seen in \cref{sec:existence_disjoint_ints,subsec:bad_configs}). Indeed, the following lemma is a version of \cref{lemma:not_too_many_isolated_intercalates} for the triangle removal process.
\begin{definition}
    Let $\vec \cU^{\mr{int}}_{\mr{upper}}(H, \alpha)\subseteq \vec{\mc P_n}$ be the set of partial ordered Latin squares $\vec P$ for which there is a family of more than $20 \alpha^4\varepsilon^4 n^2$ disjoint intercalates in $H\cap \vec P$.
\end{definition}
\begin{lemma}
\label{lemma:not_too_many_isolated_intercalates_bin}
Let $\alpha \in (0,1)$ be a sufficiently small constant. Consider independent random hypergraphs $\vec\bfR \sim \trp(n, \alpha n^2)$ and $\bfH \sim \mb G^{(3)}(n,\varepsilon)$. Then we have
\[\Pr[\vec {\mbf R}\in \vec \cU^{\mr{int}}_{\mr{upper}}(\mbf H,\alpha)]\le \exp(-\omega(n\log^2 n)).\]
\end{lemma}
\begin{proof}[Proof of \cref{lemma:not_too_many_isolated_intercalates} given \cref{lemma:not_too_many_isolated_intercalates_bin}] 
Let $\alpha>0$ be small enough for \cref{lemma:not_too_many_isolated_intercalates_bin}, let $m = \alpha n^2$ and let $\rho = \alpha^4/4$.

By \cref{lemma:not_too_many_isolated_intercalates_bin} and \cref{fact:joint_prob_vs_most_outcomes}(\ref{fact:joint_to_separate}), with probability $1-\exp(-\omega(n\log^2{n}))$ our random hypergraph $\bfH \sim \mb G^{(3)}(n,\varepsilon)$ satisfies $\mathbb{P}[\vec{\bfR} \in \vec\cU^{\mr{int}}_{\mr{upper}}(\mbf H, \alpha)\,|\,\mbf H] \leq \exp(-\omega(n \log^2{n}))$. Let $H$ be such an outcome of $\bfH$.

To show $\vec\cU^{\mr{int}}_{\mr{upper}}(H,\alpha)$ is $(\rho,m)$-inherited from $\cT^{\mr{int}}_{\mr{upper}}(H)$, let $L \in \cT^{\mr{int}}_{\mr{upper}}(H)$, so there is a family of $X \geq 80 \varepsilon^4 n^2$ disjoint intercalates in $H\cap L$. Let $\bfY$ be the number of these intercalates which lie in a random ordered $m$-subset $\vec \bfP_m(L)$ of $L$. We have
\[\mathbb{E}[\bfY] =X\cdot (1+o(1))\alpha^4\geq \alpha^4 X /2.\]
By Markov's inequality,
\[\Pr[\vec \bfP_m(L)\in \vec\cU^{\mr{int}}_{\mr{upper}}(H, \alpha)]\ge \Pr\left[\bfY\ge \frac{\alpha^4}{4}X\right]=1-\Pr\left[X-\bfY> \left(1-\frac{\alpha^4}{4}\right)X\right]\ge 1-\frac{1-\alpha^4/2}{1-\alpha^4/4}\ge \rho.\]
That is to say, $\vec\cU^{\mr{int}}_{\mr{upper}}(H,\alpha)$ is $(\rho,m)$-inherited from $\cT^{\mr{int}}_{\mr{upper}}(H)$. Thus, by \cref{lem:random_LS_to_triangle_removal},
\[ \mathbb{P}[\bfL \in \cT^{\mr{int}}_{\mr{upper}}(H)] \leq \exp(2n\log^2{n})\mathbb{P}[\vec{\bfR} \in \vec\cU^{\mr{int}}_{\mr{upper}}(H,\alpha)]=\exp(-\omega(n\log^2n)).\]
Recalling our choice of $H$, \cref{fact:joint_prob_vs_most_outcomes}(\ref{fact:separate_to_joint}) concludes the proof.
\end{proof}
\begin{proof}[Proof of \cref{lemma:not_too_many_isolated_intercalates_bin}]
For each $H \subseteq K^{(3)}_{n,n,n}$, note that $\vec \cU^{\mr{int}}_{\mr{upper}}(H,\alpha)$ is a monotone increasing property. So, by \cref{lemma:triangle_removal_chunk_to_binomial} (applied to each possible outcome $H$ of $\mbf H$), we have
  \[\Pr[\bfH\cap \vec{\bfR}\in \vec \cU^{\mr{int}}_{\mr{upper}}(H,\alpha)]\le\Pr[\bfH\cap\bfG'\in \vec \cU^{\mr{int}}_{\mr{upper}}(H,\alpha)]+\exp(-\Omega(n^2))=\Pr[\bfG\in \vec \cU^{\mr{int}}_{\mr{upper}}(H,\alpha)]+\exp(-\Omega(n^2)),\]
  where $\bfG' \sim \mb G^{(3)}(n,2\alpha/n)$ and $\bfG \sim \mb G^{(3)}(n,2\varepsilon\alpha/n)$.
Let $\mbf X$ be the size of the maximum family of disjoint intercalates in $\mbf G$. We have

\[\mathbb{E}[\bfX] \leq n^6 (2\alpha \varepsilon/n)^4 = 16 \alpha^4 \varepsilon^4 n^2,\]
and changing an edge of $\mbf G$ changes $\bfX$ by at most $1$ (since an entry can be part of at most one intercalate in a maximum disjoint family). By \cref{freedman_inequality},
$\Pr[\bfG\in \vec{\cU}^{\mr{int}}_{\mr{upper}}(H,\alpha)]$ is at most
\[\Pr[\mbf X>20 \alpha^4 \varepsilon^4 n^2]\le \exp\Bigg(-\Omega\Big(\frac{\varepsilon^8 n^4}{n^3 \cdot \varepsilon/n + \varepsilon^4 n^2}\Big)\Bigg)= \exp(-\Omega(\varepsilon^7 n^2))=\exp(-\omega(n\log^2 n)),\]
and the desired result follows.
\end{proof}

We now deduce \cref{lemma:upper_bound_threatened_pairs} from \cref{lemma:not_too_many_isolated_intercalates,fact:switch,lemma:upper_bound_reduced_basic_threats}.

\begin{proof}[Proof of \cref{lemma:upper_bound_threatened_pairs}] 
Let $C$ be as in \cref{lemma:upper_bound_reduced_basic_threats}, and let $C'=100C$. 

By \cref{fact:joint_prob_vs_most_outcomes}(\ref{fact:joint_to_separate}) and \cref{lemma:not_too_many_isolated_intercalates,lemma:upper_bound_reduced_basic_threats}, with probability $1 - \exp(-\omega(n\log^2{n}))$ over the randomness of $\bfH$, we have
\[\mathbb{P}[\bfL \in \cT^{\mr{basic}}(\bfH, C)\,|\,\bfH]\leq \exp(- \Omega(\varepsilon^3 n^2)),\quad \mathbb{P}[\bfL \in \cT^{\mr{int}}_{\mr{upper}}(\bfH)\,|\,\bfH]\le \exp(-\omega(n\log^2 n)).
\]
Fix such an outcome $H \in K^{(3)}_{n,n,n}$ of $\bfH$, and let $\mc S=\cT^{\mr{threat}}(H, C', \alpha, \phi)\setminus \mc T^{\mr{int}}_{\mr{upper}}(H)$. By \cref{fact:joint_prob_vs_most_outcomes}(\ref{fact:separate_to_joint}), it suffices to show that 
\[\Pr[\mbf L\in \mc S]\le \exp(-\omega(n\log^2 n)).\]
Let $\bfP\in\mc P_{n,\alpha n^2}$ be a uniformly random subset of $\alpha n^2$ edges of $\bfL$ and let $\bfP'$ be obtained from $\bfP$ by switching each isolated intercalate in $H\cap \bfP$ with probability $1/2$ independently. Let $\bfL'=\bfP'\cup (\bfL\setminus \bfP)$.

Let $\bfX$ be the maximum over all choices of $\sR^*$ of the number of threatened pairs in $H\cap \bfP $, and let $\bfY\le \bfX$ be the maximum over all choices of $\sR^*$ of the number of threatened pairs in $H \cap \bfP$ which become basic threatened pairs in $H \cap\bfP'$. By \cref{fact:switch} with $\bfR=H \cap \bfP$ (and the fact that an expected maximum is at least the maximum expectation) we have $\mb E[\mbf Y\,|\,\mbf X]\ge(1/32)\mbf X$, so by Markov's inequality
\[\Pr\left[\mbf Y\ge  \frac{1}{100}\mbf X\,\middle|\,\mbf X\right] = 1-\Pr\left[\mbf X-\mbf Y> \frac{99}{100}\mbf X\,\middle|\,\mbf X\right]\ge 1-\frac{31/32}{99/100} \ge \frac{1}{100}.\]
Also, by the definition of $\cT^{\mr{threat}}(H, C', \alpha, \phi)\supseteq \mc S$, we have 
\[\Pr[\bfX\ge C'\varepsilon^3 n^3\ell \,|\,\mbf L\in \mc S]\ge \phi.\]
Recalling that $C'=100C$, and letting $\mc T=\cT^{\mr{basic}}(H, C)$, we deduce that $\Pr[\mbf L'\in \cT\,|\,\mbf L\in \mc S]\ge \phi/100$ or in other words
\[\Pr[\mbf L\in \mc S]\le \frac{100}{\phi}\mathbb{P}[\bfL \in \cS\text{ and } \bfL' \in \cT].\]
So, it suffices to prove that $\Pr[\mbf L\in \cS\text{ and }\mbf L'\in \mc T]\le \exp(-\omega(n\log^2n))$.

We introduce the notation $L_1 \rightarrow L_2$ to mean that we can obtain $L_2$ from $L_1$ by switching some disjoint intercalates in $H\cap L_1$. Recall that $\bfL'$ is obtained from $\bfL$ by switching some intercalates which are isolated in $H\cap \mbf P$, and these intercalates are certainly disjoint, so for any of the possible outcomes $(L_1,L_2)$ of $(\bfL,\bfL')$ we have $L_1 \rightarrow L_2$. We therefore have
\begin{align*}
\mathbb{P}[\bfL \in \cS\text{ and } \bfL' \in \cT] &\leq \sum_{L_2 \in \cT} \sum_{L_1 \in \cS, L_1 \rightarrow L_2} \mathbb{P}[\bfL = L_1] \mathbb{P}[\bfL' = L_2 \,|\, \bfL = L_1]\\
&\leq \sum_{L_2 \in \cT} \mathbb{P}[\bfL = L_2] \sum_{L_1 \in \cS, L_1 \rightarrow L_2} 1 \\
&\leq \sum_{L_2 \in \cT} \mathbb{P}[\bfL = L_2] \binom{n^4+80\varepsilon^4 n^2}{80\varepsilon^4 n^2}\\ 
&\leq \mathbb{P}[\bfL \in \mc T]\exp(O(80\varepsilon^4 n^2\log n))\\
&\leq \exp(- \Omega(\varepsilon^3 n^2))\cdot \exp(O(80\varepsilon^4 n^2\log n))\\
&\le \exp(-\omega(n\log^2 n)),
\end{align*}
as desired. Here, in the second inequality we used the fact that $\mathbb{P}[\bfL = L_1] = \mathbb{P}[\bfL = L_2]$ (since $\mbf L$ has the same probability of being equal to any Latin square). In the third inequality we used the fact that every Latin square has at most $n^4$ intercalates, and the difference between $L_1$ and $L_2$ is described by at most $80\varepsilon^4 n^2$ disjoint intercalate switches (note that $L_1\in \mc S$ implies that $L_1\notin \mc T^{\mr{int}}_{\mr{upper}}(H)$, i.e., $H\cap L_1$ has at most $80\varepsilon^4 n^2$ disjoint intercalates). Given a set of size $N$, the number of subsets of size at most $k$ is bounded by $\binom{N+k}k$. The fifth inequality is by our choice of $H$, and the final one uses that $\varepsilon=o(1/\log n)$ (recall \cref{def:ell_eps}). It is vital here that the error probability in \cref{lemma:upper_bound_reduced_basic_threats} is $\exp(-\Omega(\varepsilon^3n^2))$.
\end{proof}

\subsection{Upper-bounding the number of basic threatened pairs}\label{subsec:basic}
It remains to prove \cref{lemma:upper_bound_reduced_basic_threats}. This does not really involve any new ideas. Specifically, the approach is to carefully consider how basic threat configurations can emerge through the triangle removal process, using \cref{lem:random_LS_to_triangle_removal} to relate this to a random Latin square. There are four different cases for the structure of a basic threat configuration (cf.\ \cref{fig:basic}), each of which needs to be treated in a slightly different way.

We start by giving each of the four cases in \cref{fig:basic} a name.
\begin{definition}
Consider the four cases for the structure of a basic threat configuration, as follows\footnote{These are exactly as in \cref{fig:basic}, but in the last two cases we have swapped the first and second rows.}.
\begin{center}
    \begin{tabular}{ c| c c c }
     & $c_1$ & $c_2$ & $c_3$  \\
    \hline
    $r_1^*$&&&\\ 
    $r_2$ & $s_2$ & $s_1$ & $s_3$\\ 
    $r_3$ &  & $s_3$ & $s_2$ 
    \end{tabular} 
    \qquad
 \begin{tabular}{ c| c c c }
     & $c_1$ & $c_2$ & $c_3$  \\
    \hline
    $r_1^*$&&&\\ 
    $r_2$ & $s_1$ & $s_2$ & $s_3$\\ 
    $r_3$ &  & $s_3$ & $s_2$ 
    \end{tabular}
\qquad
\begin{tabular}{ c| c c c }
     & $c_1$ & $c_2$ & $c_3$  \\
    \hline
    $r^*_1$ & & & $s_3$\\ 
    $r_2$ & $s_1$ & $s_2$ & \\
    $r_3$ &  & $s_3$ & $s_2$ 
    \end{tabular}
    \qquad
    \begin{tabular}{ c| c c c }
     & $c_1$ & $c_2$ & $c_3$  \\
    \hline
    $r^*_1$ & & & $s_3$\\ 
    $r_2$ & $s_2$ & $s_1$ & \\
    $r_3$ &  & $s_3$ & $s_2$ 
    \end{tabular}
\end{center}
These four cases describe four partial Latin squares $Q_1,Q_2,Q_3,Q_4\in \mc P_3$ (from left to right). Each $Q_t$ has three rows, three columns and three symbols (though in $Q_1$ and $Q_2$, one of the rows does not have any entries in it).

For each type $t$, and any partial Latin square $P\in \mc P_{n}$, an \emph{embedding} from $Q_t$ into $P$ is an injective map from the 9 vertices of $Q_t$ into the vertices of $P$ (where row-vertices are mapped to row-vertices, column-vertices are mapped to column-vertices, and symbol-vertices are mapped to symbol-vertices), such that the image of every hyperedge of $Q_t$ is a hyperedge of $P_t$.

Then, for a partial Latin square $P\in \mc P_{n}$, note that a basic threatened pair (together with an associated basic threat configuration) is specified by an embedding of some $Q_t$ into $P$, where we demand that $r_1^*$ is mapped into $\sR^*$. In this case we say that the basic threatened pair has \emph{type} $t$.
\end{definition}
Note that a basic threatened pair can have multiple types (if it has multiple associated basic threat configurations).
\begin{definition}
For some $t\in \{1,2,3,4\}$, let $\pi=(e_1,\dots,e_5)$ be an ordering of the entries of $Q_t$, and let $\phi(Q_t)$ be an embedding of $Q_t$ into an ordered partial Latin square $\vec P\in \vec {\mc P}_n$. We say that $\phi(Q_t)$ is \emph{$\pi$-consistent} with $\vec P$ if $\phi(e_i) \in \vec{P}[(i-1)/5, i/5]$ for all $i\in \{1,\dots,5\}$. We say that a basic threatened pair of type $t$ is $\pi$-consistent with $\vec P$ if it is associated with some $\pi$-consistent embedding $\phi(Q_t)$.
\end{definition}

Now, we state a version of \cref{lemma:upper_bound_reduced_basic_threats} for the triangle removal process. Recall the definitions of $\ell$ and $\varepsilon$ from \cref{def:ell_eps}.
\begin{definition}
    For a hypergraph $H$, a type $t\in \{1,2,3,4\}$ and an ordering $\pi$ of the entries of $Q_t$, let $\vec\cU^{\mr{basic}}(H,t,\pi) \subseteq \vec \cP_{n}$ be the set of partial ordered Latin squares $\vec{P}$ such that for some set of rows $\sR^*$ of size $|\sR^*| = \ell$, the number of type-$t$ basic threatened pairs in $H \cap \vec{P}$ which are $\pi$-consistent with $\vec{P}$ is more than $\varepsilon^3 n^3 \ell$.
\end{definition}
\begin{lemma}
\label{lemma:upper_bound_reduced_basic_threats_bin}
Fix a sufficiently small constant $\alpha >0$, and any $t\in \{1,2,3,4\}$. There is an ordering $\pi$ of $Q_t$ such that for independent random hypergraphs $\vec{\bfR} \sim \trp(n, \alpha n^2)$ and $\bfH \sim \mb G^{(3)}(n, \varepsilon)$, we have
\[\mathbb{P}[ \vec{\bfR} \in \vec\cU^{\mr{basic}}(\bfH,t,\pi)] \leq \exp(-\Omega(\varepsilon^3 n^2)).\]
\end{lemma}
Before proving \cref{lemma:upper_bound_reduced_basic_threats_bin}, we show how to deduce \cref{lemma:upper_bound_reduced_basic_threats}.
\begin{proof}[Proof of \cref{lemma:upper_bound_reduced_basic_threats}]
Let $\alpha \in (0,1)$ be small enough for \cref{lemma:upper_bound_reduced_basic_threats_bin}, let $m = \alpha n^2$ and let $\rho =  \alpha^5/5^7$, and assume $C\ge 4(5^7/\alpha^5)$. For each $t$, let $\pi_t$ be the ordering in \cref{lemma:upper_bound_reduced_basic_threats_bin}.

By \cref{lemma:upper_bound_reduced_basic_threats_bin} and \cref{fact:joint_prob_vs_most_outcomes}(\ref{fact:joint_to_separate}), with probability $1 - \exp(-\Omega(\varepsilon^3 n^2))$ over the randomness of $\bfH$, we have that for each $t \in \{1,2,3,4\}$,
\begin{align}
\label{eq:basic_unlikely}
\mathbb{P}[\vec \bfR \in \vec\cU^{\mr{basic}}(\bfH, t, \pi_t) \,|\,\bfH] \leq \exp(-\Omega(\varepsilon^3n^2)).
\end{align}
Fix such an outcome $H \subseteq K^{(3)}_{n,n,n}$ of $\bfH$.

Let $\vec\cU^{\mr{basic}}(H)=\bigcup_{t=1}^4\vec\cU^{\mr{basic}}(H,t,\pi_t)$. Given an ordered partial Latin square $\vec P$ and a set of rows $\sR^*$, we say a type-$t$ basic threatened pair is \emph{consistent} if it is $\pi_t$-consistent.

To show $\vec\cU^{\mr{basic}}(H)$ is $(\rho,m)$-inherited from $\cT^{\mr{basic}}(H,C)$, let $L \in \cT^{\mr{basic}}(H,C)$, so there is $\sR^*$ of size $|\sR^*| = \ell$ for which there are $X >C\varepsilon^3 n^3 \ell$ many basic threatened pairs in $H\cap L$. Let $\bfY$ be the number of basic threatened pairs in $H\cap L$ which are present and consistent in an ordered random $m$-subset $\vec \bfP_m(L)$ of $L$. We have
\[ \mathbb{E}[\bfY] = X \cdot (1+o(1))\alpha^5(1/5)^5\ge \frac{ \alpha^5}{5^6} X.\]
Now, if $\bfY\ge (\alpha^5/5^7)X>4\varepsilon^3 n^3\ell$ then $\vec\cU^{\mr{basic}}(H,t,\pi_t)$ holds for some $t$. Using Markov's inequality, it follows that
\[\mathbb{P}[\vec \bfP_m(L)\in \vec\cU^{\mr{basic}}(H)]\ge \mathbb{P}\left[\bfY \geq \frac{ \alpha^5}{5^7} X\right]=1-\mathbb{P}\left[X-\bfY > \left(1-\frac{ \alpha^5}{5^7}\right) X\right]\ge 1-\frac{1-(\alpha^5/5^6)}{1-(\alpha^5/5^7)}\ge \rho,\]
which means $\vec\cU^{\mr{basic}}(H)$ is $(\rho,m)$-inherited from $\cT^{\mr{basic}}(H,C)$.

Thus, we have
\[ \mathbb{P}[\bfL \in \cT^{\mr{basic}}(H,C)] \leq \exp(2n\log^2{n})\mathbb{P}[\vec{\bfR} \in \vec\cU^{\mr{basic}}(H)] \leq \exp(-\Omega(\varepsilon^3n^2)),\]
where the first inequality follows by \cref{lem:random_LS_to_triangle_removal}, and the second inequality follows by~\cref{eq:basic_unlikely}.

Recalling our choice of $H$, \cref{fact:joint_prob_vs_most_outcomes}(\ref{fact:separate_to_joint}) implies that 
\[ \mathbb{P}[\bfL \in \cT^{\mr{basic}}(\bfH,C)] \leq \exp(-\Omega(\varepsilon^3n^2)),\]
as desired.
\end{proof}

Now we prove \cref{lemma:upper_bound_reduced_basic_threats_bin}, separately considering types 1 and 2, and types 3 and 4.

\begin{proof}[Proof of \cref{lemma:upper_bound_reduced_basic_threats_bin} for types $t=1$ and $t=2$]
Recall that $Q_1$ and $Q_2$ have the following forms.
\begin{center}
\begin{tabular}{ c| c c c }
 & $c_1$ & $c_2$ & $c_3$ \\
\hline
$r_1^*$ & & & \\
$r_2$ &  $s_2$ & $s_1$ & $s_3$\\ 
$r_3$ &        & $s_3$ & $s_2$
\end{tabular}
\qquad \qquad 
\begin{tabular}{ c| c c c }
 & $c_1$ & $c_2$ & $c_3$ \\
\hline
$r_1^*$ &  & &\\ 
$r_2$ &  $s_1$ & $s_2$ & $s_3$\\ 
$r_3$ &        & $s_3$ & $s_2$
\end{tabular}
\end{center}
If $t=1$, we take our ordering $\pi$ to be
\[e_1=(r_3, c_3, s_2),\;e_2=(r_2,c_3,s_3),\;e_3=(r_3,c_2,s_3),\;e_4=(r_2, c_2, s_1),\;e_5=(r_2, c_1, s_2).\]
If $t=2$, we take our ordering $\pi$ to be
\[e_1=(r_3, c_3, s_2),\;e_2=(r_2,c_3,s_3),\;e_3=(r_3,c_2,s_3),\;e_4=(r_2, c_2, s_2),\;e_5=(r_2, c_1, s_1).\]

Fix $t\in \{1,2\}$ and let $\sR^*$ be a set of rows with $|\sR^*| = \ell$. Among all the $\pi$-consistent embeddings $\phi(Q_t)$ of $Q_t$ into $\mbf H\cap \vec {\mbf R}$ which satisfy $\phi(r_1^*)\in \sR^*$, let $\mbf N$ be the number of possibilities for $(\phi(r_1^*),\phi(c_1),\phi(c_2),\phi(s_1),\phi(s_2))$. Then, $\mbf N$ is an upper bound on the number of type-$t$ basic threatened pairs in $\bfH \cap \vec{\bfR}$ which are $\pi$-consistent with $\vec{\bfR}$.

We first reveal $\vec{\bfR}[0,1/5]$, which has $\alpha n^2/5$ entries. Each of them is present in $\bfH$ with probability $\varepsilon$ independently. Thus, by a Chernoff bound, with probability $1 - \exp (-\Omega(\varepsilon n^2))$, there are at most $ \varepsilon n^2$ many viable choices for $\phi(e_1)$ (and therefore $(\phi(r_3), \phi(c_3), \phi(s_2))$). Assume this is the case.

Next, we reveal $\mbf H\cap \vec{\bfR}[1/5,3/5]$. For each choice of $(\phi(r_3),\phi(c_3), \phi(s_2))$, there are at most $n$ choices for $\phi(e_2)$ in column $\phi(c_3)$ of $\vec{\bfR}[1/5,2/5]$. This determines $\phi(r_2)$ and $\phi(s_3)$, and then there is at most one occurrence of $\phi(s_3)$ in row $\phi(r_3)$ of $\vec{\bfR}[2/5,3/5]$, determining $\phi(c_2)$. This makes a total of at most $ \varepsilon n^3$ choices for $(\phi(r_2),\phi(r_3),\phi(c_2),\phi(c_3),\phi(s_2),\phi(s_3))$.

Let $S$ be the set of all possible choices for $(\phi(r_2),\phi(c_1),\phi(c_2),\phi(s_1),\phi(s_2))$ at this stage, so we have $|S|\le \varepsilon n^5$ (there are at most $n^2$ choices for $(\phi(c_1),\phi(s_1))$). Let $\mbf S_4\subseteq S$ be the set of all $(\phi(r_2),\phi(c_1),\phi(c_2),\phi(s_1),\phi(s_2))\in S$ for which $\phi(e_4)$ is an edge of $\mbf H \cap \vec{\bfR}[3/5,4/5]$, and let $\mbf S_5\subseteq \mbf S_4$ be the set of all such tuples for which $\phi(e_5)$ is an edge of $\mbf H \cap \vec{\bfR}[4/5,1]$. Note that $\mbf N\le \ell |\mbf S_5|$ (there are at most $\ell$ choices for $\phi(r_1^*)\in \sR^*$).

We next consider $\mbf H \cap \vec{\bfR}[3/5,4/5]$. We claim that $|\mbf S_4|\le \varepsilon^2 n^4$ with probability $1-\exp(-\Omega(\varepsilon^2 n^2))$. To see this, first note that, by \cref{lemma:triangle_removal_chunk_to_binomial}, it suffices to prove an analogous bound with $\mbf H\cap \mbf G\sim \mb G^{(3)}(n,(2\varepsilon\alpha/5)/n)$ in place of $\mbf H \cap \vec{\bfR}[3/5,4/5]$, where $\bfG \sim \mb G^{(3)}(n,2(\alpha/5)/n)$. But this is a direct consequence of \cref{thm:chernoff_dependency_graph}, since each potential entry $\phi(e_4)$ is present in $\mbf H\cap \bfG$ with probability at most $2\varepsilon \alpha /(5n)$, and appears in at most $n^2$ tuples in $S$.

Reveal an outcome of $\mbf H \cap \vec{\bfR}[3/5,4/5]$ satisfying $|\mbf S_4|\le \varepsilon^2 n^4$. We next consider $\mbf H \cap \vec{\bfR}[4/5,1]$: we claim that $|\mbf S_5|\le \varepsilon^3 n^3$ with probability $1-\exp(-\Omega(\varepsilon^3 n^2))$. The proof is basically as before: first note that it suffices to prove an analogous bound with $\mbf H\cap \mbf G\sim \mb G^{(3)}(n,(2\varepsilon\alpha/5)/n)$ in place of $\mbf H \cap \vec{\bfR}[4/5,1]$, and then the desired result follows from \cref{thm:chernoff_dependency_graph} (noting that each potential entry $\phi(e_5)$ appears in at most $n$ tuples in $\bfS_4$, since such a tuple is determined by a choice of $\phi(e_4)$ in the same row).

We have proved that the number of type-$t$ basic threatened pairs in $\bfH \cap \vec{\bfR}$ which are $\pi$-consistent with $\vec{\bfR}$ is at most $\mbf N\le \ell |\mbf S_5|\le \varepsilon^3 n^3 \ell$,  with probability at least $1-\exp(-\Omega(\varepsilon^3 n^2))$. A union bound over all (at most $2^n$) choices of $\sR^*$ finishes the proof.
\end{proof}

\begin{proof}[Proof of \cref{lemma:upper_bound_reduced_basic_threats_bin} for types $t=3$ and $t=4$]
Recall that $Q_3$ and $Q_4$ have the following forms.
\begin{center}
\begin{tabular}{ c| c c c }
 & $c_1$ & $c_2$ & $c_3$ \\
\hline
$r_1^*$ &        &       & $s_3$\\
$r_2$ &  $s_1$ & $s_2$ &      \\ 
$r_3$ &        & $s_3$ & $s_2$
\end{tabular}
\qquad\qquad
\begin{tabular}{ c| c c c }
 & $c_1$ & $c_2$ & $c_3$ \\
\hline
$r_1^*$ &        &       & $s_3$\\
$r_2$ &  $s_2$ & $s_1$ &      \\ 
$r_3$ &        & $s_3$ & $s_2$
\end{tabular}
\end{center}
If $t=3$, we take our ordering $\pi$ to be
\[e_1=(r_3, c_3, s_2),\;e_2=(r_1^*,c_3,s_3),\;e_3=(r_3,c_2,s_3),\;e_4=(r_2, c_2, s_2),\;e_5=(r_2, c_1, s_1).\]
If $t=4$, we take our ordering $\pi$ to be
\[e_1=(r_3, c_3, s_2),\;e_2=(r_1^*,c_3,s_3),\;e_3=(r_3,c_2,s_3),\;e_4=(r_2, c_2, s_1),\;e_5=(r_2, c_1, s_2).\]

Fix $t\in \{3,4\}$ and let $\sR^*$ be a set of rows with $|\sR^*| = \ell$. Among all the $\pi$-consistent embeddings $\phi(Q_t)$ of $Q_t$ into $\mbf H\cap \vec {\mbf R}$ which satisfy $\phi(r_{1}^*)\in \sR^*$, let $\mbf N$ be the number of possibilities for $(\phi(r_1^*),\phi(c_1),\phi(c_2),\phi(s_1),\phi(s_2))$. Then, $\mbf N$ is an upper bound on the number of type-$t$ basic threatened pairs in $\bfH \cap \vec{\bfR}$ which are $\pi$-consistent with $\vec{\bfR}$.

The first part is exactly the same as the previous proof: we first reveal $\vec{\bfR}[0,1/5]$, which has $\alpha n^2/5$ entries. By a Chernoff bound over the randomness of $\mbf H$, with probability $1 - \exp (-\Omega(\varepsilon n^2))$, there are at most $ \varepsilon n^2$ many viable choices for $\phi(e_1)$ (and therefore $(\phi(r_3), \phi(c_3), \phi(s_2))$). Assume this is the case.

Next, we reveal $\mbf H\cap \vec{\bfR}[1/5,3/5]$. The numbers are slightly different than in the last proof: for each choice of $(\phi(r_3),\phi(c_3), \phi(s_2))$, there are at most $\ell$ choices for $\phi(e_2)$ in column $\phi(c_3)$ of $\vec{\bfR}[1/5,2/5]$ (in a row in $\sR^*$). This determines $\phi(r^*_1)$ and $\phi(s_3)$, and then there is at most one occurrence of $\phi(s_3)$ in row $\phi(r_3)$ of $\vec{\bfR}[2/5,3/5]$. This makes a total of at most $ \varepsilon n^2\ell$ choices for $(\phi(r_1^*),\phi(r_3),\phi(c_2),\phi(c_3),\phi(s_2),\phi(s_3))$.

Let $S$ be the set of all possible choices for $(\phi(r^*_1),\phi(r_2),\phi(c_1),\phi(c_2),\phi(s_1),\phi(s_2))$ at this stage, so we have $|S|\le \varepsilon n^5\ell$ (there are at most $n^3$ choices for $(\phi(r_2),\phi(c_1),\phi(s_1))$). Let $\mbf S_4\subseteq S$ be the set of all $(\phi(r^*_1),\phi(r_2),\phi(c_1),\phi(c_2),\phi(s_1),\phi(s_2))\in S$ for which $\phi(e_4)$ is an edge of $\mbf H \cap \vec{\bfR}[3/5,4/5]$, and let $\mbf S_5\subseteq \mbf S_4$ be the set of all such tuples for which $\phi(e_5)$ is an edge of $\mbf H \cap \vec{\bfR}[4/5,1]$. Note that $\mbf N\le |\mbf S_5|$.

The end of the proof is similar to before.
We next consider $\mbf H \cap \vec{\bfR}[3/5,4/5]$: we claim that $|\mbf S_4|\le \varepsilon^2 n^4\ell$ with probability $1-\exp(-\Omega(\varepsilon^2 n^2))$. To see this, we first note that it suffices to prove an analogous bound with $\mbf H\cap \mbf G\sim \mb G^{(3)}(n,(2\varepsilon\alpha/5)/n)$ in place of $\mbf H \cap \vec{\bfR}[3/5,4/5]$, and then the desired result follows from \cref{thm:chernoff_dependency_graph}, noting that each potential entry $\phi(e_4)$ appears in at most $n^2\ell$ tuples in $S$.

Finally we consider $\mbf H \cap \vec{\bfR}[4/5,1]$: we claim that $|\mbf S_5|\le \varepsilon^3 n^3\ell$ with probability $1-\exp(-\Omega(\varepsilon^3 n^2))$. Again, we first note that it suffices to prove an analogous bound with $\mbf H\cap \mbf G\sim \mb G^{(3)}(n,(2\varepsilon\alpha/5)/n)$ in place of $\mbf H \cap \vec{\bfR}[4/5,1]$, and then the desired result follows from \cref{thm:chernoff_dependency_graph} (noting that each potential entry $\phi(e_5)$ appears in at most $n\ell$ tuples in $\bfS_4$, since such a tuple is determined by a choice of $\phi(e_4)$ in the same row, and a choice of $\phi(r_1^*)$).

We then finish with a union bound over choices of $\sR^*$, as before.
\end{proof}

\section{Concluding remarks}
\label{sec:concluding}
In this paper we have proved (a generalisation of) Cameron's conjecture, describing the joint distribution of the number of odd row/column/symbol permutations in a random $n\times n$ Latin square. There are various directions for further research, as follows.
\begin{itemize}
    \item Can quantitative aspects be improved? For example, in \cref{thm:shiny}(4), we see that the total variation error of our approximation is $o(1)$. Actually, our proof gives an error of $n^{-1/2+o(1)}$, and it seems this can be improved to $n^{-1+o(1)}$ with some additional calculations using the fact that $\mb E [N_{\mathrm{row}}(\mathbf{L})]=\mb E[N_{\mathrm{col}}(\mathbf{L})]=\mb E[N_{\mathrm{sym}}(\mathbf{L})]=n/2$. However, it seems plausible that the true error of the approximation might be super-exponentially small.
    \item What about constrained distributions on Latin squares (e.g.\ symmetric Latin squares)? We expect that similar results should hold, but the available enumeration estimates are much weaker. It might be possible to prove analogues of \cref{thm:shiny}(1) and \cref{thm:shiny}(5) (using the ideas described in \cref{subsec:intercalate-switchings}, which have relatively weak quantitative requirements), but new ideas would be required for the other parts of \cref{thm:shiny}.
    \item What about analogous questions for higher-dimensional analogues of Latin squares (sometimes called \emph{high dimensional permutations}, see e.g.\ \cite{LL16})? Here it seems completely new ideas would be required, because (we predict that) small switchings become rarer as the dimension increases.
\end{itemize}
\bibliographystyle{amsplain_initials_nobysame_nomr}
\bibliography{bibliography}

\providecommand{\bysame}{\leavevmode\hbox to3em{\hrulefill}\thinspace}
\providecommand{\MR}{\relax\ifhmode\unskip\space\fi MR }
\providecommand{\MRhref}[2]{%
  \href{http://www.ams.org/mathscinet-getitem?mr=#1}{#2}
}
\providecommand{\href}[2]{#2}
\begin{thebibliography}{10}

\bibitem{ADRA23}
Y.~Alimohammadi, P.~Diaconis, M.~Roghani, and A.~Saberi, \emph{Sequential importance sampling for estimating expectations over the space of perfect matchings}, Ann. Appl. Probab. \textbf{33} (2023), no.~2, 799--833.

\bibitem{AW25}
J.~Allsop and I.~M. Wanless, \emph{Subsquares in random {L}atin rectangles}, arXiv:2409.08446.

\bibitem{Alp17}
L.~Alpoge, \emph{Square-root cancellation for the signs of {L}atin squares}, Combinatorica \textbf{37} (2017), no.~2, 137--142.

\bibitem{Boll88}
B.~Bollob\'{a}s, \emph{The chromatic number of random graphs}, Combinatorica \textbf{8} (1988), no.~1, 49--55.

\bibitem{BM25}
C.~Bowtell and R.~Montgomery, \emph{Almost every {L}atin square has a decomposition into transversals}, arXiv:2501.05438.

\bibitem{Bre73}
L.~M. Br\`egman, \emph{Certain properties of nonnegative matrices and their permanents}, Dokl. Akad. Nauk SSSR \textbf{211} (1973), 27--30.

\bibitem{Cam02}
P.~J. Cameron, \emph{Permutations}, Paul {E}rd\H{o}s and his mathematics, {II} ({B}udapest, 1999), Bolyai Soc. Math. Stud., vol.~11, J\'{a}nos Bolyai Math. Soc., Budapest, 2002, pp.~205--239.

\bibitem{Cam15b}
P.~J. Cameron, \emph{Asymmetric {L}atin squares, {S}teiner triple systems, and edge-parallelisms}, arXiv:1507.02190.

\bibitem{Cam92}
P.~J. Cameron, \emph{Almost all quasigroups have rank {$2$}}, Discrete Math. \textbf{106/107} (1992), 111--115.

\bibitem{Cam01}
P.~J. Cameron, \emph{{BCC} problem list}, \url{https://webspace.maths.qmul.ac.uk/p.j.cameron/bcc/allprobs.pdf}, 2001.

\bibitem{Cam03}
P.~J. Cameron, \emph{Random {L}atin squares}, slides for a presentation at the RAND-APX meeting at Oxford in 2003, \url{https://webspace.maths.qmul.ac.uk/p.j.cameron/slides/rand-apx.pdf}, 2003.

\bibitem{Cam05}
P.~J. Cameron, \emph{Problems from ``{P}ermutations''}, \url{https://webspace.maths.qmul.ac.uk/p.j.cameron/permgps/permutations.html}, 2005.

\bibitem{Cam07}
P.~J. Cameron, \emph{Peter {C}ameron's {BCC} problems}, \url{https://webspace.maths.qmul.ac.uk/p.j.cameron/pjcatbcc.html}, 2007.

\bibitem{Cam13}
P.~J. Cameron, \emph{Problems}, \url{https://webspace.maths.qmul.ac.uk/p.j.cameron/oldprob.html}, 2013.

\bibitem{Cam15}
P.~J. Cameron, \emph{A niggling problem}, posted on Peter Cameron's blog, \url{https://cameroncounts.wordpress.com/2015/01/24/a-niggling-problem/}, 2015.

\bibitem{CR17}
N.~Cavenagh and R.~Ramadurai, \emph{On the distances between {L}atin squares and the smallest defining set size}, J. Combin. Des. \textbf{25} (2017), no.~4, 147--158.

\bibitem{CGW08}
N.~J. Cavenagh, C.~Greenhill, and I.~M. Wanless, \emph{The cycle structure of two rows in a random {L}atin square}, Random Structures Algorithms \textbf{33} (2008), no.~3, 286--309.

\bibitem{CW16}
N.~J. Cavenagh and I.~M. Wanless, \emph{There are asymptotically the same number of {L}atin squares of each parity}, Bull. Aust. Math. Soc. \textbf{94} (2016), no.~2, 187--194.

\bibitem{DP}
M.~Delcourt and L.~Postle, \emph{Refined absorption: A new proof of the existence conjecture}, arXiv:2402.17855.

\bibitem{DZ10}
A.~Dembo and O.~Zeitouni, \emph{Large deviations techniques and applications}, Stochastic Modelling and Applied Probability, vol.~38, Springer-Verlag, Berlin, 2010, Corrected reprint of the second (1998) edition.

\bibitem{Ego81}
G.~P. Egorychev, \emph{The solution of van der {W}aerden's problem for permanents}, Adv. in Math. \textbf{42} (1981), no.~3, 299--305.

\bibitem{Fal81}
D.~I. Falikman, \emph{Proof of the van der {W}aerden conjecture on the permanent of a doubly stochastic matrix}, Mat. Zametki \textbf{29} (1981), no.~6, 931--938, 957.

\bibitem{Fel68}
W.~Feller, \emph{An introduction to probability theory and its applications. {V}ol. {I}}, third ed., John Wiley \& Sons, Inc., New York-London-Sydney, 1968.

\bibitem{ferber2020almost}
A.~Ferber and M.~Kwan, \emph{Almost all {S}teiner triple systems are almost resolvable}, Forum of Mathematics, Sigma, vol.~8, Cambridge University Press, 2020, p.~e39.

\bibitem{Fre75}
D.~A. Freedman, \emph{On tail probabilities for martingales}, Ann. Probability \textbf{3} (1975), 100--118.

\bibitem{FM19}
B.~Friedman and S.~McGuinness, \emph{The {A}lon-{T}arsi conjecture: a perspective on the main results}, Discrete Math. \textbf{342} (2019), no.~8, 2234--2253.

\bibitem{GKLO23}
S.~Glock, D.~K\"{u}hn, A.~Lo, and D.~Osthus, \emph{The existence of designs via iterative absorption: hypergraph {$F$}-designs for arbitrary {$F$}}, Mem. Amer. Math. Soc. \textbf{284} (2023), no.~1406, v+131.

\bibitem{GK23}
S.~Gould and T.~Kelly, \emph{Hamilton transversals in random {L}atin squares}, Random Structures Algorithms \textbf{62} (2023), no.~2, 450--478.

\bibitem{HJ96}
R.~H\"{a}ggkvist and J.~C.~M. Janssen, \emph{All-even {L}atin squares}, Proceedings of the 6th {C}onference on {F}ormal {P}ower {S}eries and {A}lgebraic {C}ombinatorics ({N}ew {B}runswick, {NJ}, 1994), vol. 157, 1996, pp.~199--206.

\bibitem{HR94}
R.~Huang and G.-C. Rota, \emph{On the relations of various conjectures on {L}atin squares and straightening coefficients}, Discrete Math. \textbf{128} (1994), no.~1-3, 225--236.

\bibitem{JM96}
M.~T. Jacobson and P.~Matthews, \emph{Generating uniformly distributed random {L}atin squares}, J. Combin. Des. \textbf{4} (1996), no.~6, 405--437.

\bibitem{janson2011random}
S.~Janson, T.~{\L}uczak, and A.~Rucinski, \emph{Random graphs}, Wiley-Interscience Series in Discrete Mathematics and Optimization, Wiley-Interscience, New York, 2000.

\bibitem{janson2002infamous}
S.~Janson and A.~Ruci\'{n}ski, \emph{The infamous upper tail}, Random Structures Algorithms \textbf{20} (2002), no.~3, 317--342, Probabilistic methods in combinatorial optimization.

\bibitem{Jan95}
J.~C.~M. Janssen, \emph{On even and odd {L}atin squares}, J. Combin. Theory Ser. A \textbf{69} (1995), no.~1, 173--181.

\bibitem{KD15}
A.~D. Keedwell and J.~D\'{e}nes, \emph{Latin squares and their applications}, second ed., Elsevier/North-Holland, Amsterdam, 2015, With a foreword to the previous edition by Paul Erd\"{o}s.

\bibitem{Kee14}
P.~Keevash, \emph{The existence of designs}, arXiv:1401.3665.

\bibitem{Kee18c}
P.~Keevash, \emph{The existence of designs {II}}, arXiv:1802.05900.

\bibitem{Kee24}
P.~Keevash, \emph{A short proof of the existence of designs}, arXiv:2411.18291.

\bibitem{Kee18}
P.~Keevash, \emph{Counting designs}, J. Eur. Math. Soc. (JEMS) \textbf{20} (2018), no.~4, 903--927.

\bibitem{kwan2020almost}
M.~Kwan, \emph{Almost all {S}teiner triple systems have perfect matchings}, Proc. Lond. Math. Soc. (3) \textbf{121} (2020), no.~6, 1468--1495.

\bibitem{KSS21}
M.~Kwan, A.~Sah, and M.~Sawhney, \emph{Note on random latin squares and the triangle removal process}, arXiv:2109.15201.

\bibitem{KSSS22}
M.~Kwan, A.~Sah, and M.~Sawhney, \emph{Large deviations in random {L}atin squares}, Bull. Lond. Math. Soc. \textbf{54} (2022), no.~4, 1420--1438.

\bibitem{KSSS23}
M.~Kwan, A.~Sah, M.~Sawhney, and M.~Simkin, \emph{Substructures in {L}atin squares}, Israel J. Math. \textbf{256} (2023), no.~2, 363--416.

\bibitem{KS18}
M.~Kwan and B.~Sudakov, \emph{Intercalates and discrepancy in random {L}atin squares}, Random Structures Algorithms \textbf{52} (2018), no.~2, 181--196.

\bibitem{LL13}
N.~Linial and Z.~Luria, \emph{An upper bound on the number of {S}teiner triple systems}, Random Structures Algorithms \textbf{43} (2013), no.~4, 399--406.

\bibitem{LL16}
N.~Linial and Z.~Luria, \emph{Discrepancy of high-dimensional permutations}, Discrete Anal. (2016), Paper No. 11, 8.

\bibitem{LS18}
N.~Linial and M.~Simkin, \emph{Monotone subsequences in high-dimensional permutations}, Combin. Probab. Comput. \textbf{27} (2018), no.~1, 69--83.

\bibitem{LS87}
J.~Lynch and J.~Sethuraman, \emph{Large deviations for processes with independent increments}, Ann. Probab. \textbf{15} (1987), no.~2, 610--627.

\bibitem{McD89}
C.~McDiarmid, \emph{On the method of bounded differences}, Surveys in combinatorics, 1989 ({N}orwich, 1989), London Math. Soc. Lecture Note Ser., vol. 141, Cambridge Univ. Press, Cambridge, 1989, pp.~148--188.

\bibitem{MW99}
B.~D. McKay and I.~M. Wanless, \emph{Most {L}atin squares have many subsquares}, J. Combin. Theory Ser. A \textbf{86} (1999), no.~2, 322--347.

\bibitem{Pit97}
A.~O. Pittenger, \emph{Mappings of {L}atin squares}, Linear Algebra Appl. \textbf{261} (1997), 251--268.

\bibitem{Rad97}
J.~Radhakrishnan, \emph{An entropy proof of {B}regman's theorem}, J. Combin. Theory Ser. A \textbf{77} (1997), no.~1, 161--164.

\bibitem{RT96}
V.~R\"{o}dl and L.~Thoma, \emph{Asymptotic packing and the random greedy algorithm}, Random Structures Algorithms \textbf{8} (1996), no.~3, 161--177.

\bibitem{Ros11}
N.~Ross, \emph{Fundamentals of {S}tein's method}, Probab. Surv. \textbf{8} (2011), 210--293.

\bibitem{Spe95}
J.~Spencer, \emph{Asymptotic packing via a branching process}, Random Structures Algorithms \textbf{7} (1995), no.~2, 167--172.

\bibitem{vLW01}
J.~H. van Lint and R.~M. Wilson, \emph{A course in combinatorics}, 2nd ed., Cambridge University Press, Cambridge, 2001.

\bibitem{warnke2016method}
L.~Warnke, \emph{On the method of typical bounded differences}, Combinatorics, Probability and Computing \textbf{25} (2016), no.~2, 269--299.

\bibitem{Zap96}
P.~Zappa, \emph{Triplets of {L}atin squares}, Boll. Un. Mat. Ital. A (7) \textbf{10} (1996), no.~1, 63--69.

\end{thebibliography}

\end{document}